\documentclass[10pt]{article}

\usepackage[a4paper,margin=0.85in]{geometry}
\usepackage[T1]{fontenc}
\usepackage{lmodern}
\usepackage{amsmath,amssymb,amsthm}
\usepackage{xcolor}
\usepackage{graphicx}
\usepackage{authblk}
\usepackage{microtype}
\definecolor{linkblue}{RGB}{0,72,138}
\usepackage[colorlinks=true,allcolors=linkblue]{hyperref}

\allowdisplaybreaks[2]
\newtheorem{theorem}{Theorem}
\newtheorem{lemma}[theorem]{Lemma}
\newtheorem{corollary}[theorem]{Corollary}
\newtheorem{proposition}[theorem]{Proposition}
\theoremstyle{definition}
\newtheorem{definition}[theorem]{Definition}
\theoremstyle{remark}

\newcommand{\vol}{\operatorname{vol}}
\newcommand{\lazy}{\mathrm{lazy}}
\newcommand{\Prob}{\mathbb{P}}
\newcommand{\E}{\mathbb{E}}
\newcommand{\diag}{\mathrm{diag}}
\newcommand{\Var}{\operatorname{Var}}
\newcommand{\Bin}{\operatorname{Bin}}
\newcommand{\Ber}{\operatorname{Ber}}
\newcommand{\dist}{\mathrm{dist}}
\newcommand{\diam}{\operatorname{diam}}
\newcommand{\Poisson}{\operatorname{Poisson}}
\newcommand{\Z}{\mathbb Z}
\newcommand{\R}{\mathbb R}
\newcommand{\op}{\mathrm{op}}
\newcommand{\eff}{\mathrm{eff}}
\newcommand{\1}{\mathbf 1}
\newcommand{\norm}[1]{\left\lVert #1\right\rVert}

\title{Band Random Graphs: Connectivity, Geometry, and Transport}
\author[1]{Shantanav Chakraborty \thanks{\texttt{\href{mailto:shchakra@iiit.ac.in}{shchakra@iiit.ac.in}}}}
\author[2]{Kyle Luh \thanks{\texttt{\href{mailto:kyle.luh@colorado.edu}{kyle.luh@colorado.edu}}}}
\author[1]{Rudransh Pratap Singh \thanks{\texttt{\href{mailto:rudransh.s@research.iiit.ac.in}{rudransh.s@research.iiit.ac.in}}}}

\affil[1]{%
CQST and CSTAR, IIIT Hyderabad, India
}

\affil[2]{%
Department of Mathematics, University of Colorado Boulder, USA
}

\date{October 30, 2026}

\begin{document}
\maketitle

\begin{abstract}
Let $G(n,p,m)$ be the random subgraph of the cycle power $C_n^m$ obtained by
retaining each edge independently with probability $p$. We study connectivity,
component geometry, and random walks throughout the regime $\log n\ll m<n/2$.

We prove that $p\sim(\log n)/(2m)$ is a sharp threshold for connectivity.
If $n(1-p)^{2m}\to\lambda\in(0,\infty)$, then the number of isolated vertices
converges to $\Poisson(\lambda)$, and with high probability every other vertex
lies in a single component $C_1$. In particular, the connectivity probability
converges to $e^{-\lambda}$. The component $C_1$ contains a cyclic chain of
connected sets occupying all but an $o(1)$ fraction of each consecutive block
of length $\asymp m$, with $\Theta(m\log n)$ edges between consecutive sets.
Its diameter has order $n/m+\log n/\log\log n$ with high probability.

For every fixed $c>1$, uniformly over $p\geq c\log n/(2m)$, the lazy-walk
spectral gap is asymptotic in probability to that of $C_n^m$. In the same range, the mean
hitting time of simple random walk satisfies, uniformly over distinct $x,y$,
$$
H_G(x,y)=\frac{2|E(G)|}{d_G(y)}
+\frac{r_{xy}(n-r_{xy})}{\kappa_m}
+o_{\Prob}\left(n+\frac{n^2}{m^2}\right),
$$
where $r_{xy}$ is cyclic distance between vertices $x$ and $y$, $d_G(z)$ is the degree of vertex $z$ and $\kappa_m=(m+1)(2m+1)/6$. Thus, for this random graph model, connectivity has the Erd\H{o}s--R\'enyi asymptotics with $2m$ possible neighbours, while distances, relaxation, and hitting times retain the geometry of the underlying cycle.

\end{abstract}

\clearpage
\begingroup
\small
\tableofcontents
\endgroup

\section{Introduction}

The connectivity threshold of the Erd\H{o}s--R\'enyi graph is governed by
isolated vertices: in the critical window, with high probability every
component outside the largest is an isolated vertex
\cite{Gilbert1959,ErdosRenyi1959,ErdosRenyi1960,BollobasRandomGraphs}.
What remains true when edges are restricted to a prescribed spatial range?
Such a restriction can create additional obstructions to connectivity and
can strongly affect distances and random walks even after the graph is
connected \cite{AlonBenjaminiStacey2004}.

We consider the band random graph $G(n,p,m)$, obtained by retaining each edge
of the $m$th power of the cycle independently with probability $p$.
Each vertex has $2m$ possible neighbours, all within cyclic distance $m$.
As $m$ grows, the host graph interpolates between a cycle and a complete
graph. Our results hold throughout the regime
$$
\log n \ll m < \frac{n}{2}.
$$
At the connectivity scale $p\asymp(\log n)/(2m)$, the expected degree is of
order $\log n$, regardless of $m$. Yet a path crossing a positive fraction
of the cycle needs at least order $n/m$ edges. This model therefore allows
us to vary the spatial range while keeping the expected degree at
connectivity on the same scale.

We show that the connectivity threshold and critical-window limit agree
with the Erd\H{o}s--R\'enyi predictions after replacing $n-1$ possible
neighbours by $2m$. The geometry and random-walk behavior are more sensitive
to the bandwidth. In the critical window, the largest component has
diameter of order $n/m+\log n/\log\log n$. Above any fixed factor of the
connectivity threshold, the lazy-walk spectral gap is asymptotic to that
of the host graph. The mean hitting time has two explicit contributions:
one from the degree of the target vertex, and one from the cyclic distance
to that vertex.

The spatial restriction is also the main difficulty in the proofs.
In $G(n,p)$, a set $S$ has $|S|(n-|S|)$ possible edges to its complement.
For a band graph, this number depends on the positions of the vertices,
not just on $|S|$. We handle this dependence by organizing the cycle into
consecutive blocks. The same block structure then provides the comparisons
needed for the spectral-gap and hitting-time estimates.

\subsection{Model and notation}
\label{sec:band-random-graphs-model}

For each $n\geq3$, choose a deterministic integer $m=m_n$ satisfying
\begin{equation}
\label{eq:global-growing-bandwidth-regime}
\log n\ll m<\frac n2.
\end{equation}
For positive sequences, we write $a_n\ll b_n$ when $a_n/b_n\to0$ and call
\eqref{eq:global-growing-bandwidth-regime} the \emph{growing-bandwidth regime}.
All limits are as $n\to\infty$ unless stated otherwise. The vertex set is
\begin{equation}
\label{eq:cycle-vertex-set}
V:=\mathbb Z_n=\{0,1,\ldots,n-1\},
\end{equation}
with arithmetic modulo $n$. The cyclic distance is
\begin{equation}
\label{eq:cyclic-distance-definition}
r_{xy}:=\min\{|x-y|,\,n-|x-y|\},
\qquad
x,y\in\mathbb Z_n.
\end{equation}
For $a\in\mathbb Z_n$ and $1\leq\ell\leq n$, a \emph{cyclic interval
of length $\ell$} is the set $\{a,a+1,\ldots,a+\ell-1\}$, with
arithmetic modulo $n$; an interval of
length zero is empty. Two disjoint cyclic intervals of total length less
than $n$ are \emph{adjacent} if their union is a cyclic interval.

\begin{definition}[Periodic band graph]
\label{def:periodic-band-graph}
The graph $H=C_n^m$ has vertex set $\mathbb Z_n$, with distinct vertices $x,y$ adjacent if and
only if
\begin{equation}
\label{eq:cycle-power-edge-condition}
r_{xy}\leq m.
\end{equation}
\end{definition}

We call $H=C_n^m$ the \emph{host graph}, its edges \emph{admissible edges},
and its neighbours \emph{admissible neighbours}. Each vertex has $m$
admissible neighbours in either direction, so
\begin{equation}
\label{eq:host-degree-edge-count}
d_H(x)=2m
\quad\text{for every }x\in\mathbb Z_n,
\qquad
|E(H)|=nm.
\end{equation}

\begin{definition}[Band random graph]
\label{def:band-random-graph}
Let $p=p_n\in(0,1]$ be deterministic. Independently for each edge $e\in E(H)$, let
$\xi_e\sim\Ber(p)$, and define
\begin{equation}
\label{eq:random-edge-set-definition}
E(G):=\{e\in E(H):\xi_e=1\}.
\end{equation}
The resulting graph is denoted by
\begin{equation}
\label{eq:band-random-graph-notation}
G=G(n,p,m).
\end{equation}
\end{definition}
Unless stated otherwise, $\Prob$ and $\E$ refer to this product measure.
We write $\Prob_p$ when the dependence on $p$ matters.

For a finite simple graph $\Gamma=(W,E_\Gamma)$, let $\Gamma[S]$ denote
the subgraph induced by $S\subseteq W$, and let $d_\Gamma(v)$ be the degree
of $v$. We call $S$ connected when $\Gamma[S]$ is connected. For vertices
in the same component, $\dist_\Gamma(x,y)$ is their graph distance. The
diameter of a nonempty connected graph is
$$
\diam(\Gamma):=\max_{x,y\in W}\dist_\Gamma(x,y).
$$
Write $C_1,C_2,\ldots$ for the components of $G$ in nonincreasing order
of size, breaking ties by the smallest vertex in $\{0,1,\ldots,n-1\}$.
Set $C_j=\varnothing$ if there are fewer than $j$ components. We use $C_j$
for both the vertex set and the induced graph. Let
\begin{equation}
\label{eq:global-isolated-vertex-count}
N_{\mathrm{iso}}
:=
\sum_{v\in\mathbb Z_n}\mathbf 1_{\{d_G(v)=0\}}
\end{equation}
be the number of isolated vertices.

Thus $G$ is Bernoulli bond percolation on $C_n^m$, with $G=H$ when $p=1$.
For each $x\in\mathbb Z_n$,
\begin{equation}
\label{eq:random-degree-and-isolation-distribution}
d_G(x)\sim\Bin(2m,p),
\qquad
\Prob(x\text{ is isolated})=(1-p)^{2m},
\end{equation}
and
\begin{equation}
\label{eq:expected-degree-edge-count}
\E[d_G(x)]=2mp,
\qquad
\E[|E(G)|]=nmp.
\end{equation}

The second moment of the uniform one-sided jump length will be denoted by
\begin{equation}
\label{eq:step-variance-scale}
\kappa_m
:=
\frac1m\sum_{r=1}^m r^2
=
\frac{(m+1)(2m+1)}6.
\end{equation}
Thus $\kappa_m\sim m^2/3$. This variance scale determines both the
small-frequency spectral gap and the spatial term in the hitting-time
expansion.

An event holds \emph{with high probability} if its probability tends to one.
We use $\longrightarrow$, $\Longrightarrow$, and $\xrightarrow{\Prob}$
for deterministic convergence, convergence in distribution, and convergence
in probability, respectively. For deterministic $a_n>0$, write
$X_n=o_{\Prob}(a_n)$ if $X_n/a_n\xrightarrow{\Prob}0$, and
$X_n=O_{\Prob}(a_n)$ if for every $\varepsilon>0$ there is
$M_\varepsilon<\infty$ such that
\begin{equation}
\label{eq:probabilistic-big-O-definition}
\Prob\bigl(|X_n|\leq M_\varepsilon a_n\bigr)\geq1-\varepsilon
\end{equation}
for all sufficiently large $n$. For nonnegative $X_n$, the notation
$X_n=\Theta_{\Prob}(a_n)$ means that, for every $\varepsilon>0$, there are constants
$0<c_\varepsilon<C_\varepsilon<\infty$ such that
\begin{equation}
\label{eq:probabilistic-theta-definition}
\Prob\bigl(c_\varepsilon a_n\leq X_n\leq C_\varepsilon a_n\bigr)\geq1-\varepsilon
\end{equation}
for all sufficiently large $n$.

When $c>1$ is fixed and $p\geq c\log n/(2m)$, constants indexed by $c$
depend only on $c$. High-probability statements in this range are uniform
in $p$: for each bandwidth sequence satisfying
\eqref{eq:global-growing-bandwidth-regime}, the supremum of the failure
probability over $p\in[c\log n/(2m),1]$ tends to zero. The retention
probability may be any deterministic sequence in this interval.

\subsection{Main results}
\label{subsec:intro-main-results}

We first determine the connectivity threshold and describe the largest component in the critical window. We then obtain sharp spectral-gap and hitting-time asymptotics for $p\geq c\log n/(2m)$, with $c>1$ fixed.

\paragraph{Connectivity and the critical window.}
We prove a sharp connectivity threshold throughout the growing-bandwidth regime.

\begin{theorem}[Sharp connectivity threshold]
\label{thm:sharp-connectivity}
Assume $\log n\ll m<n/2$ and let $p=c\log n/(2m)$, where $c>0$ is fixed. If $0<c<1$, then $G(n,p,m)$ is disconnected with high probability. If $c>1$, then
$G(n,p,m)$ is connected with high probability. Consequently, $\log n/(2m)$ is a sharp
threshold for connectivity.
\end{theorem}

The location of the threshold is suggested by the exact identity
$$
\E[N_{\mathrm{iso}}]=n(1-p)^{2m}=n^{1-c+o(1)}.
$$
For $c<1$, a second-moment argument gives an isolated vertex with high
probability. The main task for $c>1$ is to rule out components of size at
least two.

Our next result identifies every component in the critical window.

\begin{theorem}[Component structure in the critical window of connectivity]
\label{thm:critical-window-only-isolated-outside-giant}
Assume
\begin{equation}
\label{eq:critical-window-bandwidth-assumptions}
\log n\ll m<\frac n2,
\end{equation}
fix $t\in\mathbb R$, and choose $p$ by
\begin{equation}
\label{eq:critical-window-exact-centering}
1-p=\exp\left(-\frac{\log n+t}{2m}\right).
\end{equation}
Then
\begin{equation}
\label{eq:critical-window-poisson-conclusion}
N_{\mathrm{iso}}\Longrightarrow\Poisson(e^{-t}),
\end{equation}
and
\begin{equation}
\label{eq:critical-window-largest-component-conclusion}
\Prob\bigl(|C_1|=n-N_{\mathrm{iso}}\bigr)\longrightarrow1.
\end{equation}
Consequently,
\begin{equation}
\label{eq:critical-window-connectivity-limit}
\Prob(G\text{ is connected})\longrightarrow\exp(-e^{-t}),
\end{equation}
and
\begin{equation}
\label{eq:critical-window-second-component}
\Prob\bigl(|C_2|\in\{0,1\}\bigr)\longrightarrow1.
\end{equation}
Equivalently, with probability tending to one, every non-isolated vertex belongs to $C_1$
and every component other than $C_1$ consists of one isolated vertex.
\end{theorem}

The natural parameter for the window is $\lambda_n=n(1-p)^{2m}$.
Corollary~\ref{cor:intrinsic-critical-window} gives the same conclusions
whenever $\lambda_n\to\lambda\in(0,\infty)$, with $e^{-t}$ replaced by
$\lambda$.

The exact centering in \eqref{eq:critical-window-exact-centering} matters:
it gives $(1-p)^{2m}=e^{-t}/n$. Replacing $\log(1-p)$ by $-p$ introduces
an error $O(mp^2)=O((\log n)^2/m)$ in the exponent, which need not vanish
under $m\gg\log n$.

\paragraph{Geometry of the largest component.}
The next theorem describes how $C_1$ occupies the cycle and determines
its diameter up to constant factors.

\begin{theorem}[Connected sets in consecutive blocks and largest-component diameter]
\label{thm:critical-window-geometry}
Assume $\log n\ll m<n/2$, fix $t\in\mathbb R$, and choose $p$ by
$$
1-p=\exp\left(-\frac{\log n+t}{2m}\right).
$$
Set
\begin{equation}
\label{eq:critical-window-balanced-block-count}
q:=\left\lceil\frac nm\right\rceil,
\end{equation}
and let
\begin{equation}
\label{eq:critical-window-balanced-block-partition}
\mathbb Z_n=B_1\sqcup\cdots\sqcup B_q,
\qquad |B_b|=\ell_b,
\end{equation}
be a partition into consecutive cyclic intervals whose lengths differ by at most one. For all
sufficiently large $n$,
\begin{equation}
\label{eq:critical-window-balanced-block-lengths}
\frac m2\leq \ell_b\leq m,
\qquad 1\leq b\leq q.
\end{equation}

Let $(\eta_n)$ be any deterministic sequence satisfying
\begin{equation}
\label{eq:critical-window-backbone-deficit-conditions}
\eta_n\longrightarrow0,
\qquad
\eta_nm\longrightarrow\infty,
\qquad
\log(1/\eta_n)=o(\log n).
\end{equation}
For disjoint sets $A,B\subseteq\mathbb Z_n$, let $e_G(A,B)$ denote the number of edges of
$G$ with one endpoint in $A$ and the other in $B$. There are constants
$0<a_t<A_t<\infty$, depending only on $t$, such that, with probability tending to one,
there exist vertex sets, each connected in $G$,
\begin{equation}
\label{eq:critical-window-backbone-sets}
H_b\subseteq C_1\cap B_b,
\qquad 1\leq b\leq q,
\end{equation}
satisfying
\begin{equation}
\label{eq:critical-window-backbone-sizes}
|H_b|\geq(1-\eta_n)\ell_b,
\qquad 1\leq b\leq q,
\end{equation}
and, with $H_{q+1}:=H_1$,
\begin{equation}
\label{eq:critical-window-backbone-interface-bounds}
a_t\,m\log n
\leq e_G(H_b,H_{b+1})
\leq A_t\,m\log n,
\qquad 1\leq b\leq q.
\end{equation}

Moreover,
\begin{equation}
\label{eq:critical-window-giant-complement-identity}
|\mathbb Z_n\setminus C_1|=N_{\mathrm{iso}}=O_{\Prob}(1),
\end{equation}
and there are constants $0<c_t<C_t<\infty$, depending only on $t$, such that, with
probability tending to one,
\begin{equation}
\label{eq:critical-window-giant-diameter-two-sided}
c_t\left(\frac nm+\frac{\log n}{\log\log n}\right)
\leq \diam(C_1)
\leq
C_t\left(\frac nm+\frac{\log n}{\log\log n}\right).
\end{equation}
\end{theorem}

The sets $H_b$ form a connected chain around the cycle. Each fills all
but an $o(1)$ fraction of its block, and consecutive sets are joined by
order $m\log n$ retained edges.

The two terms in the diameter have different origins. Every edge spans
cyclic distance at most $m$, giving the lower bound of order $n/m$.
Also, $C_1$ has $n-O_{\Prob}(1)$ vertices and maximum degree $O(\log n)$,
giving the lower bound of order $\log n/\log\log n$. The matching upper
bound requires connected chains with controlled internal distances.
The spatial term dominates when $m\ll n\log\log n/\log n$; the degree
term dominates when $m\gg n\log\log n/\log n$.

\paragraph{Spectral gap.}
For a connected graph $\Gamma$, let $P_\Gamma$ be the transition matrix of simple random walk
on $\Gamma$, let $Q_\Gamma:=(I+P_\Gamma)/2$, and let $\lambda_2(Q_\Gamma)$ be the
second-largest eigenvalue of $Q_\Gamma$, counted with multiplicity. Define
$$
\gamma_{\lazy}(\Gamma):=1-\lambda_2(Q_\Gamma),
$$
and set $\gamma_{\lazy}(\Gamma):=0$ when $\Gamma$ is disconnected. Define
\begin{equation}
\label{eq:gap-reference-value}
\Gamma_{n,m}:=\frac1m\sum_{r=1}^m\sin^2\left(\frac{\pi r}{n}\right).
\end{equation}
This is the Rayleigh quotient of the first sine and cosine modes of
$H=C_n^m$. We show that it is asymptotic to $\gamma_{\lazy}(H)$ throughout
the growing-bandwidth regime.

\begin{theorem}[Uniform sharp spectral gap asymptotics]
\label{thm:spectral-gap}
Fix $c>1$ and assume
\begin{equation}
\label{eq:gap-main-assumptions}
\log n\ll m<\frac n2,
\qquad
p\geq\frac{c\log n}{2m}.
\end{equation}
Then, for every $\varepsilon>0$,
\begin{equation}
\label{eq:gap-main-uniform-conclusion}
\sup_{p\in[c\log n/(2m),1]}
\Prob_p\left(
G\text{ is disconnected}
\ \text{or}\
\left|
\frac{\gamma_{\lazy}(G)}{\Gamma_{n,m}}-1
\right|>\varepsilon
\right)
\longrightarrow0.
\end{equation}
Equivalently, uniformly over $p\in[c\log n/(2m),1]$,
\begin{equation}
\label{eq:gap-main-conclusion}
\gamma_{\lazy}(G)=\bigl(1+o_{\Prob}(1)\bigr)\Gamma_{n,m}.
\end{equation}
\end{theorem}

\begin{corollary}[Explicit spectral gap asymptotics]
\label{cor:explicit-spectral-gap}
Under the assumptions of Theorem~\ref{thm:spectral-gap}, the following statements hold
uniformly over $p\in[c\log n/(2m),1]$:
\begin{enumerate}
\item if $m=o(n)$, then
\begin{equation}
\label{eq:gap-sublinear-explicit}
\gamma_{\lazy}(G)
=
\left(\pi^2+o_{\Prob}(1)\right)\frac{\kappa_m}{n^2}
=
\left(\frac{\pi^2}{3}+o_{\Prob}(1)\right)\frac{m^2}{n^2};
\end{equation}
\item if $m/n\to\rho\in(0,1/2]$, then
\begin{equation}
\label{eq:gap-macroscopic-explicit}
\gamma_{\lazy}(G)
\xrightarrow{\Prob}
\frac12\left(1-\frac{\sin(2\pi\rho)}{2\pi\rho}\right).
\end{equation}
\end{enumerate}
In particular,
\begin{equation}
\label{eq:gap-host-comparison}
\frac{\gamma_{\lazy}(G)}{\gamma_{\lazy}(H)}
\xrightarrow{\Prob}1
\end{equation}
uniformly over the same retention probabilities.
\end{corollary}

The estimate is relative, which is essential when $m=o(n)$: the host
gap then has order $m^2/n^2$ and tends to zero. An additive $o_{\Prob}(1)$
error would not determine its order or leading constant. The theorem
shows that the gap retains the host value even at expected degree of
order $\log n$.

\paragraph{Point-to-point mean hitting times.}
On a connected realization of $G$, let $H_G(x,y)$ denote the mean time at which simple
random walk started at $x$ first visits $y$. Define
\begin{equation}
\label{eq:hitting-max-definition}
H_{\max}(G):=
\max_{\substack{x,y\in\mathbb Z_n\\x\neq y}}H_G(x,y),
\end{equation}
write
\begin{equation}
\label{eq:hitting-main-notation}
V_G:=2|E(G)|,
\qquad
d_{\min}(G):=\min_{y\in\mathbb Z_n}d_G(y),
\end{equation}
and set
\begin{equation}
\label{eq:global-hitting-scale}
T_{n,m}:=n+\frac{n^2}{m^2}.
\end{equation}
The next theorem separates the dependence on the target degree from the
dependence on cyclic distance.
\begin{theorem}[Uniform point-to-point mean hitting-time asymptotics]
\label{thm:sharp-hitting}
Fix $c>1$ and assume
\begin{equation}
\label{eq:sharp-hitting-assumptions}
\log n\ll m<\frac n2,
\qquad
p\geq\frac{c\log n}{2m}.
\end{equation}
There exist a deterministic sequence $\varepsilon_n\to0$, events $\Omega_{n,p}$, and random
remainders $R_{n,p}(x,y)$ such that
\begin{equation}
\label{eq:sharp-hitting-uniform-events}
\inf_{p\in[c\log n/(2m),1]}\Prob_p(\Omega_{n,p})\longrightarrow1,
\end{equation}
$G$ is connected on $\Omega_{n,p}$, and, on $\Omega_{n,p}$, simultaneously for all distinct
$x,y\in\mathbb Z_n$,
\begin{equation}
\label{eq:pointwise-hitting-expansion}
H_G(x,y)
=
\frac{V_G}{d_G(y)}
+
\frac{r_{xy}(n-r_{xy})}{\kappa_m}
+
R_{n,p}(x,y),
\end{equation}
where
\begin{equation}
\label{eq:pointwise-hitting-remainder}
\max_{\substack{x,y\in\mathbb Z_n\\x\neq y}}
|R_{n,p}(x,y)|
\leq
\varepsilon_nT_{n,m}.
\end{equation}
Consequently,
\begin{equation}
\label{eq:sharp-hitting-conclusion}
H_{\max}(G)
=
\bigl(1+o_{\Prob}(1)\bigr)
\left(
\frac{n^2}{4\kappa_m}
+
\frac{V_G}{d_{\min}(G)}
\right)
\end{equation}
uniformly over $p\in[c\log n/(2m),1]$.
\end{theorem}

To express the maximal hitting time deterministically, define the
minimum-degree profile $\delta_{\min}:(1,\infty]\to(0,1]$ as follows. For $\theta\in(1,\infty)$,
$\delta_{\min}(\theta)$ is the unique solution in $(0,1)$ of
\begin{equation}
\label{eq:minimum-degree-profile}
\theta\bigl(1-\delta_{\min}(\theta)
+\delta_{\min}(\theta)\log\delta_{\min}(\theta)\bigr)=1,
\end{equation}
and $\delta_{\min}(\infty):=1$. Let
\begin{equation}
\label{eq:normalized-mean-degree}
\theta_n:=\frac{2mp}{\log n}.
\end{equation}

\begin{corollary}[Deterministic equivalent for the maximal hitting time]
\label{cor:explicit-maximal-hitting}
Under the assumptions of Theorem~\ref{thm:sharp-hitting},
\begin{equation}
\label{eq:explicit-maximal-hitting}
H_{\max}(G)
=
\bigl(1+o_{\Prob}(1)\bigr)
\left(
\frac{n^2}{4\kappa_m}
+
\frac{n}{\delta_{\min}(\theta_n)}
\right)
\end{equation}
uniformly over $p\in[c\log n/(2m),1]$. If
$\theta_n\to\theta\in[c,\infty]$, then
\begin{equation}
\label{eq:explicit-maximal-hitting-limit}
H_{\max}(G)
=
\bigl(1+o_{\Prob}(1)\bigr)
\left(
\frac{n^2}{4\kappa_m}
+
\frac{n}{\delta_{\min}(\theta)}
\right).
\end{equation}
For the deterministic cycle power,
\begin{equation}
\label{eq:host-maximal-hitting}
H_{\max}(C_n^m)
=
\bigl(1+o(1)\bigr)
\left(
\frac{n^2}{4\kappa_m}+n
\right).
\end{equation}
\end{corollary}

The first term in \eqref{eq:pointwise-hitting-expansion} is $V_G/d_G(y)=1/\pi_G(y)$, where $\pi_G$ is the stationary distribution. It is larger for targets of smaller degree. The second term depends on cyclic distance and is maximized at $r_{xy}=\lfloor n/2\rfloor$.
Since the error bound is uniform over all ordered pairs, we may choose a target of minimum degree and a starting vertex at cyclic distance $\lfloor n/2\rfloor$ from it. This gives the formula for $H_{\max}(G)$.

At $p=c\log n/(2m)$, the target-degree contribution to $H_{\max}(G)$ is
asymptotic to $n/\delta_{\min}(c)$. The two contributions have the same
order when $m\asymp\sqrt n$. The spatial term dominates for
$m=o(\sqrt n)$, and the target-degree term dominates for $m\gg\sqrt n$.

\subsection{Proof strategy}
\label{subsec:intro-proof-strategy}

\paragraph{Connectivity and diameter.}
In the critical window, for each vertex $v$ we build two connected chains
$\mathcal K_v^-$ and $\mathcal K_v^+$ from large connected subsets of
consecutive blocks. We leave all edges incident to $v$ unexposed during
the construction. Each chain contains more than $n/2$ vertices, and their
union misses only $o(m/\log n)$ admissible neighbours of $v$.
Conditional on the chains, the edges at $v$ are still independent.
This lets us compare the probability that $v$ fails to attach with the
probability that it is isolated. The error is small enough to sum over
all vertices.

For the diameter bound, the chains must also have diameter
$$
O\left(\frac nm+\frac{\log n}{\log\log n}\right).
$$
Any two chains intersect because each contains more than half the
vertices. Once every non-isolated vertex is attached to a chain, the
upper bound follows by joining vertices through an intersection.

\paragraph{Spectral gap.}
Testing the Rayleigh quotient on the first Fourier modes gives the
upper bound. For the lower bound when $m=o(n)$, we decompose functions
into their block averages and their fluctuations within blocks.
The Laplacian is bounded below by a constant times $2mp$ on the
fluctuation space. Eliminating these fluctuations by a Schur complement
reduces the problem to a cycle power on the blocks, with a relative
error tending to zero. When $m/n$ is bounded away from zero, the host
gap stays positive, and an operator-norm comparison of normalized
adjacency matrices suffices.

\paragraph{Hitting times.}
We first extract the target-degree term $V_G/d_G(y)$ exactly from the
hitting-time equation. The remaining Laplacian equation is treated by
the same decomposition into block averages and fluctuations.
Effective-resistance estimates control the error uniformly in the
starting and target vertices. Fourier inversion on the resulting cycle
power gives $r_{xy}(n-r_{xy})/\kappa_m$. This requires finer information
than the spectral gap alone.

\subsection{Context and related work}
\label{subsec:intro-related-work}

For $G(n,p)$ at $p=(\log n+t)/n$, with fixed $t\in\mathbb R$, the
number of isolated vertices converges to $\Poisson(e^{-t})$, every other
vertex lies in the largest component with high probability, and the
connectivity probability tends to $\exp(-e^{-t})$
\cite{BollobasRandomGraphs,JLR2000,hofstad2017randomgraphs}.
Theorem~\ref{thm:critical-window-only-isolated-outside-giant} gives the corresponding law for band random graphs under the exact centering \eqref{eq:critical-window-exact-centering}.

Random geometric graphs provide another natural comparison. Gilbert introduced the continuum model in \cite{Gilbert1961}, and Penrose developed its connectivity and component theory in
\cite{Penrose2003RGG}. Minimum-degree and connectivity thresholds coincide under suitable assumptions for hard geometric graphs
\cite{Penrose1999KConnectivity,GoelRaiKrishnamachari2005,PenroseYang2025}. For soft geometric graphs, connectivity is asymptotically equivalent to the absence of isolated vertices in broad parameter ranges
\cite{Penrose2016Soft,Penrose2018Poisson}.
In one dimension, a nontrivial interval component can survive without
isolated vertices if no retained edge crosses either boundary.
Related interval and torus models are studied in
\cite{WilsherDettmannGanesh2020,WilsherDettmannGanesh2023}, and
bounded-range percolation on a discrete cycle in
\cite{Federico2024BoundedRange}. Our critical-window theorem rules out
all such nontrivial components outside the largest one when
$m/\log n\to\infty$.

More generally, random subgraphs of deterministic hosts have been studied
through expansion, degree, and product structure. Alon, Benjamini, and
Stacey proved uniqueness of a component of positive density for
percolation on bounded-degree expanders
\cite{AlonBenjaminiStacey2004}. Connectivity laws are known for
Cartesian powers and Hamming graphs
\cite{Joos2015,FedericoHofstadHulshof2016}, and other criteria use
density, edge connectivity, or expansion
\cite{GuptaLeeLi2021,DiskinGeisler2025}. Triangle conditions give
mean-field behavior near the emergence of a giant component in several
high-dimensional product graphs
\cite{BorgsChayesHofstadSladeSpencer2005,DiskinErdeKangKrivelevich2025}.
These results do not directly yield our critical-window theorem.
When $m=o(n)$, a half-cycle has conductance of order $m/n$, so the host
does not have expansion bounded away from zero. Our proof instead uses
the geometry of consecutive blocks and preserves independence by
leaving the edges at a prescribed vertex unexposed.

Spatial constraints also influence random walks after connectivity.
The largest supercritical percolation cluster in a box of side length
$n$ in $\mathbb Z^d$ has mixing time of order $n^2$
\cite{BenjaminiMossel2003}; heat-kernel bounds and invariance principles
for the infinite cluster appear in
\cite{Barlow2004,MathieuPiatnitski2007}.
For geometric graphs on a torus, the mixing scale is $r^{-2}$, where $r$
is the connection radius \cite{BoydGhoshPrabhakarShah2005}.
Avin and Ercal separated the radii needed for optimal-order cover time
and rapid mixing \cite{AvinErcal2007}, while M\"uller and Penrose proved continuum limits for rescaled Cheeger constants \cite{MullerPenrose2020}. Long-range percolation on a finite cycle can have polynomial mixing time despite polylogarithmic diameter \cite{BenjaminiBergerYadin2008}. For the band random graph, we determine the leading constant of the lazy-walk spectral gap for $p\geq c\log n/(2m)$, with $c>1$ fixed.

General spectral comparisons for independent edge sampling are
typically additive. In particular, the estimate of Chung and Horn
\cite{ChungHorn2007} does not resolve the host gap of order $m^2/n^2$
at $p\asymp(\log n)/m$ when $m=o(n)$. Our comparison has relative error
$o_{\Prob}(1)$ and gives the normalization
$\pi^2\kappa_m/n^2\sim(\pi^2/3)m^2/n^2$.
When $m$ is proportional to $n$, the proof uses the matrix-norm
estimate of Bandeira and van Handel \cite{BandeiraVanHandel2016}.

Effective resistance gives exact formulas for commute times
\cite{ChandraRaghavanRuzzoSmolenskyTiwari1989} and directed hitting
times \cite{Tetali1991}. Under suitable degree and spectral assumptions,
hitting times in several random-neighbourhood models are asymptotic
to the graph volume divided by the target degree
\cite{VonLuxburgRadlHein2014}. Near connectivity in sparse $G(n,p)$,
effective resistance concentrates around the sum of reciprocal endpoint
degrees \cite{Sylvester2021}. Our hitting-time expansion retains the
target-degree term $V_G/d_G(y)$ and adds an explicit spatial correction,
$$
\frac{r_{xy}(n-r_{xy})}{\kappa_m}.
$$
The minimum-degree asymptotics then identify the leading constant of
the maximal hitting time.

Finally, the adjacency matrix decomposes as
$$
A_G=pA_H+\bigl(A_G-pA_H\bigr),
$$
with a deterministic circulant mean and a centered symmetric band
fluctuation. Random band matrices are expected to exhibit a
localization--delocalization transition at bandwidth $\sqrt n$
\cite{ErdosKnowlesYauYin2013Delocalization,Bourgade2018RandomBandMatrices}.
Recent work establishes this transition for broad classes of centered
one-dimensional models
\cite{YauYin2025Delocalization,ErdosRiabov2025,Drogin2025}.
Those results do not apply directly here: $pA_H$ has rank proportional
to $n$ and nonconstant Fourier eigenvalues, and $p$ may tend to zero. The graph estimates in this paper provide a starting point for studying the bulk and upper spectral edge of $A_G$.

\paragraph{Organization.}
Section~\ref{sec:prelim} details the preliminary tools used in the subsequent sections.
Sections~\ref{sec:connectivity-threshold}
and~\ref{sec:critical-connectivity-window} prove the connectivity and
geometry results. Sections~\ref{sec:spectral-gap}
and~\ref{sec:hitting-time} treat the spectral gap and hitting times,
respectively. Section~\ref{sec:conclusion} presents some open problems.

\section{Preliminaries}
\label{sec:prelim}

We collect the notation and standard facts used in the proofs.

\subsection{Graphs and random walks}
\label{subsec:graph-markov-prelim}

Let $\Gamma=(V,E)$ be a finite simple undirected graph, with notation
as in Subsection~\ref{sec:band-random-graphs-model}. Its adjacency matrix,
degree matrix, and combinatorial Laplacian are
\begin{align}
\label{eq:adjacency-degree-definitions}
A_\Gamma(x,y)&:=\mathbf 1_{\{\{x,y\}\in E\}},\\
\label{eq:degree-matrix-laplacian-definitions}
D_\Gamma&:=\diag\bigl(d_\Gamma(x):x\in V\bigr),
&
L_\Gamma&:=D_\Gamma-A_\Gamma.
\end{align}
If $\Gamma$ has no isolated vertices, the transition matrices of simple and lazy random
walk are
\begin{equation}
\label{eq:simple-lazy-walk-definitions}
P_\Gamma:=D_\Gamma^{-1}A_\Gamma,
\qquad
Q_\Gamma:=\frac{I+P_\Gamma}{2}.
\end{equation}
Define the volume of $S\subseteq V$ and, when $\Gamma$ is connected, the stationary
distribution of simple random walk by
\begin{equation}
\label{eq:graph-stationary-distribution}
\vol_\Gamma(S):=\sum_{x\in S}d_\Gamma(x),
\qquad
\vol_\Gamma(V)=2|E|,
\qquad
\pi_\Gamma(x):=\frac{d_\Gamma(x)}{\vol_\Gamma(V)}.
\end{equation}
When $\Gamma$ is connected, $P_\Gamma$ is irreducible and reversible.

For a finite irreducible reversible transition matrix $P$ with stationary distribution
$\pi$ and $f:V\to\mathbb R$, define
\begin{align}
\label{eq:markov-inner-product-variance}
\Var_\pi(f)
&:=
\min_{a\in\mathbb R}\sum_{x\in V}\pi(x)\bigl(f(x)-a\bigr)^2,\\
\label{eq:markov-dirichlet-form}
\mathcal E_P(f,f)
&:=
\frac12\sum_{x,y\in V}\pi(x)P(x,y)\bigl(f(x)-f(y)\bigr)^2.
\end{align}
For the simple random walk on a connected graph,
\begin{equation}
\label{eq:graph-dirichlet-edge-formula}
\mathcal E_{P_\Gamma}(f,f)
=
\frac{1}{\vol_\Gamma(V)}
\sum_{\{x,y\}\in E}\bigl(f(x)-f(y)\bigr)^2.
\end{equation}
Let $\lambda_2(P)$ be the second-largest eigenvalue of $P$, with eigenvalues ordered
algebraically and counted with multiplicity. For a connected graph $\Gamma$, define
\begin{equation}
\label{eq:graph-lazy-gap-definition}
\gamma_{\lazy}(\Gamma)
:=
1-\lambda_2(Q_\Gamma)
=
\frac{1-\lambda_2(P_\Gamma)}{2}.
\end{equation}
For a disconnected graph, including one with isolated vertices, we set
$\gamma_{\lazy}(\Gamma):=0$.

\begin{lemma}[Variational characterization of the spectral gap]
\label{lem:variational-gap}
Let $P$ be a finite irreducible reversible Markov chain with stationary distribution $\pi$.
Then
\begin{equation}
\label{eq:variational-gap-formula}
1-\lambda_2(P)
=
\inf_{f\not\equiv\mathrm{const}}
\frac{\mathcal E_P(f,f)}{\Var_\pi(f)}.
\end{equation}
\end{lemma}

\subsection{Electrical networks and hitting times}
\label{subsec:electrical-prelim}

Let $\Gamma=(V,E)$ be a finite connected network with positive symmetric conductances
$c_{xy}=c_{yx}$. Its weighted Laplacian is
\begin{equation}
\label{eq:weighted-laplacian-definition}
(L_\Gamma f)(x)
:=
\sum_{y:\{x,y\}\in E}c_{xy}\bigl(f(x)-f(y)\bigr).
\end{equation}
For distinct vertices $a,b\in V$, let $h$ solve
\begin{equation}
\label{eq:unit-current-potential-equation}
L_\Gamma h=\mathbf 1_{\{a\}}-\mathbf 1_{\{b\}}.
\end{equation}
The solution is unique up to an additive constant, and the effective resistance is
\begin{equation}
\label{eq:effective-resistance-potential-definition}
R_{\mathrm{eff}}^\Gamma(a,b)
:=
h(a)-h(b)
=
\bigl(\mathbf 1_{\{a\}}-\mathbf 1_{\{b\}}\bigr)^{\mathsf T}
L_\Gamma^+
\bigl(\mathbf 1_{\{a\}}-\mathbf 1_{\{b\}}\bigr),
\end{equation}
where $L_\Gamma^+$ denotes the Moore--Penrose pseudoinverse.
We set $R_{\mathrm{eff}}^\Gamma(a,a):=0$.
In an unweighted graph, all conductances equal $1$.

For a finite connected unweighted graph $\Gamma$, let $(X_t)_{t\geq0}$ be its simple random
walk, define $\tau_v:=\inf\{t\geq0:X_t=v\}$, and write
$H_\Gamma(u,v):=\E_u[\tau_v]$.

\begin{theorem}[Commute-time identity]
\label{thm:commute-time}
For every $u,v\in V$,
\begin{equation}
\label{eq:commute-time-identity}
H_\Gamma(u,v)+H_\Gamma(v,u)
=
2|E|R_{\mathrm{eff}}^\Gamma(u,v)
=
\vol_\Gamma(V)R_{\mathrm{eff}}^\Gamma(u,v).
\end{equation}
\end{theorem}

The commute-time identity is due to Chandra et
al.~\cite{ChandraRaghavanRuzzoSmolenskyTiwari1989}. For further background
on electrical networks and hitting times, see
\cite{Tetali1991, AldousFill2002, LawlerLimic2010,LyonsPeres2016,levin2017markov}.

\subsection{Comparison notation and concentration inequalities}
\label{subsec:prob-prelim}

For deterministic nonnegative quantities $a$ and $b$, the notation $a\lesssim_\alpha b$
means $a\leq C_\alpha b$ for a constant
depending only on $\alpha$, and $a\asymp_\alpha b$ means that both
$a\lesssim_\alpha b$ and $b\lesssim_\alpha a$ hold. A missing subscript means that the
constants are absolute. Subscripts on $O(\,\cdot\,)$ and $\Theta(\,\cdot\,)$ have the
same meaning.

We use the following forms of the Chernoff and Bernstein inequalities; see, for example,
\cite[Chapter~2]{vershynin2018high}.

\begin{lemma}[Multiplicative Chernoff bound]
\label{lem:chernoff}
Let $X_1,\ldots,X_N$ be independent Bernoulli random variables, set
$X:=\sum_{i=1}^NX_i$, and let $\mu:=\E[X]$.
For every $\delta\in(0,1)$,
\begin{align}
\label{eq:chernoff-lower-tail}
\Prob\bigl(X\leq(1-\delta)\mu\bigr)
&\leq
\exp\left(-\frac{\delta^2\mu}{2}\right),\\
\label{eq:chernoff-upper-tail}
\Prob\bigl(X\geq(1+\delta)\mu\bigr)
&\leq
\exp\left(-\frac{\delta^2\mu}{3}\right).
\end{align}
Consequently,
\begin{equation}
\label{eq:chernoff-two-sided}
\Prob\bigl(|X-\mu|\geq\delta\mu\bigr)
\leq
2\exp\left(-\frac{\delta^2\mu}{3}\right).
\end{equation}
\end{lemma}

\begin{lemma}[Bernstein inequality]
\label{lem:bernstein}
Let $X_1,\ldots,X_N$ be independent real-valued random variables satisfying
$\E[X_i]=0$ and $|X_i|\leq K$ almost surely for some $K>0$. Let
$S_N:=\sum_{i=1}^NX_i$ and $\sigma^2:=\sum_{i=1}^N\Var(X_i)$.
Then, for every $t>0$,
\begin{equation}
\label{eq:bernstein-tail}
\Prob\bigl(|S_N|\geq t\bigr)
\leq
2\exp\left(-\frac{t^2}{2(\sigma^2+Kt/3)}\right).
\end{equation}
\end{lemma}

\section{Sharp connectivity threshold}
\label{sec:connectivity-threshold}

We prove Theorem~\ref{thm:sharp-connectivity} by treating the two sides of the threshold
separately. Throughout this section, $G=G(n,p,m)$, $\log n\ll m<n/2$, and
$p=c\log n/(2m)$ with fixed $c>0$. The identity
$$
\E[N_{\mathrm{iso}}]=n(1-p)^{2m}=n^{1-c+o(1)},
$$
suggests that isolated vertices are the obstruction when $c<1$. For $c>1$, we first
join large components in consecutive blocks, then attach every remaining vertex.
To preserve independence in the last step, we construct the relevant block component
without exposing edges incident to the vertex being attached. We will refine this
construction in the critical window.

\subsection{Subcritical disconnectedness from isolated vertices}
\label{subsec:conn-lower}

For $p=c\log n/(2m)$ with fixed $0<c<1$, a second-moment argument suffices.
The bandwidth assumption gives $mp^2=o(\log n)$, so
$2m\log(1-p)=-2mp+o(\log n)$ and the expected number of isolated vertices diverges.

\begin{theorem}[Subcritical disconnectedness from isolated vertices]
\label{thm:connectivity-lower}
Assume $\log n\ll m<n/2$ and
\begin{equation}
\label{eq:connectivity-subcritical-scaling}
p=\frac{c\log n}{2m},
\qquad
0<c<1.
\end{equation}
Then $G(n,p,m)$ is disconnected with high probability.
\end{theorem}

\begin{proof}
Let $I_v$ be the indicator that $v$ is isolated, set $X:=\sum_v I_v$, and put
$q:=(1-p)^{2m}$. Since $mp^2=O((\log n)^2/m)=o(\log n)$,
$$
\E X=nq=\exp\{\log n-2mp+O(mp^2)\}=n^{1-c+o(1)}\longrightarrow\infty.
$$
The indicators $I_u$ and $I_v$ are independent unless $u$ and $v$ are adjacent in the
host. For an adjacent pair, there are $4m-1$ distinct incident host edges, giving
$\operatorname{Cov}(I_u,I_v)=p(1-p)^{4m-1}$. Summing over the $nm$ host edges yields
$$
\frac{\Var X}{(\E X)^2}
\leq \frac1{\E X}+\frac{2mp}{n(1-p)}\longrightarrow0.
$$
Chebyshev's inequality gives $X/\E X\to1$ in probability, and hence
$\Prob(X>0)\to1$.
\end{proof}

\subsection{Supercritical connectivity}
\label{subsec:conn-upper}

We need two ingredients: a large component inside each block, and enough admissible
pairs between adjacent blocks to join these components.

\begin{lemma}[Large components in subsets inducing complete host subgraphs]
\label{lem:band-block-large-component}
Let $\beta>0$. Suppose that
\begin{equation}
\label{eq:connectivity-general-block-density}
\log n\ll m<\frac n2,
\qquad
pm\geq\beta\log n.
\end{equation}
Let $(\rho_n)$ be deterministic and satisfy
\begin{equation}
\label{eq:connectivity-general-block-deficit}
0<\rho_n\leq\frac14,
\qquad
\rho_nm\longrightarrow\infty,
\qquad
\log(1/\rho_n)=o(\log n).
\end{equation}
If $J\subseteq\mathbb Z_n$ induces a complete subgraph of $C_n^m$ and
$m/2\leq |J|\leq m$, then, for all sufficiently large $n$,
\begin{equation}
\label{eq:connectivity-single-complete-subset-failure}
\Prob\left(
G[J]\text{ has no component of size at least }(1-\rho_n)|J|
\right)
\leq
2\exp\left(-\frac{\beta}{12}\rho_nm\log n\right).
\end{equation}
Consequently, for any deterministic collection $J_1,\ldots,J_L$ of such sets,
\begin{equation}
\label{eq:connectivity-complete-subset-family-failure}
\Prob\left(
\exists j\leq L:\
G[J_j]\text{ has no component of size at least }(1-\rho_n)|J_j|
\right)
\leq
2L\exp\left(-\frac{\beta}{12}\rho_nm\log n\right).
\end{equation}
\end{lemma}

\begin{proof}
Write $\ell=|J|$. If no component has size at least $(1-\rho_n)\ell$, there is a
union $S$ of components with $\rho_n\ell\leq |S|\leq\ell/2$: take the complement
of a component larger than $\ell/2$, if one exists, and otherwise collect components
greedily. No edge joins this set to its complement. Thus
$$
\Prob(\text{failure})
\leq
\sum_{s=\lceil\rho_n\ell\rceil}^{\lfloor\ell/2\rfloor}
\binom{\ell}{s}(1-p)^{s(\ell-s)}.
$$
For each summand, $\ell-s\geq m/4$ and
$\binom{\ell}{s}\leq(e/\rho_n)^s$. For all large $n$,
$\log(e/\rho_n)\leq(\beta/12)\log n$ and
$p(\ell-s)\geq(\beta/4)\log n$, so the summand is at most
$\exp\{-(\beta/6)s\log n\}$. Summing this geometric tail and using
$s\geq\rho_nm/2$ proves \eqref{eq:connectivity-single-complete-subset-failure};
the family estimate follows by a union bound.
\end{proof}

The next lemma shows that removing a small fraction of the vertices from two adjacent
blocks still leaves many admissible pairs between them.

\begin{lemma}[Admissible pairs across an adjacent-block interface]
\label{lem:adjacent-block-admissible-pairs}
Let $I$ and $I'$ be adjacent cyclic intervals with
$m/2\leq |I|,|I'|\leq m$. For $K\subseteq I$ and $K'\subseteq I'$, define
\begin{equation}
\label{eq:connectivity-adjacent-block-pair-definition}
M_H(K,K')
:=
\left|
\left\{
(x,y)\in K\times K':r_{xy}\leq m
\right\}
\right|.
\end{equation}
If, for some $\rho\leq1/32$,
$$
|K|\geq(1-\rho)|I|,
\qquad
|K'|\geq(1-\rho)|I'|,
$$
then
\begin{equation}
\label{eq:connectivity-adjacent-block-pair-bounds}
\frac{m^2}{16}\leq M_H(K,K')\leq m^2.
\end{equation}
\end{lemma}

\begin{proof}
Set $s:=\min\{|I|,|I'|\}$ and assume $|I|=s$ by symmetry. Order $I$ toward the
common boundary and $I'$ away from it. The $j$th vertex of $I$ is within $s\leq m$
steps of each of the first $j$ vertices of $I'$. Hence
$$
M_H(I,I')\geq1+2+\cdots+s\geq\frac{m^2}{8}.
$$
Removing $I\setminus K$ and $I'\setminus K'$ deletes at most
$2\rho m^2\leq m^2/16$ of these pairs. This proves the lower bound; the upper bound follows
from $M_H(K,K')\leq |K||K'|\leq m^2$.
\end{proof}

A balanced partition lets us apply both block estimates around the entire cycle.

\begin{lemma}[Balanced partition into cyclic intervals]
\label{lem:balanced-cyclic-partition}
Assume $m\to\infty$ and $m<n/2$. Set $q:=\lceil n/m\rceil$ and partition
$\mathbb Z_n$ into $q$ consecutive cyclic intervals
$$
\mathbb Z_n=B_1\sqcup\cdots\sqcup B_q
$$
whose lengths $\ell_b:=|B_b|$ differ by at most one. Then, for all sufficiently large
$n$,
\begin{equation}
\label{eq:connectivity-supercritical-balanced-block-lengths}
\frac m2\leq\ell_b\leq m,
\qquad 1\leq b\leq q.
\end{equation}
\end{lemma}

\begin{proof}
The definition of $q$ gives
$$
\frac{nm}{n+m}\leq\frac nq\leq m.
$$
Because $m<n/2$, the lower bound is greater than $2m/3$. Each $\ell_b$ belongs to
$\{\lfloor n/q\rfloor,\lceil n/q\rceil\}$. Since $m$ is an integer, the upper bound
$n/q\leq m$ gives $\ell_b\leq m$, while
$\lfloor n/q\rfloor>2m/3-1\geq m/2$ for all sufficiently large $n$.
\end{proof}

We can now prove connectivity. The three preceding lemmas produce a connected chain
of large block components. To attach a vertex $v$, we also need the large component
in its own block after deleting $v$. This leaves every edge incident to $v$ available
for the final exposure.

\begin{theorem}[Supercritical connectivity]
\label{thm:connectivity-upper}
Assume $\log n\ll m<n/2$ and
\begin{equation}
\label{eq:connectivity-supercritical-fixed-factor}
p=p_c:=\frac{c\log n}{2m},
\qquad
c>1.
\end{equation}
Then $G(n,p,m)$ is connected with high probability.
\end{theorem}

\begin{proof}
The proof has three steps: construct large block components, join them, and attach
every vertex to their union.

\medskip\noindent\emph{Step 1: large block components.}
Choose
\begin{equation}
\label{eq:connectivity-supercritical-delta-choice}
0<\delta<
\min\left\{\frac1{40},\frac{2(c-1)}{5c}\right\},
\qquad
\alpha:=\frac{c(2-5\delta)}2>1.
\end{equation}
Set $q:=\lceil n/m\rceil$ and partition $\mathbb Z_n$ into consecutive cyclic intervals
$$
\mathbb Z_n=B_1\sqcup\cdots\sqcup B_q
$$
whose lengths $\ell_b:=|B_b|$ differ by at most one. The bounds in
\eqref{eq:connectivity-supercritical-balanced-block-lengths} follow from
Lemma~\ref{lem:balanced-cyclic-partition}. Block indices are interpreted modulo $q$.

Let $H_b$ be the component of $G[B_b]$ containing at least
$(1-\delta)\ell_b$ vertices when such a component exists, and set $H_b=\varnothing$
otherwise. For each $v\in B_b$, let
$K_{b,v}$ be the component of $G[B_b\setminus\{v\}]$ containing at least
$(1-\delta)(\ell_b-1)$ vertices when such a component exists, and set
$K_{b,v}=\varnothing$ otherwise. These components are unique whenever they exist,
since their required sizes exceed half of the corresponding vertex sets for all sufficiently large $n$.

Each $B_b$, as well as each $B_b\setminus\{v\}$, induces a complete subgraph of the host.
The proof of Lemma~\ref{lem:balanced-cyclic-partition} gives
$$
\ell_b\geq\left\lfloor\frac nq\right\rfloor>\frac{2m}{3}-1.
$$
As $m\to\infty$, we also have $\ell_b-1\geq\frac m2$ for all sufficiently large $n$.
Apply Lemma~\ref{lem:band-block-large-component} with $\beta=c/2$ and $\rho_n=\delta$
to these $q+n$ sets. Let $\mathcal E_{\mathrm{core}}$ be the event that every $H_b$
and $K_{b,v}$ is nonempty. Since $q\leq2n/m$,
\begin{equation}
\label{eq:connectivity-supercritical-all-cores}
\Prob(\mathcal E_{\mathrm{core}}^c)
\leq
2(q+n)\exp\left(-\frac{c\delta}{24}m\log n\right)
=o(1).
\end{equation}
On $\mathcal E_{\mathrm{core}}$, we claim that $K_{b,v}\subseteq H_b$. By
\eqref{eq:connectivity-supercritical-delta-choice}, $\delta<1/40$, while
$\ell_b\to\infty$ uniformly in $b$. Thus, for all sufficiently large $n$,
$$
|K_{b,v}|
\geq
(1-\delta)(\ell_b-1)
>
\frac{\ell_b}{2}.
$$
The component of $G[B_b]$ containing $K_{b,v}$ must therefore be $H_b$, the unique
component with more than $\ell_b/2$ vertices.

\medskip\noindent\emph{Step 2: join consecutive components.}
Conditional on all within-block edges, Lemma~\ref{lem:adjacent-block-admissible-pairs}
provides at least $m^2/16$ unexposed admissible pairs between $H_b$ and $H_{b+1}$.
Consequently,
\begin{equation}
\label{eq:connectivity-supercritical-interface-failure}
\Prob\left(
e_G(H_b,H_{b+1})=0
\;\middle|\;
\text{all within-block edges}
\right)
\leq
(1-p)^{m^2/16}
\leq
\exp\left(-\frac{c}{32}m\log n\right)
\end{equation}
on $\mathcal E_{\mathrm{core}}$. Put
$\mathcal H:=\bigcup_{b=1}^qH_b$. A conditional union bound over the
$q=O(n/m)$ interfaces shows that the event
\begin{equation}
\label{eq:connectivity-supercritical-connected-block-chain}
\mathcal E_{\mathrm{chain}}
:=
\left\{
\mathcal H\text{ is connected}
\right\}
\end{equation}
has probability tending to one, since
$$
\Prob(\mathcal E_{\mathrm{chain}}^c)
\leq
\Prob(\mathcal E_{\mathrm{core}}^c)
+
q\exp\left(-\frac{c}{32}m\log n\right)
=o(1).
$$

\medskip\noindent\emph{Step 3: attach every vertex.}
Starting at $B_{b-1}$, take consecutive blocks in decreasing cyclic order until their
union $L_b$ has at least $m$ vertices. Define $R_b$ in the same way, starting at
$B_{b+1}$ and proceeding in increasing cyclic order. Each block has between $m/2$
and $m$ vertices, so each of $L_b$ and $R_b$ uses one or two blocks and satisfies
$$
m\leq |L_b|<2m,
\qquad
m\leq |R_b|<2m.
$$
Moreover,
$$
n-\ell_b\geq n-m>m.
$$
Thus both constructions reach $m$ vertices before returning to $B_b$, and
$L_b$ and $R_b$ are disjoint from $B_b$.

If $v\in B_b$, every admissible neighbour of $v$ outside $B_b$ is among the first
$m$ vertices beyond one of the endpoints of $B_b$. These vertices lie in $L_b\cup R_b$.

Fix $v\in B_b$ and expose all edge indicators except those incident to $v$.
This determines $K_{b,v}$ and the components $H_a$ in $L_b\cup R_b$.
If these components exist, the admissible edges from $v$ to
\begin{equation}
\label{eq:connectivity-supercritical-vertex-target}
K_{b,v}\ \cup\
\bigcup_{B_a\subseteq L_b\cup R_b}H_a
\end{equation}
remain independent Bernoulli$(p)$ variables. There are at least
$(1-\delta)(\ell_b-1)$ targets inside $B_b$. Of the $2m-\ell_b+1$ admissible
neighbours outside $B_b$, fewer than $4\delta m$ are omitted by the components
$H_a$ in $L_b$ and $R_b$. We count a block appearing on both sides only once.
Thus the number of distinct admissible neighbours in
\eqref{eq:connectivity-supercritical-vertex-target} is at least
\begin{align}
(1-\delta)(\ell_b-1)
+2m-\ell_b+1-4\delta m
&=
2m-\delta(\ell_b-1)-4\delta m
\notag\\
&\geq(2-5\delta)m.
\label{eq:connectivity-supercritical-vertex-target-count}
\end{align}
Let $\mathcal A_v$ be the event that the components in
\eqref{eq:connectivity-supercritical-vertex-target} exist but no retained edge
joins $v$ to that union. After averaging over the exposed indicators,
\begin{equation}
\label{eq:connectivity-supercritical-single-vertex-failure}
\Prob(\mathcal A_v)
\leq
(1-p)^{(2-5\delta)m}
\leq n^{-\alpha}.
\end{equation}
A union bound gives
\begin{equation}
\label{eq:connectivity-supercritical-all-vertices-attach}
\Prob\left(\bigcup_{v\in\mathbb Z_n}\mathcal A_v\right)
\leq n^{1-\alpha}=o(1).
\end{equation}

On
$\mathcal E_{\mathrm{core}}\cap\mathcal E_{\mathrm{chain}}
\cap\bigcap_{v\in\mathbb Z_n}\mathcal A_v^c$,
the target set in \eqref{eq:connectivity-supercritical-vertex-target} lies in
$\mathcal H$, because $K_{b,v}\subseteq H_b$. Every $v\in B_b$ therefore has an
edge to the connected set $\mathcal H$, proving that $G(n,p,m)$ is connected.
\end{proof}

Together, Theorems~\ref{thm:connectivity-lower} and~\ref{thm:connectivity-upper} prove
Theorem~\ref{thm:sharp-connectivity}. For later use, couple all retention probabilities
by retaining edge $e$ at parameter $p$ when $U_e\leq p$, where the $U_e$ are
independent uniform random variables. Connectivity is increasing under this coupling,
so Theorems~\ref{thm:spectral-gap} and~\ref{thm:sharp-hitting} may use the uniform estimate
\begin{equation}
\label{eq:connectivity-supercritical-uniform}
\sup_{p\in[p_c,1]}
\Prob_p\bigl(G(n,p,m)\text{ is disconnected}\bigr)
=
\Prob_{p_c}\bigl(G(n,p_c,m)\text{ is disconnected}\bigr)
\longrightarrow0.
\end{equation}

We next determine the components inside the critical window.

\section{Critical window of connectivity and the giant component}
\label{sec:critical-connectivity-window}

Throughout this section, assume \eqref{eq:critical-window-bandwidth-assumptions} and
\eqref{eq:critical-window-exact-centering}. The exact centering gives
\begin{equation}
\label{eq:critical-window-exact-isolation-probability}
\Prob(v\text{ is isolated})=(1-p)^{2m}=\frac{e^{-t}}n,
\end{equation}
and hence $\E N_{\mathrm{iso}}=e^{-t}$. Factorial moments yield the Poisson limit.
The main task is to show that all other vertices lie in one component. For each
vertex $v$, we construct two connected chains covering all but $o(m/\log n)$ of
its host-neighbours, without using any edge incident to $v$. Conditioning on these
chains reduces the problem to an isolation estimate. We then use the chains to
bound the diameter and a balanced partition to describe the block structure.

\subsection{Poisson limit for isolated vertices}
\label{subsubsec:critical-window-poisson}

We use the factorial-moment criterion for Poisson convergence
\cite{BarbourHolstJanson1992,JLR2000}.

\begin{proof}[Proof of the Poisson convergence in
Theorem~\ref{thm:critical-window-only-isolated-outside-giant}]
Let $I_v$ be the isolation indicator and fix $r\geq1$. The factorial moment is
$$
\E[(N_{\mathrm{iso}})_r]
=
\sum_{(v_1,\ldots,v_r)\in(\mathbb Z_n)^r_{\neq}}
\Prob(I_{v_1}=\cdots=I_{v_r}=1).
$$
For $S=\{v_1,\ldots,v_r\}$, write $e_H(S)$ for the number of host edges within $S$.
There are $2mr-e_H(S)$ distinct host edges incident to $S$, so
$$
\Prob(I_{v_1}=\cdots=I_{v_r}=1)
=(1-p)^{2mr-e_H(S)}
=n^{-r}e^{-rt}(1-p)^{-e_H(S)}.
$$
Because $0\leq e_H(S)\leq\binom r2$ and $p\to0$, this equals
$n^{-r}e^{-rt}(1+o(1))$ uniformly over the tuple. Therefore
$$
\E[(N_{\mathrm{iso}})_r]
=(n)_rn^{-r}e^{-rt}(1+o(1))\longrightarrow e^{-rt}.
$$
The limiting factorial moments are those of $\Poisson(e^{-t})$, proving
\eqref{eq:critical-window-poisson-conclusion}. We also have
$\E N_{\mathrm{iso}}=e^{-t}$ by \eqref{eq:critical-window-exact-isolation-probability},
and convergence in distribution gives
\begin{equation}
\label{eq:critical-window-no-isolated-limit}
\Prob(N_{\mathrm{iso}}=0)\longrightarrow\exp(-e^{-t}).
\end{equation}
\end{proof}

\subsection{Large components in blocks and their interfaces}
\label{subsubsec:critical-window-block-estimates}

A cyclic interval of length at most $m$ induces a complete host subgraph, so its
retained subgraph is an Erd\H{o}s--R\'enyi graph. We apply
Lemma~\ref{lem:band-block-large-component} to obtain a component containing a
$1-o(1)$ fraction of each block.

By exact centering and $\log n\ll m$, uniformly for $t$ in a fixed compact set,
\begin{equation}
\label{eq:critical-window-p-asymptotic}
p
=
1-\exp\left(-\frac{\log n+t}{2m}\right)
=
\frac{\log n+t}{2m}\bigl(1+o(1)\bigr).
\end{equation}

Thus a block of length comparable to $m$ has mean degree comparable to $\log n$.
The following corollary applies uniformly to families of at most $2n/m$ disjoint
intervals, including the balanced partition in Theorem~\ref{thm:critical-window-geometry}.

\begin{corollary}[A large component in each block of length between $m/2$ and $m$]
\label{lem:critical-window-large-block}
Assume \eqref{eq:critical-window-bandwidth-assumptions} and
\eqref{eq:critical-window-exact-centering}. Let $(\delta_n)$ be a deterministic sequence such
that
\begin{equation}
\label{eq:critical-window-delta-conditions}
\delta_n\longrightarrow0,
\qquad
\delta_n m\longrightarrow\infty,
\qquad
\log(1/\delta_n)=o(\log n).
\end{equation}
Let $I_1,\ldots,I_L$ be pairwise disjoint cyclic intervals satisfying
\begin{equation}
\label{eq:critical-window-comparable-block-lengths}
\frac m2\leq |I_j|\leq m,
\qquad
1\leq j\leq L,
\end{equation}
and $L\leq2n/m$. Then, with probability at least
\begin{equation}
\label{eq:critical-window-large-block-probability}
1-
\frac{4n}{m}
\exp\left(-\frac1{32}\delta_n m\log n\right),
\end{equation}
every $G[I_j]$ contains a connected component of size at least
$(1-\delta_n)|I_j|$.
\end{corollary}

\begin{proof}
For all sufficiently large $n$, \eqref{eq:critical-window-p-asymptotic} gives
$pm\geq3(\log n)/8$. Each $I_j$ induces a complete host subgraph, so
Lemma~\ref{lem:band-block-large-component} with $\beta=3/8$ and $\rho_n=\delta_n$
bounds the failure probability by
$$
2L\exp\left(-\frac1{32}\delta_nm\log n\right)
\leq
\frac{4n}{m}\exp\left(-\frac1{32}\delta_nm\log n\right),
$$
which proves \eqref{eq:critical-window-large-block-probability}.
\end{proof}

To join these components, we expose the edges between consecutive blocks. Their
indicators are still independent after conditioning on all within-block edges.
Combining this observation with Lemma~\ref{lem:adjacent-block-admissible-pairs}
gives the following estimate.

\begin{lemma}[Edges joining large sets in adjacent blocks]
\label{lem:critical-window-cross-edges}
Assume \eqref{eq:critical-window-bandwidth-assumptions} and
\eqref{eq:critical-window-exact-centering}. Fix $\delta\in(0,1/32)$. Let $I,I'$ be adjacent
cyclic intervals satisfying
\begin{equation}
\label{eq:critical-window-adjacent-block-lengths}
\frac m2\leq |I|,|I'|\leq m,
\end{equation}
and let $K\subseteq I$, $K'\subseteq I'$ be determined by the within-block edges and satisfy
\begin{equation}
\label{eq:critical-window-adjacent-large-set-sizes}
|K|\geq(1-\delta)|I|,
\qquad
|K'|\geq(1-\delta)|I'|.
\end{equation}
Let $M_H(K,K')$ be the number of admissible pairs defined in
\eqref{eq:connectivity-adjacent-block-pair-definition}.
Conditional on the within-block edges,
\begin{equation}
\label{eq:critical-window-cross-binomial-law}
e_G(K,K')\sim\Bin\bigl(M_H(K,K'),p\bigr),
\end{equation}
and, for all sufficiently large $n$,
\begin{equation}
\label{eq:critical-window-cross-pair-count}
\frac{m^2}{16}
\leq
M_H(K,K')
\leq
m^2.
\end{equation}
Consequently, there are constants $0<a<A<\infty$ and $c_0>0$, depending only on $t$ and
$\delta$, such that
\begin{equation}
\label{eq:critical-window-cross-edge-concentration}
\Prob\left(
 a m\log n\leq e_G(K,K')\leq A m\log n
 \;\middle|\;
 \text{within-block edges}
\right)
\geq
1-2e^{-c_0m\log n}.
\end{equation}
\end{lemma}

\begin{proof}
Lemma~\ref{lem:adjacent-block-admissible-pairs} proves
\eqref{eq:critical-window-cross-pair-count}.

The sets $K$ and $K'$ depend only on within-block indicators. Conditional on these
indicators, the $M_H(K,K')$ admissible edges between $I$ and $I'$ remain independent
Bernoulli$(p)$ variables. This proves \eqref{eq:critical-window-cross-binomial-law}.

Set $\mu_{K,K'}:=pM_H(K,K')$. \eqref{eq:critical-window-p-asymptotic} gives
$(\log n)/(3m)\leq p\leq2(\log n)/(3m)$ for all sufficiently large $n$. Hence
\eqref{eq:critical-window-cross-pair-count} implies
$$
\frac1{48}m\log n\leq\mu_{K,K'}\leq\frac23m\log n.
$$
Conditional on the within-block edges, Lemma~\ref{lem:chernoff} with relative deviation
$1/2$ bounds the probability that
$e_G(K,K')\notin[\mu_{K,K'}/2,3\mu_{K,K'}/2]$ by
$2\exp(-\mu_{K,K'}/12)$. Thus
\eqref{eq:critical-window-cross-edge-concentration} holds with
$a=1/96$, $A=1$, and $c_0=1/576$.
\end{proof}

\subsection{Leave-one-vertex-out connected chains}
\label{subsubsec:critical-window-shifted-chains}

The preceding estimates produce a connected set meeting every block. To compare
$\{v\notin C_1\}$ with isolation, however, we need a construction independent of
the edges incident to $v$. We build two chains separately for each vertex, using
only edges with both endpoints in $\mathbb Z_n\setminus\{v\}$.

The uncovered fraction must be small enough that $\delta_n\log n\to0$; this makes
its contribution to the isolation probability a factor $1+o(1)$. At the same time,
$\delta_nm\to\infty$ and $\log(1/\delta_n)=o(\log n)$ ensure concentration
across all blocks and chains. Choose
\begin{equation}
\label{eq:critical-window-adaptive-h}
a_n:=\frac{m}{\log n},
\qquad
h_n:=\min\{a_n^{1/2},\log\log n\},
\end{equation}
and
\begin{equation}
\label{eq:critical-window-adaptive-delta}
\delta_n:=\frac{1}{h_n\log n}.
\end{equation}
Then
\begin{equation}
\label{eq:critical-window-adaptive-delta-properties}
\delta_n\longrightarrow0,
\qquad
\delta_nm=\frac{a_n}{h_n}\longrightarrow\infty,
\qquad
\delta_n\log n=\frac1{h_n}\longrightarrow0,
\qquad
\log(1/\delta_n)=o(\log n).
\end{equation}
To check these properties, note that $h_n\to\infty$ and $h_n\leq\log\log n$.
If $h_n=a_n^{1/2}$, then $a_n/h_n=a_n^{1/2}\to\infty$. Otherwise,
$h_n=\log\log n$, $a_n\geq(\log\log n)^2$, and
$a_n/h_n\geq\log\log n\to\infty$. Finally,
$\log(1/\delta_n)=\log\log n+\log h_n=O(\log\log n)=o(\log n)$.

We first build a connected seed of diameter $O(\log n/\log\log n)$ in an interval
of length $\lfloor m/2\rfloor$. We then grow the seed through consecutive intervals.
All pairs between consecutive intervals are admissible, so each new block can
attach directly to the preceding one. The resulting chain contains more than half
the vertices and uses no edge incident to the prescribed vertex.

\begin{lemma}[A short connected seed in an interval of length $\lfloor m/2\rfloor$]
\label{lem:critical-window-short-seed}
Assume \eqref{eq:critical-window-bandwidth-assumptions} and
\eqref{eq:critical-window-exact-centering}, and set
\begin{equation}
\label{eq:critical-window-half-block-length}
\ell:=\left\lfloor\frac m2\right\rfloor.
\end{equation}
There is a constant $C_0>0$, depending only on $t$, such that the following holds. For every
fixed $D>0$ and every cyclic interval $I\subseteq\mathbb Z_n$ of length $\ell$, for all
sufficiently large $n$ the induced graph $G[I]$ contains, with probability at least
\begin{equation}
\label{eq:critical-window-short-seed-probability}
1-n^{-D},
\end{equation}
a connected set $S_I$ satisfying
\begin{equation}
\label{eq:critical-window-short-seed-size}
|S_I|\geq \frac{\ell}{8}
\end{equation}
and
\begin{equation}
\label{eq:critical-window-short-seed-diameter}
\diam\bigl(G[S_I]\bigr)
\leq
C_0\frac{\log n}{\log\log n}.
\end{equation}
The set $S_I$ may be chosen as a measurable function of the edges with both endpoints in $I$.
\end{lemma}

\begin{proof}
We find a vertex with many neighbours, then run breadth-first search from those
neighbours. Since $|I|=\ell\leq m$, the induced graph $G[I]$ is an Erd\H{o}s--R\'enyi graph
$G(\ell,p)$. Write $\mu_{\mathrm{blk}}:=p\ell$.
By \eqref{eq:critical-window-p-asymptotic} and $\ell/m\to1/2$,
\begin{equation}
\label{eq:critical-window-half-block-mean-degree-asymptotic}
\mu_{\mathrm{blk}}=\left(\frac14+o(1)\right)\log n.
\end{equation}
In particular, $\mu_{\mathrm{blk}}\to\infty$ and $\ell/\mu_{\mathrm{blk}}\to\infty$.

\medskip\noindent\emph{Choosing a root.}
Partition $I=A\sqcup B$ into consecutive sets with $|A|=\lfloor\ell/4\rfloor$
and $|B|=\ell-|A|$. For $x\in A$, let $D_B(x)$ count its retained edges to $B$.
The variables $\{D_B(x):x\in A\}$ are independent, with
$D_B(x)\sim\Bin(|B|,p)$ and
$\E[D_B(x)]=|B|p\geq2\mu_{\mathrm{blk}}/3$ for all sufficiently large $n$.
By Lemma~\ref{lem:chernoff},
$\Prob(D_B(x)<\mu_{\mathrm{blk}}/2)\leq e^{-c\mu_{\mathrm{blk}}}$ for an absolute
constant $c>0$. Independence therefore gives a constant $c_1>0$ such that
\begin{equation}
\label{eq:critical-window-seed-no-root-bound}
\Prob\left(D_B(x)<\frac{\mu_{\mathrm{blk}}}{2}\text{ for every }x\in A\right)
\leq
\exp(-c_1\mu_{\mathrm{blk}}\ell).
\end{equation}
On the complementary event, let $x_\star$ be the first vertex of $A$, in the order
inherited from $I$, for which $D_B(x_\star)\geq\mu_{\mathrm{blk}}/2$.
This choice uses only edges between $A$ and $B$. Conditional on $x_\star$ and its
neighbours in $B$, the graph inside $B$ is still an independent $G(|B|,p)$.

\medskip\noindent\emph{Growing the exploration.}
Run breadth-first search inside $B$, starting from
$F_1:=\{y\in B:\{x_\star,y\}\in E(G)\}$; by construction,
$|F_1|\geq\mu_{\mathrm{blk}}/2$.
Let $U_k:=F_1\cup\cdots\cup F_k$. While $|U_k|<\ell/8$, expose the edges from
$F_k$ to $B\setminus U_k$ and let $F_{k+1}$ consist of the newly reached vertices.
These edges have not previously been exposed. Conditional on the exploration,
the indicators defining $F_{k+1}$ are independent, so, with $f:=|F_k|$,
\begin{equation}
\label{eq:critical-window-seed-frontier-conditional-law}
|F_{k+1}|
\,\bigm|\,
(F_1,\ldots,F_k)
\sim
\Bin\left(
|B\setminus U_k|,1-(1-p)^f
\right).
\end{equation}
If $pf\leq1$, then
$1-(1-p)^f\geq1-e^{-pf}\geq pf/2$. While $|U_k|<\ell/8$, the set
$B\setminus U_k$ contains at least $5\ell/8$ vertices. Consequently,
$\E[|F_{k+1}|\mid F_1,\ldots,F_k]\geq\mu_{\mathrm{blk}}f/4$.
Until the first unsuccessful generation, each frontier grows by a factor at least
$\mu_{\mathrm{blk}}/8>1$. Thus $f\geq|F_1|\geq\mu_{\mathrm{blk}}/2$, and
Lemma~\ref{lem:chernoff} gives
\begin{equation}
\label{eq:critical-window-seed-growth-failure}
\Prob\left(
|F_{k+1}|<\frac{\mu_{\mathrm{blk}}f}{8}
\,\middle|\,
F_1,\ldots,F_k
\right)
\leq
\exp\left(-\frac{\mu_{\mathrm{blk}}f}{32}\right)
\leq
\exp\left(-\frac{\mu_{\mathrm{blk}}^2}{64}\right).
\end{equation}
When $pf>1$, we have $1-(1-p)^f\geq1-e^{-1}$, so
Lemma~\ref{lem:chernoff} gives
\begin{equation}
\label{eq:critical-window-seed-large-frontier-failure}
\Prob\left(
|F_{k+1}|<\frac\ell8
\,\middle|\,
F_1,\ldots,F_k
\right)
\leq e^{-c_2\ell}
\end{equation}
for an absolute constant $c_2>0$.

Set
\begin{equation}
\label{eq:critical-window-seed-depth}
R_n:=
2+
\left\lceil
\frac{\log(\ell/\mu_{\mathrm{blk}})}{\log(\mu_{\mathrm{blk}}/8)}
\right\rceil.
\end{equation}
For all sufficiently large $n$, $\mu_{\mathrm{blk}}/8>1$. On the complement of the two
failure events bounded in \eqref{eq:critical-window-seed-growth-failure}
and~\eqref{eq:critical-window-seed-large-frontier-failure}, the exploration reaches at least
$\ell/8$ vertices within $R_n$ generations. By
\eqref{eq:critical-window-half-block-mean-degree-asymptotic},
\begin{equation}
\label{eq:critical-window-seed-depth-bound}
R_n=O\left(\frac{\log n}{\log\log n}\right).
\end{equation}
Stop when the exploration first contains at least $\ell/8$ vertices. Together with $x_\star$,
the discovered vertices form a connected set of radius at most $R_n+1$, and hence
diameter at most $2R_n+2$.

\medskip\noindent\emph{Bounding the failure probability.}
A union bound over at most $R_n$ generations gives
\begin{align}
&\Prob\left(
 G[I]\text{ contains no set satisfying
 \eqref{eq:critical-window-short-seed-size} and
 \eqref{eq:critical-window-short-seed-diameter}}
\right)
\notag\\
&\qquad\leq
\exp(-c_1\mu_{\mathrm{blk}}\ell)
+
R_n\exp\left(-\frac{\mu_{\mathrm{blk}}^2}{64}\right)
+
R_ne^{-c_2\ell}.
\label{eq:critical-window-short-seed-total-failure}
\end{align}
By \eqref{eq:critical-window-half-block-mean-degree-asymptotic},
$\mu_{\mathrm{blk}}\ell\asymp m\log n$ and
$\mu_{\mathrm{blk}}^2\asymp(\log n)^2$; moreover,
$R_n=O(\log n/\log\log n)$ by
\eqref{eq:critical-window-seed-depth-bound}.
The first two terms in \eqref{eq:critical-window-short-seed-total-failure} are
$o(n^{-D})$ for every fixed $D>0$. The third is also $o(n^{-D})$, since
$\ell\asymp m$ and $\log n\ll m$. Their sum is therefore at most $n^{-D}$ for
all sufficiently large $n$. The fixed vertex order and the exposed edges determine
every choice, so $S_I$ is measurable with respect to the edges inside $I$.
\end{proof}

We now extend the seed through consecutive blocks, leaving all edges incident to
the prescribed vertex unexposed.

\begin{lemma}[Shifted connected chains with controlled diameter]
\label{lem:critical-window-shifted-short-chain}
Assume the hypotheses of Theorem~\ref{thm:critical-window-only-isolated-outside-giant}, and
let $\ell$ be given by \eqref{eq:critical-window-half-block-length}. Set
\begin{equation}
\label{eq:critical-window-half-block-count}
L:=\left\lfloor\frac{n-1}{\ell}\right\rfloor,
\qquad
r:=n-L\ell.
\end{equation}
Then
\begin{equation}
\label{eq:critical-window-half-block-remainder}
1\leq r\leq\ell,
\qquad
L\geq4.
\end{equation}
For every $v\in\mathbb Z_n$ and every $\sigma\in\{-,+\}$, there is an event
$\mathcal E_v^\sigma$ and, whenever $\mathcal E_v^\sigma$ occurs, a connected set
$\mathcal K_v^\sigma\subseteq\mathbb Z_n\setminus\{v\}$ with the following properties.

Both $\mathcal E_v^\sigma$ and $\mathcal K_v^\sigma$ are measurable with respect to the
indicators $\xi_{\{x,y\}}$ with $x,y\in\mathbb Z_n\setminus\{v\}$. Moreover,
\begin{equation}
\label{eq:critical-window-shifted-chain-size}
0\leq n-|\mathcal K_v^\sigma|
\leq\frac{15}{8}\ell+\delta_n n,
\end{equation}
and, for all sufficiently large $n$,
\begin{equation}
\label{eq:critical-window-shifted-chain-majority}
|\mathcal K_v^\sigma|>\frac n2.
\end{equation}
There is a constant $C_{\mathrm{ch}}>0$, depending only on $t$, such that
\begin{equation}
\label{eq:critical-window-shifted-chain-diameter}
\diam\bigl(G[\mathcal K_v^\sigma]\bigr)
\leq
C_{\mathrm{ch}}\left(
\frac nm+\frac{\log n}{\log\log n}
\right).
\end{equation}
In addition, $\mathcal K_v^-$ contains a set $Q_v^-$ of left neighbours of $v$, and
$\mathcal K_v^+$ contains a set $Q_v^+$ of right neighbours of $v$, such that
\begin{equation}
\label{eq:critical-window-shifted-chain-boundary-size}
|Q_v^-|\geq2(1-\delta_n)\ell,
\qquad
|Q_v^+|\geq2(1-\delta_n)\ell.
\end{equation}
Here the left neighbours of $v$ are the vertices $v-m,\ldots,v-1$, and the right neighbours
are $v+1,\ldots,v+m$, with arithmetic modulo $n$.
The deficit parameter $\delta_n$ is defined in \eqref{eq:critical-window-adaptive-delta}.
Moreover, for every fixed $D>0$,
\begin{equation}
\label{eq:critical-window-shifted-chain-failure}
\sup_{v\in\mathbb Z_n}\sup_{\sigma\in\{-,+\}}
\Prob\bigl((\mathcal E_v^\sigma)^c\bigr)
\leq n^{-D}
\end{equation}
for all sufficiently large $n$.
\end{lemma}

\begin{proof}
We arrange the blocks so that the last two lie next to the prescribed vertex.
Starting from a seed in the first block, we attach almost every vertex of each
subsequent block to the preceding one.

\medskip\noindent\emph{The block arrangement.}
To verify \eqref{eq:critical-window-half-block-remainder}, write
$n-1=L\ell+s$ with $0\leq s<\ell$. Then $r=s+1$, so $1\leq r\leq\ell$.
Furthermore, $2\ell\leq m<n/2$ implies $4\ell<n$, and therefore $n-1\geq4\ell$ because
$n$ and $\ell$ are integers. Hence $L\geq4$.

For $1\leq j\leq L$, define the consecutive blocks of length $\ell$ by
\begin{equation}
\label{eq:critical-window-left-half-blocks}
B_{v,j}^-:=
\{v-(L-j+1)\ell,\ldots,v-(L-j)\ell-1\},
\qquad
1\leq j\leq L,
\end{equation}
with arithmetic modulo $n$. Their union is the interval of $L\ell$ vertices immediately to the
left of $v$, and the complement has size $r$ and contains $v$. In particular,
\begin{equation}
\label{eq:critical-window-left-two-boundary-blocks}
B_{v,L-1}^-\cup B_{v,L}^-
=
\{v-2\ell,\ldots,v-1\}.
\end{equation}
For the right-hand construction, use
\begin{equation}
\label{eq:critical-window-right-half-blocks}
B_{v,j}^+:=
\{v+(L-j)\ell+1,\ldots,v+(L-j+1)\ell\},
\qquad
1\leq j\leq L.
\end{equation}
Then
\begin{equation}
\label{eq:critical-window-right-two-boundary-blocks}
B_{v,L-1}^+\cup B_{v,L}^+
=
\{v+1,\ldots,v+2\ell\}.
\end{equation}
Reflection $x\mapsto2v-x$ fixes $v$ and exchanges the two block families while
preserving the law of the graph. Thus the same argument applies to both choices of
$\sigma$. Since $2\ell\leq m$, every pair between consecutive blocks
$B_{v,j}^\sigma,B_{v,j+1}^\sigma$ is admissible.

\medskip\noindent\emph{Growing the chain.}
Apply Lemma~\ref{lem:critical-window-short-seed} to $B_{v,1}^\sigma$ and let
$S_{v,1}^\sigma$ be the resulting connected set. Recursively, for $1\leq j<L$, define
\begin{equation}
\label{eq:critical-window-shifted-chain-recursion}
S_{v,j+1}^\sigma
:=
\left\{
 x\in B_{v,j+1}^\sigma:
 e_G\bigl(\{x\},S_{v,j}^\sigma\bigr)>0
\right\}.
\end{equation}
Every vertex of $S_{v,j+1}^\sigma$ has an edge to $S_{v,j}^\sigma$. Consequently,
\begin{equation}
\label{eq:critical-window-shifted-chain-union}
\mathcal K_v^\sigma
:=
\bigcup_{j=1}^L S_{v,j}^\sigma
\end{equation}
is connected whenever all the sets in the construction have the required sizes.

Conditional on $S_{v,1}^\sigma,\ldots,S_{v,j}^\sigma$, the edges from
$S_{v,j}^\sigma$ to $B_{v,j+1}^\sigma$ are still unexposed. Thus the number of
vertices missed in the next block has distribution
\begin{equation}
\label{eq:critical-window-shifted-chain-deficit-law}
\ell-|S_{v,j+1}^\sigma|
\sim
\Bin\left(
\ell,(1-p)^{|S_{v,j}^\sigma|}
\right).
\end{equation}
On the event in Lemma~\ref{lem:critical-window-short-seed},
$|S_{v,1}^\sigma|\geq\ell/8$. Hence, whenever
$|S_{v,j}^\sigma|\geq\ell/8$,
\begin{align}
(1-p)^{|S_{v,j}^\sigma|}
&\leq
(1-p)^{\ell/8}
\notag\\
&=
\exp\left(-\frac{\ell(\log n+t)}{16m}\right)
\notag\\
&\leq
n^{-1/40}
\label{eq:critical-window-shifted-chain-missing-probability}
\end{align}
for all sufficiently large $n$. We bound this deficit uniformly using
$\log(1/\delta_n)=o(\log n)$ and $\delta_n\ell\to\infty$, from
\eqref{eq:critical-window-adaptive-delta-properties}. For
$X\sim\Bin(\ell,q)$ with $q\leq n^{-1/40}$, put
$k_n:=\lfloor\delta_n\ell\rfloor+1$. For all sufficiently large $n$,
$e q/\delta_n<1$ and
$$
\log\left(\frac{e q}{\delta_n}\right)
\leq
1-\frac{\log n}{40}+\log\left(\frac1{\delta_n}\right)
\leq
-\frac{\log n}{80}
$$
A union bound over $k_n$-element subsets and
$\binom{\ell}{k_n}\leq(e\ell/k_n)^{k_n}$ give
$$
\Prob(X>\delta_n\ell)
\leq
\binom{\ell}{k_n}q^{k_n}
\leq
\left(\frac{e q}{\delta_n}\right)^{k_n}
\leq
\left(\frac{e q}{\delta_n}\right)^{\delta_n\ell}
\leq
\exp\left(-\frac{\delta_n\ell\log n}{80}\right)
$$
Applying this bound to \eqref{eq:critical-window-shifted-chain-deficit-law} yields
\begin{equation}
\label{eq:critical-window-shifted-chain-deficit-tail}
\Prob\left(
\ell-|S_{v,j+1}^\sigma|>\delta_n\ell
\,\middle|\,
S_{v,1}^\sigma,\ldots,S_{v,j}^\sigma
\right)
\leq
\exp\left(-\frac{\delta_n\ell\log n}{80}\right)
\end{equation}
whenever $|S_{v,j}^\sigma|\geq\ell/8$. Thus, except for the failures in
\eqref{eq:critical-window-shifted-chain-deficit-tail}, every set after the seed has
size at least $(1-\delta_n)\ell$.

Let $\mathcal E_v^\sigma$ be the event that the seed in $B_{v,1}^\sigma$ satisfies
\eqref{eq:critical-window-short-seed-size} and
\eqref{eq:critical-window-short-seed-diameter}, and every later deficit is at most
$\delta_n\ell$. Both $\mathcal E_v^\sigma$ and the constructed sets use only
indicators with endpoints in
$B_{v,1}^\sigma,\ldots,B_{v,L}^\sigma\subseteq\mathbb Z_n\setminus\{v\}$.

\medskip\noindent\emph{The failure probability.}
Since
$$
\delta_n\ell\log n
\asymp
\frac{m}{h_n}
=
\frac{a_n}{h_n}\log n,
\qquad
\frac{a_n}{h_n}\longrightarrow\infty,
$$
the deficit bound is smaller than every fixed inverse power of $n$. More precisely,
fix $D>0$ and apply Lemma~\ref{lem:critical-window-short-seed} with $D+1$.
For all sufficiently large $n$, the bound in
\eqref{eq:critical-window-shifted-chain-deficit-tail} is at most $n^{-(D+2)}$.
Since $n/\ell\leq n$, a union bound over the remaining blocks gives
\begin{align}
\Prob\bigl((\mathcal E_v^\sigma)^c\bigr)
&\leq
n^{-(D+1)}
+
\frac n\ell
\exp\left(-\frac{\delta_n\ell\log n}{80}\right)
\notag\\
&\leq
2n^{-(D+1)}
\leq
n^{-D},
\label{eq:critical-window-shifted-chain-failure-calculation}
\end{align}
uniformly in $v$ and $\sigma$, establishing
\eqref{eq:critical-window-shifted-chain-failure}.

\medskip\noindent\emph{Size, diameter, and boundary coverage.}
On $\mathcal E_v^\sigma$,
\begin{align}
|\mathcal K_v^\sigma|
&\geq
\frac\ell8+(L-1)(1-\delta_n)\ell
\notag\\
&=
L\ell-\frac{7\ell}{8}-\delta_n(L-1)\ell
\notag\\
&\geq
n-r-\frac{7\ell}{8}-\delta_n n
\notag\\
&\geq
n-\frac{15\ell}{8}-\delta_n n.
\label{eq:critical-window-shifted-chain-size-proof}
\end{align}
This gives \eqref{eq:critical-window-shifted-chain-size}. Since
$\ell\leq m/2<n/4$, \eqref{eq:critical-window-shifted-chain-size-proof} also gives
\begin{equation}
\label{eq:critical-window-shifted-chain-uniform-majority-proof}
|\mathcal K_v^\sigma|
>
\left(\frac{17}{32}-\delta_n\right)n.
\end{equation}
Since $\delta_n\to0$, this exceeds $n/2$ for all sufficiently large $n$, proving
\eqref{eq:critical-window-shifted-chain-majority} throughout the bandwidth range.
Each vertex of $S_{v,j}^\sigma$ reaches $S_{v,1}^\sigma$ by a path of at most
$j-1$ edges inside $\mathcal K_v^\sigma$. Hence
\begin{align}
\diam\bigl(G[\mathcal K_v^\sigma]\bigr)
&\leq
2(L-1)+\diam\bigl(G[S_{v,1}^\sigma]\bigr)
\notag\\
&\leq
C_{\mathrm{ch}}\left(
\frac nm+\frac{\log n}{\log\log n}
\right),
\label{eq:critical-window-shifted-chain-diameter-proof}
\end{align}
where we used $\ell\asymp m$ and
Lemma~\ref{lem:critical-window-short-seed}. Set
\begin{equation}
\label{eq:critical-window-shifted-chain-boundary-sets}
Q_v^\sigma:=S_{v,L-1}^\sigma\cup S_{v,L}^\sigma.
\end{equation}
By \eqref{eq:critical-window-left-two-boundary-blocks} and
\eqref{eq:critical-window-right-two-boundary-blocks}, $Q_v^-$ consists of left
neighbours of $v$ and $Q_v^+$ of right neighbours. The size estimates give
\eqref{eq:critical-window-shifted-chain-boundary-size}. All blocks avoid $v$, so the
construction is measurable with respect to indicators with both endpoints in
$\mathbb Z_n\setminus\{v\}$.
\end{proof}

\subsection{Every non-isolated vertex belongs to the largest component}
\label{subsubsec:critical-window-component-completion}

Each successful chain contains more than half the vertices, so the two chains for
$v$ intersect and lie in $C_1$. Their terminal sets cover all but $o(m/\log n)$
admissible neighbours of $v$. Conditional on the indicators with both endpoints in
$\mathbb Z_n\setminus\{v\}$, the incident edges from $v$ remain independent.
We can therefore bound the probability that $v$ lies outside $C_1$ by an isolation
probability with only a small number of missing trials.

\begin{proof}[Proof of Theorem~\ref{thm:critical-window-only-isolated-outside-giant}]
Subsection~\ref{subsubsec:critical-window-poisson} gives
$N_{\mathrm{iso}}\Longrightarrow\Poisson(e^{-t})$. It remains to show that every
non-isolated vertex belongs to $C_1$.

\medskip\noindent\emph{Comparing exclusion with isolation.}
Fix $v\in\mathbb Z_n$. On $\mathcal E_v^-\cap\mathcal E_v^+$,
Lemma~\ref{lem:critical-window-shifted-short-chain} gives connected sets
$\mathcal K_v^-$ and $\mathcal K_v^+$, each of size greater than $n/2$.
Since $|\mathcal K_v^-|+|\mathcal K_v^+|>n$, they intersect, and their union lies
in the unique component with more than $n/2$ vertices, namely $C_1$.
The sets $Q_v^-$ and $Q_v^+$ lie in $\{v-2\ell,\ldots,v-1\}$ and
$\{v+1,\ldots,v+2\ell\}$, respectively. These intervals are disjoint because
$4\ell\leq2m<n$, and consist of admissible neighbours because $2\ell\leq m$.
By \eqref{eq:critical-window-shifted-chain-boundary-size},
\begin{equation}
\label{eq:critical-window-short-chain-covered-neighbours}
|Q_v^-|+|Q_v^+|
\geq
4(1-\delta_n)\ell.
\end{equation}
Since $2\ell\geq m-1$,
\begin{equation}
\label{eq:critical-window-short-chain-uncovered-neighbours}
2m-4(1-\delta_n)\ell
\leq
2+2\delta_n m.
\end{equation}

Let $\mathcal G_v$ be the $\sigma$-field generated by the indicators
$\xi_{\{x,y\}}$ with $x,y\in\mathbb Z_n\setminus\{v\}$.
The chains, boundary sets, and events $\mathcal E_v^\pm$ are
$\mathcal G_v$-measurable. If $v\notin C_1$ on
$\mathcal E_v^-\cap\mathcal E_v^+$, no edge from $v$ to $Q_v^-\cup Q_v^+$ is
retained. Conditional on $\mathcal G_v$, these at least $4(1-\delta_n)\ell$
indicators are independent Bernoulli$(p)$ variables. Conditioning and then averaging gives
\begin{align}
&\Prob\left(
 v\notin C_1,
 \mathcal E_v^-\cap\mathcal E_v^+
\right)
\notag\\
&\qquad\leq
(1-p)^{4(1-\delta_n)\ell}.
\label{eq:critical-window-short-chain-vertex-outside}
\end{align}
Apply \eqref{eq:critical-window-shifted-chain-failure} with $D=4$ to each of the two
chains. Their combined failure probability is at most $2n^{-4}=o(n^{-3})$, uniformly in $v$.
Together with \eqref{eq:critical-window-short-chain-vertex-outside}, this gives
\begin{equation}
\label{eq:critical-window-short-chain-vertex-outside-total}
\Prob(v\notin C_1)
\leq
(1-p)^{4(1-\delta_n)\ell}+o(n^{-3}).
\end{equation}

\medskip\noindent\emph{Counting vertices outside the largest component.}
Set $Y:=n-|C_1|$. By exact centering and
\eqref{eq:critical-window-short-chain-uncovered-neighbours},
\begin{align}
n(1-p)^{4(1-\delta_n)\ell}
&=
n(1-p)^{2m}
(1-p)^{-\{2m-4(1-\delta_n)\ell\}}
\notag\\
&\leq
e^{-t}
\exp\left(
\left(\delta_n+\frac1m\right)(\log n+t)
\right)
\notag\\
&=
e^{-t}+o(1),
\label{eq:critical-window-short-chain-outside-expectation-bound}
\end{align}
where we used
\eqref{eq:critical-window-adaptive-delta-properties} and $\log n\ll m$.
Summing \eqref{eq:critical-window-short-chain-vertex-outside-total} over $v$ yields
\begin{equation}
\label{eq:critical-window-outside-expectation}
\E[Y]\leq e^{-t}+o(1).
\end{equation}

Fix a vertex $v_0$. On $\mathcal E_{v_0}^-$, a component contains more than
$n/2$ vertices. Thus
\begin{equation}
\label{eq:critical-window-largest-component-nontrivial}
\Prob(|C_1|=1)
\leq
\Prob\bigl((\mathcal E_{v_0}^-)^c\bigr)
=o(1).
\end{equation}
Let $N_{\mathrm{iso}}^{\mathrm{out}}$ count the isolated vertices outside $C_1$.
An isolated vertex lies in $C_1$ precisely when $|C_1|=1$, so
$N_{\mathrm{iso}}^{\mathrm{out}}
=N_{\mathrm{iso}}-\mathbf 1_{\{|C_1|=1\}}$.
Since $\E[N_{\mathrm{iso}}]=e^{-t}$, \eqref{eq:critical-window-largest-component-nontrivial}
gives
\begin{equation}
\label{eq:critical-window-isolated-outside-expectation}
\E[N_{\mathrm{iso}}^{\mathrm{out}}]
=
e^{-t}+o(1).
\end{equation}
The difference $Y_{\mathrm{noniso}}:=Y-N_{\mathrm{iso}}^{\mathrm{out}}$ counts
non-isolated vertices outside $C_1$. By \eqref{eq:critical-window-outside-expectation}
and \eqref{eq:critical-window-isolated-outside-expectation},
$\limsup_{n\to\infty}\E[Y_{\mathrm{noniso}}]\leq0$. Since this count is
nonnegative, $\E[Y_{\mathrm{noniso}}]\to0$, and Markov's inequality gives
$\Prob(Y_{\mathrm{noniso}}>0)\to0$. Together with
\eqref{eq:critical-window-largest-component-nontrivial}, this proves
\eqref{eq:critical-window-largest-component-conclusion}.

On the event in \eqref{eq:critical-window-largest-component-conclusion}, the graph is
connected if and only if $N_{\mathrm{iso}}=0$. \eqref{eq:critical-window-connectivity-limit}
therefore follows from \eqref{eq:critical-window-no-isolated-limit}. On the same event,
every component other than $C_1$ is an isolated vertex, which proves
\eqref{eq:critical-window-second-component}.
\end{proof}

Having proved $|C_1|=n-N_{\mathrm{iso}}$ with high probability, we turn to distances
within $C_1$.

\subsection{Block-chain structure and diameter of the largest component}
\label{subsubsec:critical-window-geometry}

The diameter of $C_1$ is bounded below both by the restriction on edge lengths and by the maximum degree. We first show that the maximum degree is $O(\log n)$ with high probability.

\begin{lemma}[Uniform maximum-degree bound in the critical window of connectivity]
\label{lem:critical-window-max-degree}
Under the assumptions of Theorem~\ref{thm:critical-window-only-isolated-outside-giant},
\begin{equation}
\label{eq:critical-window-uniform-degree-bound}
\max_{v\in\mathbb Z_n}d_G(v)\leq10\log n
\qquad\text{with high probability.}
\end{equation}
\end{lemma}

\begin{proof}
For every $v\in\mathbb Z_n$,
\begin{equation}
\label{eq:critical-window-degree-binomial}
d_G(v)\sim\Bin(2m,p),
\qquad
\mu_n:=\E[d_G(v)]=2mp\leq\log n+t.
\end{equation}
Indeed, with $x_n:=(\log n+t)/(2m)$, exact centering gives $p=1-e^{-x_n}$.
The bound $\mu_n\leq\log n+t$ follows from $1-e^{-x}\leq x$, and in particular
$\mu_n\leq2\log n$ for all sufficiently large $n$.

For $Z\sim\Bin(M,\rho)$ and an integer $k\geq1$, a union bound over the $k$-element
subsets of successful trials gives
$\Prob(Z\geq k)\leq\binom Mk\rho^k\leq(eM\rho/k)^k$. Apply this estimate with
$M=2m$, $\rho=p$, and $k=\lceil10\log n\rceil$. Since $\mu_n\leq2\log n$,
$$
\Prob\bigl(d_G(v)\geq10\log n\bigr)
\leq
\left(\frac e5\right)^{10\log n}.
$$
The inequality $10\log(5/e)>2$ and a union bound over the $n$ vertices prove
\eqref{eq:critical-window-uniform-degree-bound}.
\end{proof}

The edge-length restriction gives a lower bound $\Omega(n/m)$, while maximum degree
$O(\log n)$ gives $\Omega(\log n/\log\log n)$ by volume growth. The shifted
chains in Lemma~\ref{lem:critical-window-shifted-short-chain} provide a matching upper bound.

We prove Theorem~\ref{thm:critical-window-geometry} in three steps. First we construct
the block sets $H_b$ using the balanced partition and deficit sequence $(\eta_n)$
from the theorem. We then prove the two lower bounds and use the shifted chains
to obtain the upper bound.

In particular, if $m\asymp n$, then $n/m=\Theta(1)$ and
\eqref{eq:critical-window-giant-diameter-two-sided} reduces to
\begin{equation}
\label{eq:critical-window-dense-regime-diameter}
\diam(C_1)
=
\Theta_{\Prob}\left(\frac{\log n}{\log\log n}\right).
\end{equation}

\begin{proof}[Proof of Theorem~\ref{thm:critical-window-geometry}]
\emph{Step 1: the block structure.}
Lemma~\ref{lem:balanced-cyclic-partition} gives the block-length bounds in
\eqref{eq:critical-window-balanced-block-lengths}.

By Theorem~\ref{thm:critical-window-only-isolated-outside-giant}, every vertex outside
$C_1$ is isolated with high probability. The Poisson limit gives
$N_{\mathrm{iso}}=O_{\Prob}(1)$. Since $|C_1|>1$ with high probability,
\eqref{eq:critical-window-giant-complement-identity} follows.

Since $\eta_n\to0$, a component of $G[B_b]$ with at least
$(1-\eta_n)\ell_b$ vertices is unique for all sufficiently large $n$. Denote it
by $H_b$ when it exists, and set $H_b=\varnothing$ otherwise.
Apply Corollary~\ref{lem:critical-window-large-block} to the balanced partition in
\eqref{eq:critical-window-balanced-block-partition}. The event
$\mathcal E_{\mathrm{good}}:=\bigcap_{b=1}^q
\{|H_b|\geq(1-\eta_n)\ell_b\}$ has probability tending to one. Indeed,
$n/m>2$ gives $q\leq n/m+1<2n/m$, and the failure probability is at most
$$
\frac{4n}{m}\exp\left(-\frac{\eta_nm\log n}{32}\right)=o(1);
$$
indeed, $\eta_nm\to\infty$ whereas $\log(4n/m)=O(\log n)$.

Fix $\delta_\star:=1/64$. Since $\eta_n\to0$, on $\mathcal E_{\mathrm{good}}$
we have $|H_b|\geq(1-\delta_\star)\ell_b$ for all sufficiently large $n$.
Conditional on the within-block edges, Lemma~\ref{lem:critical-window-cross-edges}
applies to each cyclic pair $H_b,H_{b+1}$, with failure probability exponentially
small in $m\log n$. A conditional union bound over the $q$ interfaces gives
\begin{equation}
\label{eq:critical-window-balanced-interface-failure-check}
\mathbf 1_{\mathcal E_{\mathrm{good}}}
\Prob\left(
\text{some interface violates
\eqref{eq:critical-window-backbone-interface-bounds}}
\,\middle|\,
\text{all within-block edges}
\right)
\leq
2q e^{-c_0m\log n}
=o(1).
\end{equation}
Here $q=O(n/m)$ and $m\to\infty$. Taking expectations and adding
$\Prob(\mathcal E_{\mathrm{good}}^c)=o(1)$ proves
\eqref{eq:critical-window-backbone-interface-bounds} with high probability.

On the intersection of $\mathcal E_{\mathrm{good}}$ and the event that every interface
satisfies \eqref{eq:critical-window-backbone-interface-bounds}, the set
$\mathcal H:=\bigcup_{b=1}^qH_b$ is connected and
\begin{equation}
\label{eq:critical-window-balanced-backbone-size}
|\mathcal H|
\geq
(1-\eta_n)\sum_{b=1}^q\ell_b
=
(1-\eta_n)n
>
\frac n2
\end{equation}
for all sufficiently large $n$. Hence $\mathcal H\subseteq C_1$, proving
\eqref{eq:critical-window-backbone-sets}--
\eqref{eq:critical-window-backbone-interface-bounds}. Every block not entirely contained in
$C_1$ contains at least one vertex of $\mathbb Z_n\setminus C_1$. Since the blocks are disjoint,
the number of such blocks is at most $|\mathbb Z_n\setminus C_1|=O_{\Prob}(1)$.

\medskip\noindent\emph{Step 2: lower bounds for the diameter.}
Set $k:=|\mathbb Z_n\setminus C_1|=O_{\Prob}(1)$, by
\eqref{eq:critical-window-giant-complement-identity}. With high probability, there
are vertices $x,y\in C_1$ satisfying
\begin{equation}
\label{eq:critical-window-far-giant-vertices}
r_{xy}\geq\left\lfloor\frac n2\right\rfloor-k.
\end{equation}
To see this, fix $x\in C_1$. If $k<n/4$, at least $k+1$ vertices satisfy
$r_{xy}\geq\lfloor n/2\rfloor-k$, and at least one belongs to $C_1$.
The event $k<n/4$ has probability tending to one. Along any path of length $L$,
the triangle inequality for cyclic distance and the host constraint give $r_{xy}\leq Lm$. Hence
\begin{equation}
\label{eq:critical-window-geometric-diameter-lower}
\diam(C_1)
\geq
\frac{r_{xy}}{m}
\geq
\left(\frac12-o_{\Prob}(1)\right)\frac nm.
\end{equation}

For the second lower bound, use Lemma~\ref{lem:critical-window-max-degree}, which
bounds the maximum degree of $C_1$ by $10\log n$ with high probability.
In any connected graph $\Gamma$ with maximum degree $\Delta\geq2$ and diameter
$D_\Gamma$, a breadth-first search from $x$ has at most $\Delta$ vertices in the
first layer. For $j\geq2$, each vertex in layer $j-1$ has at most $\Delta-1$
neighbours in layer $j$, so that layer contains at most
$\Delta(\Delta-1)^{j-1}$ vertices. Summing gives
\begin{equation}
\label{eq:critical-window-moore-bound}
|V(\Gamma)|
\leq
1+\Delta\sum_{j=0}^{D_\Gamma-1}(\Delta-1)^j
\leq
1+\sum_{j=1}^{D_\Gamma}\Delta^j
\leq
\frac{\Delta}{\Delta-1}\Delta^{D_\Gamma}
\leq
2\Delta^{D_\Gamma}.
\end{equation}
We used $\Delta(\Delta-1)^{j-1}\leq\Delta^j$ and
$\Delta/(\Delta-1)\leq2$. Since $|C_1|=n-O_{\Prob}(1)$, we have
$\Delta(C_1)\geq2$ with high probability. Applying the bound to $\Gamma=C_1$,
on the event $|C_1|\geq n/2$ and $\Delta(C_1)\leq10\log n$, gives
\begin{align}
\diam(C_1)\geq
\frac{\log(|C_1|/2)}{\log\Delta(C_1)} \geq
\frac{\log(n/4)}{\log(10\log n)}
=
(1-o(1))\frac{\log n}{\log\log n}.
\label{eq:critical-window-degree-diameter-lower}
\end{align}

Combining \eqref{eq:critical-window-geometric-diameter-lower} and
\eqref{eq:critical-window-degree-diameter-lower}, and using
$\max\{a,b\}\geq(a+b)/2$ proves the lower bound in
\eqref{eq:critical-window-giant-diameter-two-sided}.

\medskip\noindent\emph{Step 3: the matching upper bound.}
Let $\mathcal E_n:=\bigcap_{v\in\mathbb Z_n}\bigcap_{\sigma\in\{-,+\}}\mathcal E_v^\sigma$.
Taking $D=3$ in \eqref{eq:critical-window-shifted-chain-failure} and applying a union
bound over the $2n$ pairs $(v,\sigma)$ gives
\begin{equation}
\label{eq:critical-window-all-short-chains-high-probability}
\Prob(\mathcal E_n^c)
\leq2n\cdot n^{-3}
=2n^{-2}
\longrightarrow0.
\end{equation}
On $\mathcal E_n$, put $M_n:=4(1-\delta_n)\ell$ and $U_n:=2m-M_n$.
\eqref{eq:critical-window-short-chain-uncovered-neighbours} gives
\begin{equation}
\label{eq:critical-window-uncovered-neighbour-upper-bound}
0\leq U_n\leq2+2\delta_nm.
\end{equation}

We show that every non-isolated vertex attaches to one of its chains. Fix $v$ and
condition on the indicators with both endpoints in $\mathbb Z_n\setminus\{v\}$.
On $\mathcal E_v^-\cap\mathcal E_v^+$, put
\begin{equation}
\label{eq:critical-window-actual-covered-neighbour-count}
q_v:=|Q_v^-\cup Q_v^+|,
\qquad
q_v\geq M_n,
\qquad
2m-q_v\leq U_n.
\end{equation}
By conditional independence, the probability that $v$ has no edge to
$Q_v^-\cup Q_v^+$ but has an edge to another admissible neighbour is exactly
$(1-p)^{q_v}\{1-(1-p)^{2m-q_v}\}$. By
\eqref{eq:critical-window-actual-covered-neighbour-count}, this is at most
$(1-p)^{M_n}\{1-(1-p)^{U_n}\}$. Hence
\begin{align}
(1-p)^{M_n}\left(1-(1-p)^{U_n}\right)
&\leq
(1-p)^{M_n}pU_n.
\label{eq:critical-window-no-chain-edge-but-nonisolated}
\end{align}
Since $M_n=2m-U_n$, exact centering, together with
\eqref{eq:critical-window-uncovered-neighbour-upper-bound} and
\eqref{eq:critical-window-adaptive-delta-properties}, gives
\begin{equation}
\label{eq:critical-window-covered-neighbour-probability-sum}
n(1-p)^{M_n}=O_t(1).
\end{equation}
Moreover,
\begin{align}
pU_n=O_t\left(\delta_n\log n+\frac{\log n}{m}\right)=o(1)
\label{eq:critical-window-uncovered-neighbour-edge-probability}
\end{align}
by \eqref{eq:critical-window-adaptive-delta-properties} and
$\log n\ll m$. Let $\mathcal F_n$ be the event that either $\mathcal E_n^c$ occurs or,
on $\mathcal E_n$, some non-isolated vertex has no edge to either of its two chains. A union
bound, \eqref{eq:critical-window-no-chain-edge-but-nonisolated}--
\eqref{eq:critical-window-uncovered-neighbour-edge-probability}, and
\eqref{eq:critical-window-all-short-chains-high-probability} give
\begin{align}
\Prob(\mathcal F_n)
&\leq
\Prob(\mathcal E_n^c)
+
n(1-p)^{M_n}pU_n=o(1).
\label{eq:critical-window-all-nonisolated-attach}
\end{align}
Thus, with high probability, every non-isolated vertex has an edge to at least one of
$\mathcal K_v^-$ and $\mathcal K_v^+$.

Work on $\mathcal F_n^c$ and the high-probability event of
Theorem~\ref{thm:critical-window-only-isolated-outside-giant}. Then $|C_1|>1$, so
every $u\in C_1$ is non-isolated and has an edge to a shifted chain; call one such
chain $\mathcal K_u$. It is connected and contains more than $n/2$ vertices, hence
lies in $C_1$. For $u,v\in C_1$, the sets $\mathcal K_u$ and $\mathcal K_v$
both have more than $n/2$ vertices by \eqref{eq:critical-window-shifted-chain-majority}.
Therefore
\begin{equation}
\label{eq:critical-window-two-chains-intersect}
\mathcal K_u\cap\mathcal K_v\neq\varnothing.
\end{equation}
Choose $z$ in this intersection. Lemma~\ref{lem:critical-window-shifted-short-chain} gives
\begin{align}
\dist_{C_1}(u,v)
&\leq
1+\dist_{G[\mathcal K_u]}(u',z)
+\dist_{G[\mathcal K_v]}(z,v')+1
\notag\\
&\leq
2+2C_{\mathrm{ch}}\left(
\frac nm+\frac{\log n}{\log\log n}
\right),
\label{eq:critical-window-diameter-upper-path}
\end{align}
where $u'\in\mathcal K_u$ and $v'\in\mathcal K_v$ are neighbours of $u$ and $v$,
respectively. Taking the maximum over $u,v\in C_1$ proves the upper bound in
\eqref{eq:critical-window-giant-diameter-two-sided}.
\end{proof}

We finish by expressing the critical window in terms of the expected number of
isolated vertices.

\begin{corollary}[Intrinsic form of the connectivity window]
\label{cor:intrinsic-critical-window}
Assume \(\log n\ll m<n/2\), let \(p=p_n\in(0,1)\), and suppose
\begin{equation}
\label{eq:intrinsic-critical-window}
\lambda_n:=n(1-p)^{2m}\longrightarrow\lambda\in(0,\infty).
\end{equation}
Then
$$
N_{\mathrm{iso}}\Longrightarrow\Poisson(\lambda),
\qquad
\Prob\bigl(|C_1|=n-N_{\mathrm{iso}}\bigr)\longrightarrow1,
\qquad
\Prob(G\text{ is connected})\longrightarrow e^{-\lambda}.
$$
Moreover, the block-chain and diameter conclusions of
Theorem~\ref{thm:critical-window-geometry} remain valid, with constants depending only on
\(\lambda\).
\end{corollary}

\begin{proof}
Put $t_n:=-\log\lambda_n$.  Then $t_n\to-\log\lambda$ and
$$
1-p=\exp\left(-\frac{\log n+t_n}{2m}\right).
$$
The estimates in this section are uniform for $t$ in a fixed compact set.
For each fixed distinct $r$-tuple, in particular,
$$
\Prob(I_{v_1}=\cdots=I_{v_r}=1)
=n^{-r}\lambda_n^r(1-p)^{-e_H(\{v_1,\ldots,v_r\})}
=n^{-r}\lambda_n^r(1+o(1)),
$$
so the factorial moments converge to $\lambda^r$. The block, interface, and
leave-one-vertex-out estimates use only $t_n=O(1)$, and therefore hold with uniform
constants. Replacing $e^{-t}$ by $\lambda_n\to\lambda$ in the expectation comparison
completes the proof.
\end{proof}

Having determined the connectivity transition and its critical window, we turn to
the spectral gap and hitting times above a fixed multiple of the threshold.

\section{Sharp spectral gap above the connectivity threshold}
\label{sec:spectral-gap}

Fix $c>1$ and assume \eqref{eq:gap-main-assumptions}. Write $\mu:=2mp$ for the
expected degree. We prove that the random spectral gap
is asymptotic to $\Gamma_{n,m}$, defined in \eqref{eq:gap-reference-value}. The proof has
two cases. When $m=o(n)$, the reference gap tends to zero, and we need a relative comparison
of Dirichlet forms. We obtain it by partitioning the cycle into blocks and taking a Schur
complement. When $m/n$ is bounded below, the reference gap stays positive, so a direct
comparison of normalized adjacency matrices suffices. We begin with the Fourier estimates
that identify the reference gap and the constants in Corollary~\ref{cor:explicit-spectral-gap}.

\subsection{Deterministic Fourier estimates}
\label{subsec:gap-Fourier-estimates}

We first compute the lazy-walk spectral gap of a cycle power. For integers
$q\geq3$ and $1\leq R<q/2$, let
\begin{equation}
\label{eq:finite-cycle-distance}
d_q(i,j):=\min\{|i-j|,q-|i-j|\},
\qquad 0\leq i,j<q,
\end{equation}
and define
\begin{equation}
\label{eq:finite-cycle-gap}
g_{q,R}
:=
\min_{1\leq k\leq q-1}
\frac1R\sum_{r=1}^R
\sin^2\left(\frac{\pi kr}{q}\right).
\end{equation}

\begin{lemma}[Finite-cycle Poincar\'e inequality]
\label{lem:finite-cycle-Poincare}
The quantity $g_{q,R}$ is the spectral gap of the lazy simple-random-walk transition operator
on $C_q^R$. Moreover, for every
$b=(b_0,\ldots,b_{q-1})\in\mathbb R^q$, with
$\bar b:=q^{-1}\sum_{i=0}^{q-1}b_i$,
\begin{equation}
\label{eq:finite-cycle-Poincare}
\sum_{\substack{0\leq i<j<q\\1\leq d_q(i,j)\leq R}}
(b_i-b_j)^2
\geq
4R g_{q,R}\sum_{i=0}^{q-1}(b_i-\bar b)^2.
\end{equation}
\end{lemma}

\begin{proof}
For $0\leq k<q$, the Fourier vector $\psi_k(j):=\exp(2\pi i k j/q)$ is an eigenvector
of the combinatorial Laplacian of $C_q^R$, with eigenvalue
$\lambda_k=4\sum_{r=1}^R\sin^2(\pi kr/q)$. Since the graph is $2R$-regular, the nonzero
eigenvalues of the lazy normalized Laplacian are $\lambda_k/(4R)$. Their minimum is
$g_{q,R}$. Parseval's identity and $\lambda_0=0$ give
$$
\sum_{\substack{0\leq i<j<q\\1\leq d_q(i,j)\leq R}}
(b_i-b_j)^2
=
b^{\mathsf T}L_{C_q^R}b
\geq
\min_{1\leq k<q}\lambda_k
\sum_{i=0}^{q-1}(b_i-\bar b)^2.
$$
Substituting the definition of $g_{q,R}$ proves \eqref{eq:finite-cycle-Poincare}.
\end{proof}

The block quotient in Subsection~\ref{subsec:gap-sublinear-lower} has interaction range
$R/q\to0$. The next lemma identifies the Poincar\'e constant in this regime.

\begin{lemma}[Sublinear Fourier asymptotics]
\label{lem:sublinear-Fourier-gap}
If $R\to\infty$ and $R/q\to0$, then
\begin{equation}
\label{eq:sublinear-Fourier-limit}
\frac{q^2}{R^2}g_{q,R}
\longrightarrow
\frac{\pi^2}{3}.
\end{equation}
Consequently, if $m\to\infty$ and $m=o(n)$, then
\begin{equation}
\label{eq:Gamma-sublinear-asymptotic}
\Gamma_{n,m}
=
\left(\pi^2+o(1)\right)\frac{\kappa_m}{n^2}
=
\left(\frac{\pi^2}{3}+o(1)\right)\frac{m^2}{n^2},
\end{equation}
and
\begin{equation}
\label{eq:host-gap-sublinear-asymptotic}
\gamma_{\lazy}(C_n^m)
=
\bigl(1+o(1)\bigr)\Gamma_{n,m}.
\end{equation}
\end{lemma}

\begin{proof}
The first Fourier mode gives the upper bound. To prove the matching lower bound, we first
control all frequencies. For $k\in\{0,\ldots,q-1\}$, put $k_*:=\min\{k,q-k\}$ and
$\theta:=2\pi k_*/q$. We claim that there is an absolute constant $a>0$ such that, for
$R\geq8$,
\begin{equation}
\label{eq:cosine-sum-coercivity}
\sum_{r=1}^R\bigl(1-\cos(r\theta)\bigr)
\geq
aR\min\{1,(R\theta)^2\},
\qquad 0\leq\theta\leq\pi.
\end{equation}
If $R\theta\leq4$, restrict the sum to $1\leq r\leq\lfloor R/4\rfloor$. Then
$r\theta\leq1$ and $1-\cos(r\theta)\geq a_0r^2\theta^2$ for an absolute constant
$a_0>0$. The bound follows from
$\sum_{r=1}^{\lfloor R/4\rfloor}r^2\geq a_1R^3$, where $a_1>0$ is absolute.
If $R\theta>4$, the geometric-sum identity and $\sin(\theta/2)\geq\theta/\pi$ give
\begin{align}
\sum_{r=1}^R\bigl(1-\cos(r\theta)\bigr)
&\geq
R-\left|\sum_{r=1}^Re^{ir\theta}\right|
\notag\\
&\geq
R-\frac1{\sin(\theta/2)}
\geq
\left(1-\frac{\pi}{4}\right)R.
\label{eq:cosine-sum-large-angle}
\end{align}
This proves \eqref{eq:cosine-sum-coercivity}.

Using $\sin^2(r\theta/2)=(1-\cos(r\theta))/2$, we deduce from
\eqref{eq:cosine-sum-coercivity} that
\begin{equation}
\label{eq:Fourier-gap-coercivity}
\frac1R\sum_{r=1}^R
\sin^2\left(\frac{\pi kr}{q}\right)
\geq
a_2\min\left\{1,\frac{R^2k_*^2}{q^2}\right\}
\end{equation}
for an absolute constant $a_2>0$. For each fixed positive integer $k$, the assumption
$R/q\to0$ permits a uniform Taylor expansion over $1\leq r\leq R$, giving
\begin{equation}
\label{eq:fixed-Fourier-limit}
\frac{q^2}{R^2}
\frac1R\sum_{r=1}^R
\sin^2\left(\frac{\pi kr}{q}\right)
\longrightarrow
\frac{\pi^2k^2}{3}.
\end{equation}
Taking $k=1$ proves the upper bound in \eqref{eq:sublinear-Fourier-limit}.

For the lower bound, let $k_q$ minimize \eqref{eq:finite-cycle-gap}, and pass to a
subsequence along which the scaled minimum tends to its lower limit. If $(k_q)_*$ is
unbounded, a further subsequence satisfies $(k_q)_*\to\infty$, and
\eqref{eq:Fourier-gap-coercivity} gives
\begin{equation}
\label{eq:unbounded-Fourier-index-exclusion}
\frac{q^2}{R^2}g_{q,R}
\geq
a_2\min\left\{(k_q)_*^2,\frac{q^2}{R^2}\right\}
\longrightarrow\infty,
\end{equation}
contradicting the upper bound from $k=1$. Thus $(k_q)_*$ is bounded. On a further
subsequence it equals a fixed positive integer $k$, and \eqref{eq:fixed-Fourier-limit}
gives the lower limit $\pi^2k^2/3\geq\pi^2/3$. This proves
\eqref{eq:sublinear-Fourier-limit}.

It remains to identify the reference value. Uniformly for $1\leq r\leq m=o(n)$,
$$
\sin^2\left(\frac{\pi r}{n}\right)
=
\frac{\pi^2r^2}{n^2}
+O\left(\frac{r^4}{n^4}\right).
$$
Consequently,
$$
\Gamma_{n,m}
=
\frac{\pi^2}{n^2}\frac1m\sum_{r=1}^m r^2
+O\left(\frac{m^4}{n^4}\right)
=
\left(\pi^2+o(1)\right)\frac{\kappa_m}{n^2},
$$
which proves \eqref{eq:Gamma-sublinear-asymptotic}. Taking $(q,R)=(n,m)$ in
\eqref{eq:sublinear-Fourier-limit} now gives \eqref{eq:host-gap-sublinear-asymptotic}.
\end{proof}

When $m/n\to\rho\in(0,1/2]$, the Fourier sums instead converge to integrals. Write
the simple-random-walk transition operator on $C_n^m$ as
\begin{equation}
\label{eq:host-averaging-operator}
\mathsf P_{n,m}:=\frac1{2m}A_{C_n^m}.
\end{equation}
Its Fourier eigenvalues are
\begin{equation}
\label{eq:host-Fourier-eigenvalues}
\theta_{k,n}
:=
\frac1m\sum_{r=1}^m
\cos\left(\frac{2\pi kr}{n}\right),
\qquad 0\leq k<n.
\end{equation}
Since $\theta_{0,n}=1$,
\begin{equation}
\label{eq:host-gap-Fourier-formula}
\gamma_{\lazy}(C_n^m)
=
\frac12\left(1-\max_{1\leq k<n}\theta_{k,n}\right).
\end{equation}

\begin{lemma}[Fourier asymptotics for bandwidth proportional to $n$]
\label{lem:macroscopic-Fourier-gap}
If $m/n\to\rho\in(0,1/2]$, then
\begin{equation}
\label{eq:macroscopic-second-eigenvalue}
\max_{1\leq k<n}\theta_{k,n}
\longrightarrow
\begin{cases}
\displaystyle\frac{\sin(2\pi\rho)}{2\pi\rho},
&0<\rho<1/2, \\[2ex]
0,&\rho=1/2.
\end{cases}
\end{equation}
Moreover,
\begin{equation}
\label{eq:Gamma-macroscopic-asymptotic}
\Gamma_{n,m}
\longrightarrow
\frac12\left(
1-\frac{\sin(2\pi\rho)}{2\pi\rho}
\right),
\end{equation}
and
\begin{equation}
\label{eq:host-gap-macroscopic-asymptotic}
\gamma_{\lazy}(C_n^m)
=
\bigl(1+o(1)\bigr)\Gamma_{n,m}.
\end{equation}
\end{lemma}

\begin{proof}
For each fixed positive integer $k$, the Riemann sum in
\eqref{eq:host-Fourier-eigenvalues} gives
$$
\theta_{k,n}\longrightarrow\frac{\sin(2\pi k\rho)}{2\pi k\rho}.
$$
Suppose first that $0<\rho<1/2$ and set $x:=2\pi\rho\in(0,\pi)$. Induction using the
angle-addition formula gives $|\sin(kx)|\leq k\sin x$. Hence, for every fixed positive
integer $k$,
$$
\frac{\sin(kx)}{kx}
\leq
\frac{|\sin(kx)|}{kx}
\leq
\frac{\sin x}{x}.
$$
Thus among fixed indices, $k=1$ has the largest limiting eigenvalue, and this limit is positive.

We next show that diverging indices have eigenvalues tending to zero. Put
$k_*:=\min\{k,n-k\}$. The geometric-sum identity gives
$|\sum_{r=1}^m e^{2\pi i kr/n}|\leq|\sin(\pi k_*/n)|^{-1}$.
If $k_*\to\infty$ and $k_*/n\to0$, dividing by $m$ gives
$O(n/(mk_*))=o(1)$. If $k_*/n$ stays bounded away from zero, the bound is
$O(1/m)=o(1)$. Passing to subsequences therefore shows that
$\theta_{k,n}\to0$ whenever $\min\{k,n-k\}\to\infty$.

Let $k_n$ maximize $\theta_{k,n}$ over $1\leq k<n$. Along a further subsequence of any
subsequence, $\min\{k_n,n-k_n\}$ is either fixed or tends to infinity. The preceding
bounds show that the maximum converges to the $k=1$ limit when $\rho<1/2$. When
$\rho=1/2$, both limits are zero. This proves \eqref{eq:macroscopic-second-eigenvalue}.

By \eqref{eq:gap-reference-value} and \eqref{eq:host-Fourier-eigenvalues},
$\Gamma_{n,m}=(1-\theta_{1,n})/2$. Combining this identity with
\eqref{eq:host-gap-Fourier-formula} and \eqref{eq:macroscopic-second-eigenvalue} proves
\eqref{eq:Gamma-macroscopic-asymptotic} and \eqref{eq:host-gap-macroscopic-asymptotic}.
\end{proof}

Lemmas~\ref{lem:sublinear-Fourier-gap} and~\ref{lem:macroscopic-Fourier-gap} cover
$m=o(n)$ and $m/n\to\rho\in(0,1/2]$. A subsequence argument combines them into a
comparison valid throughout $\log n\ll m<n/2$.

\begin{proposition}[Spectral gap of the cycle power]
\label{prop:host-spectral-gap}
If $\log n\ll m<n/2$, then
\begin{equation}
\label{eq:host-gap-Gamma-comparison}
\gamma_{\lazy}(C_n^m)
=
\bigl(1+o(1)\bigr)\Gamma_{n,m}.
\end{equation}
There are absolute constants $0<c_\Gamma<C_\Gamma<\infty$ such that
\begin{equation}
\label{eq:Gamma-uniform-bounds}
c_\Gamma\frac{m^2}{n^2}
\leq
\Gamma_{n,m}
\leq
C_\Gamma\frac{m^2}{n^2}
\end{equation}
for all sufficiently large $n$.
\end{proposition}

\begin{proof}
Every subsequence has a further subsequence on which $m/n\to\rho$ for some
$\rho\in[0,1/2]$. Apply Lemma~\ref{lem:sublinear-Fourier-gap} when $\rho=0$ and
Lemma~\ref{lem:macroscopic-Fourier-gap} when $\rho>0$. In either case, the ratio in
\eqref{eq:host-gap-Gamma-comparison} tends to one, proving the comparison.

For $1\leq r\leq m<n/2$, the inequalities
$2r/n\leq\sin(\pi r/n)\leq\pi r/n$ give
\begin{equation}
\label{eq:Gamma-sum-square-bounds}
\frac4{mn^2}\sum_{r=1}^m r^2
\leq
\Gamma_{n,m}
\leq
\frac{\pi^2}{mn^2}\sum_{r=1}^m r^2.
\end{equation}
Now $\sum_{r=1}^m r^2=m(m+1)(2m+1)/6$ gives \eqref{eq:Gamma-uniform-bounds}.
\end{proof}

\subsection{Degree, volume, and adjacency-matrix estimates}
\label{subsec:gap-probabilistic-estimates}

We need two types of concentration estimates. Degree and volume bounds control the
stationary measure and the variational denominator. A norm bound for the centered adjacency
matrix controls the random part of the Dirichlet form. We begin with a consequence of
Bandeira and van Handel~\cite[Corollary~3.12 and Remark~3.13]{BandeiraVanHandel2016}.

\begin{lemma}[Centered adjacency-matrix norm]
\label{lem:centered-adjacency-norm}
Let $F$ be a deterministic graph on $N$ vertices, let $F_p$ be obtained by retaining each
edge independently with probability $p$, and set
\begin{equation}
\label{eq:centered-adjacency-matrix}
Z:=A_{F_p}-pA_F,
\qquad
\sigma^2:=\max_x\sum_y\E[Z(x,y)^2].
\end{equation}
For every fixed $B>0$, there is a constant $C_B<\infty$ such that
\begin{equation}
\label{eq:centered-adjacency-norm}
\Prob\left(
\lVert Z\rVert_{\mathrm{op}}>
C_B\bigl(\sigma+\sqrt{\log N}\bigr)
\right)
\leq
N^{-B}.
\end{equation}
\end{lemma}

\begin{proof}
The entries of $Z$ above the diagonal are independent, centered, and bounded in absolute
value by one. Apply Corollary~3.12 and Remark~3.13 of \cite{BandeiraVanHandel2016}, fixing
their parameter $\varepsilon$ and using $\widetilde\sigma_*\leq1$, to obtain
$$
\Prob\left(
\lVert Z\rVert_{\mathrm{op}}>
C_0\sigma+t
\right)
\leq N\exp(-t^2/C_0),
\qquad t\geq0,
$$
for an absolute constant $C_0$. Choose
$$
t=\sqrt{C_0(B+2)\log N}.
$$
Then
$$
N\exp(-t^2/C_0)=N^{-B-1}\leq N^{-B}.
$$
Increasing the coefficient of $\sigma+\sqrt{\log N}$ gives
\eqref{eq:centered-adjacency-norm}.
\end{proof}

The next lemma gives uniform degree and volume bounds, together with a sharp estimate
for the minimum degree. Connectedness throughout the parameter range already follows from
\eqref{eq:connectivity-supercritical-uniform}.

For $x\in(0,1)$ and $y>1$, define
\begin{equation}
\label{eq:chernoff-rate-functions}
I_-(x):=1-x+x\log x,
\qquad
I_+(y):=y\log y-y+1.
\end{equation}
Recall from \eqref{eq:minimum-degree-profile} that the minimum-degree profile
$\delta_{\min}$ maps $(1,\infty]$ to $(0,1]$, with $\delta_{\min}(\infty)=1$.
For $\theta\in(1,\infty)$, its defining relation is
$$
\theta I_-(\delta_{\min}(\theta))=1.
$$
The function $\delta_{\min}$ is increasing and converges to one as
$\theta\to\infty$.

\begin{lemma}[Uniform degree, minimum-degree, and volume estimates]
\label{lem:degrees-fixed-c}
There exist constants $d_-(c)\in(0,1)$ and $d_+(c)>1$ depending only on $c$.
Uniformly over $p\in[c\log n/(2m),1]$, the following bounds hold with probability tending to
one:
\begin{equation}
\label{eq:fixed-c-degree-event}
d_-(c)\mu
\leq
d_G(v)
\leq
d_+(c)\mu
\qquad\text{for every }v\in\mathbb Z_n.
\end{equation}
Consequently, on the same event,
\begin{equation}
\label{eq:fixed-c-volume-event}
d_-(c)\mu|S|
\leq
\vol_G(S)
\leq
d_+(c)\mu|S|
\qquad\text{for every }S\subseteq\mathbb Z_n.
\end{equation}
In addition,
\begin{equation}
\label{eq:uniform-total-volume}
\frac{\vol_G(\mathbb Z_n)}{n\mu}
\xrightarrow{\Prob}1
\end{equation}
uniformly over the same range of $p$.
Moreover,
\begin{equation}
\label{eq:uniform-sharp-minimum-degree}
\frac{d_{\min}(G)}{\mu}
-
\delta_{\min}(\mu/\log n)
\xrightarrow{\Prob}0
\end{equation}
uniformly over $p\in[c\log n/(2m),1]$.
\end{lemma}

\begin{proof}
\emph{Degree and volume bounds.}
The rate functions satisfy $I_-(x)\to1$ as $x\downarrow0$ and $I_+(y)\to\infty$ as
$y\to\infty$. Since $c>1$, we may choose $d_-(c)\in(0,1)$ and $d_+(c)>1$ so that
\begin{equation}
\label{eq:degree-rate-choice}
cI_-(d_-(c))>1,
\qquad
cI_+(d_+(c))>1.
\end{equation}

For $Z_0\sim\Bin(N,p_0)$ with mean $\bar\mu=Np_0$, the inequality $1+u\leq e^u$ gives
\begin{equation}
\label{eq:binomial-exponential-moment-bound}
\E[e^{\theta Z_0}]
=(1-p_0+p_0e^\theta)^N
\leq
\exp\bigl(\bar\mu(e^\theta-1)\bigr),
\qquad \theta\in\mathbb R.
\end{equation}
Markov's inequality with $\theta=\log x<0$ and $\theta=\log y>0$, respectively, yields
\begin{align}
\label{eq:binomial-rate-tail-bounds}
\Prob(Z_0\leq x\bar\mu)
&\leq
\exp\bigl[-\bar\mu I_-(x)\bigr],
\notag\\
\Prob(Z_0\geq y\bar\mu)
&\leq
\exp\bigl[-\bar\mu I_+(y)\bigr].
\end{align}

Each $v\in\mathbb Z_n$ has degree $d_G(v)$ distributed as $\Bin(2m,p)$, with mean
$\mu\geq c\log n$. Let $\mathcal D_n$ be the event in \eqref{eq:fixed-c-degree-event}.
By a union bound and \eqref{eq:degree-rate-choice}--\eqref{eq:binomial-rate-tail-bounds},
\begin{align}
\Prob(\mathcal D_n^c)
&\leq
n\exp\bigl[-\mu I_-(d_-(c))\bigr]
+n\exp\bigl[-\mu I_+(d_+(c))\bigr]
\notag\\
&\leq
n^{1-cI_-(d_-(c))}
+n^{1-cI_+(d_+(c))}
=o(1),
\label{eq:fixed-c-degree-failure}
\end{align}
uniformly in $p$. Summing the degree bounds over $v\in S$ gives
\eqref{eq:fixed-c-volume-event} for every $S$ on the same event.

For the total volume, use $|E(G)|\sim\Bin(nm,p)$, with mean $nmp=n\mu/2$, and
$\vol_G(\mathbb Z_n)=2|E(G)|$. For each fixed $\varepsilon\in(0,1)$,
\eqref{eq:chernoff-two-sided} gives
\begin{equation}
\label{eq:uniform-total-volume-tail}
\sup_{p\in[c\log n/(2m),1]}
\Prob_p\left(
\left|
\frac{\vol_G(\mathbb Z_n)}{n\mu}-1
\right|>\varepsilon
\right)
\leq
2\exp\left(-\frac{\varepsilon^2cn\log n}{6}\right),
\end{equation}
which proves \eqref{eq:uniform-total-volume}.

\smallskip\noindent
\emph{The binomial lower-tail exponent.}
Suppose that $\mu/\log n$ lies in a compact subset of $(1,\infty)$ and $x$ lies in a compact
subset of $(0,1)$. For $D\sim\Bin(2m,p)$ with mean $\mu$, we claim that
\begin{equation}
\label{eq:binomial-lower-tail-logarithmic-asymptotic}
\log\Prob(D\leq x\mu)
=
-\mu I_-(x)+o(\log n),
\end{equation}
uniformly over these parameters. The upper bound is \eqref{eq:binomial-rate-tail-bounds}.
For the lower bound, take $k=\lfloor x\mu\rfloor$ and retain only the term $\Prob(D=k)$.
Stirling's formula, $\mu=\Theta(\log n)$, $p=\mu/(2m)$, and $\log n\ll m$ give
\begin{align*}
\log\Prob(D=k)
&=
\log\binom{2m}{k}
+k\log p
+(2m-k)\log(1-p)\\
&=
k\log\left(\frac{\mu}{k}\right)
+k-\mu
+O\left(\log\mu+\frac{\mu^2}{m}\right)\\
&=
-\mu I_-(x)+o(\log n).
\end{align*}
This proves \eqref{eq:binomial-lower-tail-logarithmic-asymptotic}.

\smallskip\noindent
\emph{The minimum degree.}
Take any sequence of admissible retention probabilities and put $\theta_n:=\mu/\log n$.
Pass to a subsequence on which $\theta_n\to\theta\in[c,\infty]$. First suppose that
$\theta<\infty$, and choose $0<x_-<\delta_{\min}(\theta)<x_+<1$.
The tail bounds \eqref{eq:binomial-rate-tail-bounds} and
\eqref{eq:binomial-lower-tail-logarithmic-asymptotic} imply
\begin{align}
\label{eq:minimum-degree-lower-threshold}
n\Prob\bigl(d_G(0)\leq x_-\mu\bigr)&\longrightarrow0,\\
\label{eq:minimum-degree-upper-threshold}
n\Prob\bigl(d_G(0)\leq x_+\mu\bigr)&\longrightarrow\infty,
\end{align}
since $\theta I_-(x_-)>1>\theta I_-(x_+)$. The union bound and \eqref{eq:minimum-degree-lower-threshold} imply
$d_{\min}(G)>x_-\mu$ with probability tending to one.

For the reverse bound, count vertices of low degree:
$$
X:=\sum_{v\in\mathbb Z_n}
\mathbf 1_{\{d_G(v)\leq x_+\mu\}},
\qquad
r:=\Prob\bigl(d_G(0)\leq x_+\mu\bigr).
$$
The degrees $d_G(u)$ and $d_G(v)$ are independent when $\{u,v\}\notin E(H)$.
For a host edge $\{u,v\}$ with indicator $\xi$, write
$d_G(u)=D_u+\xi$ and $d_G(v)=D_v+\xi$, where $D_u,D_v\sim\Bin(2m-1,p)$ are independent
of each other and of $\xi$. Set
$$
F:=\Prob(D_u\leq\lfloor x_+\mu\rfloor),
\qquad
\Delta:=\Prob(D_u=\lfloor x_+\mu\rfloor),
$$
Conditioning on $\xi$ gives $r=F-p\Delta$ and joint probability
$(1-p)F^2+p(F-\Delta)^2$. Subtracting $r^2$, we obtain
\begin{equation}
\label{eq:adjacent-low-degree-covariance}
\operatorname{Cov}\left(
\mathbf 1_{\{d_G(u)\leq x_+\mu\}},
\mathbf 1_{\{d_G(v)\leq x_+\mu\}}
\right)
=
p(1-p)\Delta^2.
\end{equation}
The bounds $r=F-p\Delta\geq(1-p)F$ and $\Delta\leq F$ show that this covariance is at most
$p(1-p)^{-1}r^2$. Moreover, $p\to0$ since $\theta<\infty$ and $\log n\ll m$.
Summing over the $nm$ host edges gives
\begin{equation}
\label{eq:low-degree-count-second-moment}
\frac{\Var(X)}{\E[X]^2}
\leq
\frac1{nr}
+C\frac{mp}{n}
\longrightarrow0
\end{equation}
by \eqref{eq:minimum-degree-upper-threshold} and
$mp=\mu/2=O(\log n)$. Thus $X>0$ with high probability, so $d_{\min}(G)\leq x_+\mu$. Letting $x_-$ and
$x_+$ approach $\delta_{\min}(\theta)$ gives
$d_{\min}(G)/\mu\to\delta_{\min}(\theta)$ in probability. The inverse relation
$\theta I_-(\delta_{\min}(\theta))=1$ is continuous, so
$\delta_{\min}(\theta_n)\to\delta_{\min}(\theta)$ as well.

If $\theta=\infty$, a union bound and \eqref{eq:binomial-rate-tail-bounds} give, for
each fixed $\varepsilon\in(0,1)$,
$$
\Prob\bigl(d_{\min}(G)<(1-\varepsilon)\mu\bigr)
\leq
n\exp\bigl[-\mu I_-(1-\varepsilon)\bigr]
\longrightarrow0.
$$
The minimum degree is at most the average degree, so \eqref{eq:uniform-total-volume} gives
$$
\frac{d_{\min}(G)}{\mu}
\leq
\frac{\vol_G(\mathbb Z_n)}{n\mu}
=
1+o_{\Prob}(1).
$$
Hence $d_{\min}(G)/\mu\to1=\delta_{\min}(\infty)$ in probability.
The argument applies to every subsequence of admissible retention probabilities.
If \eqref{eq:uniform-sharp-minimum-degree} were not uniform, there would be
$\varepsilon>0$ and a sequence along which the failure probability stayed above
$\varepsilon$. Passing to a further subsequence as above gives a contradiction.
\end{proof}

\subsection{Spectral gap upper bound}
\label{subsec:gap-sharp-upper}

For the upper bound, we test the variational formula with $x\mapsto\cos(2\pi x/n)$.
Its energy and degree-weighted variance both concentrate to relative error $o(1)$,
uniformly over $p\geq c\log n/(2m)$.

\begin{proposition}[Sharp spectral gap upper bound]
\label{prop:sharp-gap-upper}
Under the assumptions of Theorem~\ref{thm:spectral-gap}, for every $\varepsilon>0$,
\begin{equation}
\label{eq:sharp-gap-upper-uniform}
\sup_{p\in[c\log n/(2m),1]}
\Prob_p\left(
G\text{ is disconnected}
\ \text{or}\
\frac{\gamma_{\lazy}(G)}{\Gamma_{n,m}}>1+\varepsilon
\right)
\longrightarrow0.
\end{equation}
\end{proposition}

\begin{proof}
Set $p_0:=c\log n/(2m)$. By \eqref{eq:connectivity-supercritical-uniform}, $G$ is
connected with high probability, uniformly over $p\in[p_0,1]$. We estimate the Rayleigh
quotient on this event.

Set $\eta_n:=(n\log n)^{-1/4}$ and
$\phi(x):=\cos(2\pi x/n)$ for $x\in\mathbb Z_n$. For a host edge $e=\{u,v\}$, put
$a_e:=(\phi(u)-\phi(v))^2$ and $X:=\sum_{e\in E(H)}\xi_ea_e$.
There are $n$ host edges of each cyclic length $r\in\{1,\ldots,m\}$. Summing
$\cos\alpha-\cos(\alpha+\theta)
=2\sin(\theta/2)\sin(\alpha+\theta/2)$ over the starting vertex gives
\begin{equation}
\label{eq:cosine-host-energy}
\sum_{e\in E(H)}a_e
=
2n\sum_{r=1}^m\sin^2\left(\frac{\pi r}{n}\right)
=
2nm\Gamma_{n,m}.
\end{equation}
The same identity gives
\begin{equation}
\label{eq:cosine-edge-maximum}
\max_{e\in E(H)}a_e
\leq
4\pi^2\frac{m^2}{n^2}.
\end{equation}

Since $a_e^2\leq(\max_fa_f)a_e$,
$\Var(X)\leq8\pi^2p(m^2/n^2)\,nm\Gamma_{n,m}$.
\eqref{eq:cosine-host-energy}, \eqref{eq:cosine-edge-maximum}, and
\eqref{eq:Gamma-uniform-bounds} imply
\begin{equation}
\label{eq:cosine-energy-relative-scales}
\frac{\Var(X)}{(\E X)^2}
\leq
\frac{C}{pmn},
\qquad
\frac{\max_ea_e}{\E X}
\leq
\frac{C}{pmn}
\end{equation}
for an absolute constant $C$. Lemma~\ref{lem:bernstein} therefore gives
\begin{equation}
\label{eq:random-cosine-energy-tail}
\Prob\left(
\left|
\frac{X}{2pnm\Gamma_{n,m}}-1
\right|>\eta_n
\right)
\leq
2\exp\bigl(-c_0\eta_n^2pmn\bigr)
\end{equation}
for an absolute constant $c_0>0$ and all sufficiently large $n$. Since
\begin{equation}
\label{eq:sharp-upper-concentration-scale}
\eta_n^2pmn
\geq
\frac c2\sqrt{n\log n}
\longrightarrow\infty,
\end{equation}
the energy estimate holds with high probability, uniformly in $p$.

We next estimate the degree-weighted variance of $\phi$. Write
$Z_\phi:=\sum_xd_G(x)\phi(x)^2$,
$Y_\phi:=\sum_xd_G(x)\phi(x)$, and
$V_G:=\vol_G(\mathbb Z_n)$.
Expressing these quantities as edge sums and using $\sum_x\phi(x)^2=n/2$ and
$\sum_x\phi(x)=0$, we obtain $\E[Z_\phi]=n\mu/2$, $\E[Y_\phi]=0$, and
$\E[V_G]=n\mu$. Each centered summand is bounded by two, and the variance sum is at most
$C_1n\mu$ for an absolute constant $C_1$. Bernstein's inequality and a union bound give
\begin{equation}
\Prob\biggl(
\left|Z_\phi-\frac{n\mu}{2}\right|>\eta_n n\mu
\ \text{or}\
|Y_\phi|>\eta_n n\mu
\ \text{or}\
|V_G-n\mu|>\eta_n n\mu
\biggr)
\leq
6\exp\bigl(-c_1\eta_n^2n\mu\bigr)
\label{eq:weighted-cosine-moment-tail}
\end{equation}
for an absolute constant $c_1>0$ and all sufficiently large $n$. Since the exponent is at
least $c_1c\sqrt{n\log n}$, this failure probability tends to zero uniformly in $p$.

Outside the exceptional event in \eqref{eq:weighted-cosine-moment-tail},
\begin{align}
\min_{a\in\mathbb R}\sum_xd_G(x)\bigl(\phi(x)-a\bigr)^2
&=
Z_\phi-\frac{Y_\phi^2}{V_G}
\notag\\
&=
\left(1+O(\eta_n)\right)\frac{n\mu}{2}.
\label{eq:weighted-cosine-variance}
\end{align}
The first equality follows by minimizing over $a$. On the intersection with the connectedness event,
Lemma~\ref{lem:variational-gap} and \eqref{eq:graph-dirichlet-edge-formula} now give
\begin{align}
\gamma_{\lazy}(G)
&\leq
\frac{X}
{2\min_{a\in\mathbb R}\sum_xd_G(x)(\phi(x)-a)^2}
\notag\\
&\leq
\bigl(1+O(\eta_n)\bigr)
\frac{2pnm\Gamma_{n,m}}{n\mu}
\notag\\
&=
\bigl(1+O(\eta_n)\bigr)\Gamma_{n,m},
\label{eq:sharp-gap-upper-Rayleigh}
\end{align}
where we used $\mu=2mp$. The failure bounds
\eqref{eq:connectivity-supercritical-uniform}, \eqref{eq:random-cosine-energy-tail}, and
\eqref{eq:weighted-cosine-moment-tail} prove \eqref{eq:sharp-gap-upper-uniform}.
\end{proof}

\subsection[Spectral gap lower bound when m=o(n)]
{Spectral gap lower bound when $m=o(n)$}
\label{subsec:gap-sublinear-lower}

We now assume
\begin{equation}
\label{eq:sublinear-bandwidth}
m=o(n).
\end{equation}
The lower bound comes from separating block averages from fluctuations within blocks.
The averages have nearly the energy of a deterministic cycle power on the block indices.
The fluctuations cost at least a positive constant times $\mu$ in Laplacian energy.
A Schur complement then shows that allowing these fluctuations changes the energy of
the block averages by only a relative $o(1)$.

Set
\begin{equation}
\label{eq:block-K-q}
K:=\left\lfloor(\log n)^{1/4}\right\rfloor,
\qquad
q:=\left\lceil\frac{nK}{m}\right\rceil.
\end{equation}
For $0\leq i\leq q$, let $t_i:=\lfloor in/q\rfloor$, and partition $\mathbb Z_n$ into the
consecutive blocks
\begin{equation}
\label{eq:sublinear-block-partition}
B_i:=\{t_i,t_i+1,\ldots,t_{i+1}-1\},
\qquad 0\leq i<q.
\end{equation}
Write $s_i:=|B_i|$, $\bar s:=n/q$,
$s_-:=\min_i s_i$, and $s_+:=\max_i s_i$. Define
$R:=\lfloor m/s_+\rfloor-1$.
Since $m=o(n)$, we have $nK/m\to\infty$ and
\begin{equation}
\label{eq:block-asymptotics}
q\sim\frac{nK}{m},
\bar s\sim\frac mK,
\qquad
R\sim K,
\qquad
\frac{s_+}{s_-}\longrightarrow1,
\qquad
\frac Rq\longrightarrow0.
\end{equation}
Here the block sizes differ by at most one, and $\bar s\sim m/K\to\infty$ because
$m/\log n\to\infty$. In particular, $1\leq R<q/2$ for all sufficiently large $n$.

Let $H^-$ contain all edges within a block and all edges between blocks at cyclic block
distance at most $R$. We check that these edges are admissible. Since
$R+1=\lfloor m/s_+\rfloor$, we have $(R+1)s_+\leq m$. Vertices in one block have
cyclic distance at most $s_+-1<m$. Vertices in distinct blocks at block distance at most
$R$ are joined by an arc meeting at most $R+1$ blocks, of length at most
$(R+1)s_+-1<m$. Hence
\begin{equation}
\label{eq:sublinear-host-inclusion}
H^-\subseteq H.
\end{equation}
Let $G^-$ be obtained from $H^-$ using the same retention indicators as $G$.

The advantage of $H^-$ is that its blocks are complete and its nonempty interfaces are
complete bipartite graphs. The next lemma gives the concentration estimates needed to
compare $G^-$ with a deterministic quadratic form on the block indices.

\begin{lemma}[Block estimates when $m=o(n)$]
\label{lem:sublinear-block-events}
There exist constants $\alpha=\alpha(c)>0$ and $A=A(c)<\infty$, and a deterministic
sequence $\varepsilon_n\to0$, such that the following estimates hold simultaneously
with probability tending to one, uniformly over $p\in[c\log n/(2m),1]$:
\begin{align}
\min_x d_{G^-}(x)&\geq\alpha\mu,
&
\max_x d_G(x)&\leq A\mu,
\label{eq:block-degree-events}\\
\left\lVert A_{G^-}-pA_{H^-}\right\rVert_{\mathrm{op}}
&\leq A\sqrt\mu,
\label{eq:block-adjacency-event}\\
\max_{0\leq i<q}
\left|
\frac{\vol_G(B_i)}{\mu s_i}-1
\right|
&\leq\varepsilon_n,
\label{eq:block-volume-event}\\
\max_{\substack{0\leq i<j<q\\1\leq d_q(i,j)\leq R}}
\left|
\frac{e_{G^-}(B_i,B_j)}{ps_is_j}-1
\right|
&\leq\varepsilon_n.
\label{eq:block-edge-event}
\end{align}
\end{lemma}

\begin{proof}
We first control degrees. For $x\in B_i$, write $h_i:=d_{H^-}(x)$. By construction,
\begin{equation}
\label{eq:sublinear-host-degree}
h_i
=
s_i-1+
\sum_{\substack{0\leq j<q\\1\leq d_q(i,j)\leq R}}s_j.
\end{equation}
Because $R\geq m/s_+-2$ and $s_-/s_+\to1$,
\begin{equation}
\label{eq:sublinear-host-degree-asymptotic}
\min_i h_i
\geq
(2R+1)s_--1
=
2m\bigl(1-o(1)\bigr).
\end{equation}
Since $H^-\subseteq H$, we also have $h_i\leq2m$. Thus
\begin{equation}
\label{eq:sublinear-retained-degree-mean}
ph_i=\bigl(1-o(1)\bigr)\mu
\end{equation}
uniformly in $i$ and $p$.

Choose $c_0\in(1,c)$. Then \eqref{eq:sublinear-retained-degree-mean} gives
$ph_i\geq c_0\log n$ for all $i$ and sufficiently large $n$. Choose
$\alpha\in(0,1/2)$ so that $c_0I_-(2\alpha)>1$, with $I_-$ as in
\eqref{eq:chernoff-rate-functions}. For large $n$,
$\alpha\mu/(ph_i)\leq2\alpha$ uniformly in $i$. Since $I_-$ decreases on $(0,1)$,
\eqref{eq:binomial-rate-tail-bounds} gives
\begin{align}
\Prob\bigl(d_{G^-}(x)<\alpha\mu\bigr)
&\leq
\exp\left[
-ph_iI_-\left(\frac{\alpha\mu}{ph_i}\right)
\right]
\notag\\
&\leq
n^{-c_0I_-(2\alpha)}.
\label{eq:sublinear-degree-lower-tail}
\end{align}
A union bound over $x\in\mathbb Z_n$ proves the lower degree bound in
\eqref{eq:block-degree-events}; the upper bound is Lemma~\ref{lem:degrees-fixed-c}.

Next set $Z:=A_{G^-}-pA_{H^-}$. Its row-variance sums are at most
$p(1-p)\max_i h_i\leq\mu$. Applying Lemma~\ref{lem:centered-adjacency-norm} with
$B=2$ and using $\mu\geq c\log n$, we obtain
$\Prob(\lVert Z\rVert_{\mathrm{op}}>A_0(c)\sqrt\mu)\leq n^{-2}$ for some
$A_0(c)<\infty$. Increase $A$ to obtain \eqref{eq:block-adjacency-event}.

For the block volumes, use the decomposition
\begin{equation}
\label{eq:block-volume-edge-decomposition}
\vol_G(B_i)
=
2\sum_{\substack{e\in E(H)\\e\subseteq B_i}}\xi_e
+
\sum_{\substack{e\in E(H)\\|e\cap B_i|=1}}\xi_e.
\end{equation}
The summands are independent, with coefficients in $\{1,2\}$. Their sum has mean
$\mu s_i$ and variance at most $2\mu s_i$, and each centered summand is bounded by two.
Bernstein's inequality gives, for $0<\eta\leq1$,
\begin{equation}
\label{eq:block-volume-tail}
\Prob\left(
\left|\vol_G(B_i)-\mu s_i\right|>\eta\mu s_i
\right)
\leq
2\exp\bigl(-a\eta^2\mu s_i\bigr)
\end{equation}
for an absolute constant $a>0$.

For $1\leq d_q(i,j)\leq R$, all pairs in $B_i\times B_j$ are edges of $H^-$. Hence
\begin{equation}
\label{eq:sublinear-block-edge-distribution}
e_{G^-}(B_i,B_j)\sim\Bin(s_is_j,p),
\end{equation}
and Lemma~\ref{lem:chernoff} gives
\begin{equation}
\label{eq:block-edge-tail}
\Prob\left(
\left|e_{G^-}(B_i,B_j)-ps_is_j\right|>\eta ps_is_j
\right)
\leq
2\exp\bigl(-a\eta^2ps_is_j\bigr).
\end{equation}
The concentration scales satisfy
\begin{equation}
\label{eq:sublinear-concentration-scales}
\frac{\mu s_-}{\log n}\longrightarrow\infty,
\qquad
\frac{ps_-^2}{\log n}
\geq
\bigl(1+o(1)\bigr)\frac{cm}{2K^2}
\longrightarrow\infty.
\end{equation}
Define $r_n:=\min\{s_-,m/K^2\}$ and
$\varepsilon_n:=r_n^{-1/4}$.
Then $r_n\to\infty$, and both quantities on the left of
\eqref{eq:sublinear-concentration-scales} are bounded below by a positive constant times
$r_n$. Taking $\eta=\varepsilon_n$ makes the exponents in
\eqref{eq:block-volume-tail} and \eqref{eq:block-edge-tail} of order
$\omega(\log n)$. Since $q\leq n$ for large $n$, a union bound over the $q$ blocks and
at most $q^2$ pairs with $1\leq d_q(i,j)\leq R$ proves
\eqref{eq:block-volume-event} and \eqref{eq:block-edge-event}. All bounds are uniform in
$p$, using only $\mu\geq c\log n$.
\end{proof}

We now separate the block averages from the fluctuations. Let $U\subseteq\mathbb R^n$
be the block-constant subspace and $W:=U^\perp$ its Euclidean orthogonal complement.
Identify $\mathbb R^q$ isometrically with $U$ by
\begin{equation}
\label{eq:block-isometry}
(\mathsf Ua)(x)
:=
\frac{a_i}{\sqrt{s_i}}
=:
b_i,
\qquad x\in B_i.
\end{equation}
Relative to the orthogonal decomposition $U\oplus W$, write
\begin{equation}
\label{eq:block-Laplacian-matrix}
L_{G^-}
=
\begin{pmatrix}
M&N\\
N^{\mathsf T}&C_W
\end{pmatrix}.
\end{equation}
Define the deterministic quadratic form $M_0$ on $\mathbb R^q$ by
\begin{equation}
\label{eq:M0-form}
a^{\mathsf T}M_0a
:=
p\sum_{\substack{0\leq i<j<q\\1\leq d_q(i,j)\leq R}}
s_is_j(b_i-b_j)^2.
\end{equation}
Thus $M$ describes the energy of block averages, $C_W$ describes the energy of
within-block fluctuations, and $N$ couples the two. The next proposition shows that
eliminating the fluctuations changes $M_0$ by a relative error tending to zero.

\begin{proposition}[Block Schur-complement comparison]
\label{prop:block-Schur}
On the event in Lemma~\ref{lem:sublinear-block-events},
\begin{equation}
\label{eq:C-coercivity}
C_W\succeq\frac{\alpha\mu}{2}I_W
\end{equation}
for all sufficiently large $n$. Moreover, there is a deterministic sequence
$\zeta_n\to0$, depending only on $c$ and the bandwidth sequence, such that
\begin{equation}
\label{eq:Schur-comparison}
(1-\zeta_n)M_0
\preceq
M-NC_W^{-1}N^{\mathsf T}
\preceq
(1+\zeta_n)M_0.
\end{equation}
The sequence $\zeta_n$ is uniform over $p\in[c\log n/(2m),1]$.
\end{proposition}

\begin{proof}
The constants $C_1,\ldots,C_4$ depend at most on $c$.

\emph{The diagonal blocks.}
For $u=\mathsf Ua$, edges within a block contribute no energy. Thus
\eqref{eq:block-edge-event} gives
\begin{equation}
\label{eq:M-M0-comparison}
(1-\varepsilon_n)M_0
\preceq
M
\preceq
(1+\varepsilon_n)M_0.
\end{equation}

For $w\in W$, each restriction $w_i:=w|_{B_i}$ has coordinate sum zero.
The off-diagonal blocks of $A_{H^-}$ are zero or all-ones matrices, so they annihilate
these restrictions. The diagonal block $J_{s_i}-I_{s_i}$ acts as $-I_{s_i}$. Hence
\begin{equation}
\label{eq:host-on-W}
A_{H^-}w=-w.
\end{equation}
Set $Z:=A_{G^-}-pA_{H^-}$. The degree and adjacency estimates in
Lemma~\ref{lem:sublinear-block-events} give
\begin{align}
w^{\mathsf T}C_Ww
&=
w^{\mathsf T}D_{G^-}w
-w^{\mathsf T}(pA_{H^-}+Z)w
\notag\\
&\geq
\left(\alpha\mu+p-\lVert Z\rVert_{\mathrm{op}}\right)
\lVert w\rVert_2^2
\notag\\
&\geq
\frac{\alpha\mu}{2}\lVert w\rVert_2^2
\label{eq:C-coercivity-proof}
\end{align}
for all sufficiently large $n$, because
$\lVert Z\rVert_{\mathrm{op}}\leq A\sqrt\mu=o(\mu)$. This proves
\eqref{eq:C-coercivity}.

\smallskip\noindent
\emph{The coupling.}
For $i<j$ with $1\leq d_q(i,j)\leq R$, let $Z_{ij}$ be the $B_i\times B_j$ submatrix
of $Z$, and let $\mathbf 1_{B_i}\in\mathbb R^{s_i}$ be the all-ones vector. Take
$u=\mathsf Ua$. Since $u$ is block-constant, edges within one block contribute nothing to
$u^{\mathsf T}L_{G^-}w$. Between $B_i$ and $B_j$ with $1\leq d_q(i,j)\leq R$, the
expected edge weights are constant, and the zero coordinate sum of each $w_i$ makes
their contribution vanish as well. Only the centered matrices $Z_{ij}$ remain:

\begin{align}
a^{\mathsf T}Nw
&=
u^{\mathsf T}L_{G^-}w
\notag\\
&=
\sum_{\substack{0\leq i<j<q\\1\leq d_q(i,j)\leq R}}
(b_i-b_j)
\left(
\left\langle Z_{ij}\mathbf 1_{B_j},w_i\right\rangle
-
\left\langle Z_{ij}^{\mathsf T}\mathbf 1_{B_i},w_j\right\rangle
\right).
\label{eq:sublinear-U-W-bilinear-form}
\end{align}
Using $\lVert Z_{ij}\rVert_{\mathrm{op}}\leq\lVert Z\rVert_{\mathrm{op}}$ and the
uniform relation $s_i/\bar s\to1$, regroup the terms by block to obtain
\begin{align}
\left|a^{\mathsf T}Nw\right|
&\leq
C_1\lVert Z\rVert_{\mathrm{op}}\sqrt{\bar s}
\sum_{i=0}^{q-1}\lVert w_i\rVert_2
\sum_{\substack{0\leq j<q\\1\leq d_q(i,j)\leq R}}
|b_i-b_j|.
\label{eq:sublinear-U-W-first-bound}
\end{align}
For each $i$, Cauchy--Schwarz gives
\begin{equation}
\label{eq:sublinear-neighbour-difference-bound}
\left(
\sum_{\substack{0\leq j<q\\1\leq d_q(i,j)\leq R}}
|b_i-b_j|
\right)^2
\leq
2R
\sum_{\substack{0\leq j<q\\1\leq d_q(i,j)\leq R}}
(b_i-b_j)^2.
\end{equation}
A second application of Cauchy--Schwarz, now over the block index, together with
$\sum_i\lVert w_i\rVert_2^2=\lVert w\rVert_2^2$, gives
\begin{align}
\left|a^{\mathsf T}Nw\right|
&\leq
C_2\lVert Z\rVert_{\mathrm{op}}\sqrt{\bar sR}\,
\lVert w\rVert_2
\left(
\sum_{i=0}^{q-1}
\sum_{\substack{0\leq j<q\\1\leq d_q(i,j)\leq R}}
(b_i-b_j)^2
\right)^{1/2}.
\label{eq:sublinear-U-W-second-bound}
\end{align}
The balanced block sizes give
\begin{equation}
\label{eq:M0-unweighted-lower}
a^{\mathsf T}M_0a
\geq
\bigl(1-o(1)\bigr)\frac{p\bar s^2}{2}
\sum_{i=0}^{q-1}
\sum_{\substack{0\leq j<q\\1\leq d_q(i,j)\leq R}}
(b_i-b_j)^2.
\end{equation}
By \eqref{eq:block-adjacency-event}, $\lVert Z\rVert_{\mathrm{op}}^2\leq A^2\mu$.
Also,
\begin{equation}
\label{eq:sublinear-Schur-scale}
\frac{\mu R}{p\bar s}
=
\frac{2mR}{\bar s}
=
\bigl(2+o(1)\bigr)K^2.
\end{equation}
Squaring \eqref{eq:sublinear-U-W-second-bound} and using
\eqref{eq:M0-unweighted-lower}--\eqref{eq:sublinear-Schur-scale}, we obtain
\begin{equation}
\label{eq:N-bound}
\left\lVert N^{\mathsf T}a\right\rVert_2^2
\leq
C_3K^2a^{\mathsf T}M_0a.
\end{equation}

\smallskip\noindent
\emph{The Schur correction.}
The lower bound \eqref{eq:C-coercivity} gives
\begin{align}
a^{\mathsf T}NC_W^{-1}N^{\mathsf T}a
&\leq
\frac{2}{\alpha\mu}
\left\lVert N^{\mathsf T}a\right\rVert_2^2
\notag\\
&\leq
C_4\frac{K^2}{\mu}a^{\mathsf T}M_0a.
\label{eq:Schur-correction-bound}
\end{align}
Since
\begin{equation}
\label{eq:sublinear-Schur-error-vanishes}
\frac{K^2}{\mu}
\leq
\frac{1+o(1)}{c\sqrt{\log n}}
\longrightarrow0,
\end{equation}
combining \eqref{eq:M-M0-comparison} with \eqref{eq:Schur-correction-bound} proves
\eqref{eq:Schur-comparison}. We may take
$\zeta_n:=\varepsilon_n+C_4K^2/(c\log n)$, which is deterministic and uniform in $p$.
\end{proof}

We can now reduce the lower bound to the Poincar\'e inequality on the quotient cycle.
Together with the Fourier upper bound, this gives the sharp gap when $m=o(n)$.

\begin{proposition}[Sharp spectral gap when $m=o(n)$]
\label{prop:sublinear-sharp-gap}
Under the assumptions of Theorem~\ref{thm:spectral-gap}, suppose in addition that
$m=o(n)$. Then, for every $\varepsilon>0$,
\begin{equation}
\label{eq:sublinear-sharp-gap-uniform}
\sup_{p\in[c\log n/(2m),1]}
\Prob_p\left(
G\text{ is disconnected}
\ \text{or}\
\left|
\frac{\gamma_{\lazy}(G)}{\Gamma_{n,m}}-1
\right|>\varepsilon
\right)
\longrightarrow0.
\end{equation}
\end{proposition}

\begin{proof}
Let $\mathcal E_n^{\mathrm{blk}}$ be the event in
\eqref{eq:block-degree-events}--\eqref{eq:block-edge-event}, and set
\begin{equation}
\label{eq:sublinear-working-event}
\mathcal E_n
:=
\mathcal E_n^{\mathrm{blk}}
\cap
\{G\text{ is connected}\}
\cap
\left\{
\gamma_{\lazy}(G)\leq2\Gamma_{n,m}
\right\}.
\end{equation}
Lemma~\ref{lem:sublinear-block-events},
\eqref{eq:connectivity-supercritical-uniform}, and
Proposition~\ref{prop:sharp-gap-upper} show that
$\inf_{p\in[c\log n/(2m),1]}\Prob_p(\mathcal E_n)\to1$.
The last condition gives an a priori energy bound for a minimizing eigenfunction.

The constants $C_5,\ldots,C_9$ depend at most on $c$, and all $o(1)$ terms below are
deterministic and uniform over $p\in[c\log n/(2m),1]$. Set $V_G:=\vol_G(\mathbb Z_n)$.

\emph{A low-energy eigenfunction.}
On $\mathcal E_n$, choose $f$ attaining the variational minimum in
Lemma~\ref{lem:variational-gap} for simple random walk, normalized by
\begin{equation}
\label{eq:eigenfunction-normalization}
f^{\mathsf T}D_Gf=V_G,
\qquad
f^{\mathsf T}D_G\mathbf 1=0.
\end{equation}
By \eqref{eq:graph-dirichlet-edge-formula} and \eqref{eq:graph-lazy-gap-definition},
$f^{\mathsf T}L_Gf=2V_G\gamma_{\lazy}(G)$. Summing \eqref{eq:block-volume-event} and
using $\sum_i s_i=n$ gives
\begin{equation}
\label{eq:sublinear-total-volume}
V_G
=
\bigl(1+o(1)\bigr)n\mu
\end{equation}
on $\mathcal E_n$. The definition of $\mathcal E_n$ and
\eqref{eq:Gamma-uniform-bounds} therefore imply
\begin{align}
f^{\mathsf T}L_Gf
&=
2V_G\gamma_{\lazy}(G)
\notag\\
&\leq
C_5\mu\frac{m^2}{n}.
\label{eq:sublinear-low-energy}
\end{align}

\smallskip\noindent
\emph{Within-block fluctuations are negligible.}
Write $f=u+w$ in $U\oplus W$, with $u=\mathsf Ua$. Complete the square in
\eqref{eq:block-Laplacian-matrix}:
\begin{align}
f^{\mathsf T}L_{G^-}f
&=
\left(w+C_W^{-1}N^{\mathsf T}a\right)^{\mathsf T}
C_W
\left(w+C_W^{-1}N^{\mathsf T}a\right)
\notag\\
&\quad+
a^{\mathsf T}
\left(M-NC_W^{-1}N^{\mathsf T}\right)a.
\label{eq:block-square-completion}
\end{align}
Since $L_G\succeq L_{G^-}$, Proposition~\ref{prop:block-Schur} and
\eqref{eq:N-bound} imply
\begin{align}
a^{\mathsf T}M_0a
&\leq
C_6 f^{\mathsf T}L_Gf,
\label{eq:M0-energy-control}\\
\left\lVert w+C_W^{-1}N^{\mathsf T}a\right\rVert_2^2
&\leq
\frac{C_7}{\mu}f^{\mathsf T}L_Gf,
\label{eq:shifted-w-control}\\
\left\lVert C_W^{-1}N^{\mathsf T}a\right\rVert_2^2
&\leq
C_8\frac{K^2}{\mu^2}f^{\mathsf T}L_Gf.
\label{eq:shift-control}
\end{align}
For the last estimate, use \eqref{eq:C-coercivity} and \eqref{eq:N-bound} to write
\begin{align*}
\left\lVert C_W^{-1}N^{\mathsf T}a\right\rVert_2^2
&\leq
\frac{2}{\alpha\mu}
\left(C_W^{-1}N^{\mathsf T}a\right)^{\mathsf T}
C_W
\left(C_W^{-1}N^{\mathsf T}a\right)\\
&=
\frac{2}{\alpha\mu}
a^{\mathsf T}NC_W^{-1}N^{\mathsf T}a\\
&\leq
\frac{4}{\alpha^2\mu^2}
\left\lVert N^{\mathsf T}a\right\rVert_2^2\\
&\leq
C\frac{K^2}{\mu^2}a^{\mathsf T}M_0a
\leq
C\frac{K^2}{\mu^2}f^{\mathsf T}L_Gf.
\end{align*}
The last inequality is \eqref{eq:M0-energy-control}. Combining
\eqref{eq:shifted-w-control}--\eqref{eq:shift-control} with
\eqref{eq:sublinear-low-energy} gives
\begin{equation}
\label{eq:w-vanishes}
\frac{\lVert w\rVert_2^2}{n}
\leq
C_9\frac{m^2}{n^2}
\left(1+\frac{K^2}{\mu}\right)
=o(1).
\end{equation}

The degree bound \eqref{eq:block-degree-events} and the estimates
\eqref{eq:sublinear-total-volume} and \eqref{eq:w-vanishes} then give
\begin{equation}
\label{eq:w-degree-mass-vanishes}
\frac{w^{\mathsf T}D_Gw}{V_G}
\leq
\frac{A\mu\lVert w\rVert_2^2}
{\bigl(1+o(1)\bigr)n\mu}
=o(1).
\end{equation}
The triangle inequality in the norm induced by $D_G$ gives
\begin{equation}
\label{eq:D-norm-triangle}
\left|
\sqrt{u^{\mathsf T}D_Gu}
-
\sqrt{f^{\mathsf T}D_Gf}
\right|
\leq
\sqrt{w^{\mathsf T}D_Gw}.
\end{equation}
The centering condition \eqref{eq:eigenfunction-normalization} and Cauchy--Schwarz also give
\begin{align}
\left|u^{\mathsf T}D_G\mathbf 1\right|
&=
\left|w^{\mathsf T}D_G\mathbf 1\right|
\notag\\
&\leq
\sqrt{w^{\mathsf T}D_Gw}\sqrt{\mathbf 1^{\mathsf T}D_G\mathbf 1}
\notag\\
&=
\sqrt{w^{\mathsf T}D_Gw}\sqrt{V_G}.
\label{eq:block-constant-mean-bound}
\end{align}
\eqref{eq:eigenfunction-normalization}--\eqref{eq:block-constant-mean-bound}
therefore imply
\begin{equation}
\label{eq:u-mass-mean}
\frac{u^{\mathsf T}D_Gu}{V_G}=1+o(1),
\qquad
\frac{u^{\mathsf T}D_G\mathbf 1}{V_G}=o(1).
\end{equation}

\smallskip\noindent
\emph{Pass to the quotient cycle.}
Recall that $u=b_i$ on $B_i$. The block-volume estimate, \eqref{eq:sublinear-total-volume},
and \eqref{eq:u-mass-mean} give $\sum_i s_ib_i^2=O(n)$. Replacing block volumes by their
means therefore incurs a quadratic error of at most
$\varepsilon_n\mu\sum_i s_ib_i^2=o(n\mu)$. The linear error satisfies
\begin{align}
\left|
\sum_{i=0}^{q-1}
\bigl(\vol_G(B_i)-\mu s_i\bigr)b_i
\right|
&\leq
\varepsilon_n\mu
\sum_{i=0}^{q-1}s_i|b_i|
\notag\\
&\leq
\varepsilon_n\mu
\left(
n\sum_{i=0}^{q-1}s_ib_i^2
\right)^{1/2}
=o(n\mu).
\label{eq:block-linear-moment-error}
\end{align}
Both errors are $o(n\mu)$. Thus \eqref{eq:sublinear-total-volume} and
\eqref{eq:u-mass-mean} imply
\begin{equation}
\label{eq:block-b-moments}
\frac1n\sum_{i=0}^{q-1}s_ib_i^2=1+o(1),
\qquad
\frac1n\sum_{i=0}^{q-1}s_ib_i=o(1).
\end{equation}

Since $\max_i|s_i/\bar s-1|\to0$, these weighted relations imply
\begin{equation}
\label{eq:unweighted-b-variance}
\frac1q\sum_{i=0}^{q-1}(b_i-\bar b)^2=1+o(1),
\qquad
\bar b:=\frac1q\sum_{i=0}^{q-1}b_i=o(1).
\end{equation}
Indeed, the second moments differ by a relative $o(1)$, and the means differ by at most
$\max_i|s_i/\bar s-1|\,(q^{-1}\sum_i b_i^2)^{1/2}=o(1)$.

Apply the Schur-complement comparison and Lemma~\ref{lem:finite-cycle-Poincare}, using
the balanced block sizes, to obtain
\begin{align}
f^{\mathsf T}L_Gf
&\geq
\bigl(1-o(1)\bigr)a^{\mathsf T}M_0a
\notag\\
&\geq
\bigl(1-o(1)\bigr)p\bar s^2
4Rg_{q,R}
\sum_{i=0}^{q-1}(b_i-\bar b)^2
\notag\\
&=
\bigl(1-o(1)\bigr)4p\bar s^2Rq\,g_{q,R}.
\label{eq:sublinear-energy-lower}
\end{align}
Dividing by $2V_G=(1+o(1))4pmn$ and using $n=q\bar s$ gives
\begin{equation}
\label{eq:sublinear-gap-quotient-lower}
\gamma_{\lazy}(G)
\geq
\bigl(1-o(1)\bigr)\frac{R\bar s}{m}g_{q,R}.
\end{equation}
The block definitions give $R\bar s/m\to1$ and $R/q\sim m/n$.
Lemma~\ref{lem:sublinear-Fourier-gap} therefore gives
\begin{equation}
\label{eq:sublinear-gap-lower-final}
\gamma_{\lazy}(G)
\geq
\left(\frac{\pi^2}{3}-o(1)\right)\frac{m^2}{n^2}
=
\bigl(1-o(1)\bigr)\Gamma_{n,m}
\end{equation}
with high probability, uniformly in $p$. Together with
Proposition~\ref{prop:sharp-gap-upper}, this proves \eqref{eq:sublinear-sharp-gap-uniform}.
The explicit constant in Corollary~\ref{cor:explicit-spectral-gap} follows from
\eqref{eq:Gamma-sublinear-asymptotic}.
\end{proof}

\subsection[Spectral gap lower bound when m/n is bounded below]
{Spectral gap lower bound when $m/n$ is bounded below}
\label{subsec:gap-macroscopic-lower}

Suppose now that $m$ is bounded below by a positive multiple of $n$. Every column of
$\mathsf P_{n,m}$ has Euclidean norm $(2m)^{-1/2}=O(n^{-1/2})$. This small column norm,
together with an $\ell^2$ bound on the relative degree deviations, makes the error from
degree normalization $o_{\Prob}(1)$ in operator norm. The next proposition gives the
comparison.

\begin{proposition}[Normalized-adjacency comparison with the cycle power]
\label{prop:macroscopic-operator-comparison}
Fix $c>1$ and $\rho_0\in(0,1/2)$, and suppose that
\begin{equation}
\label{eq:macroscopic-comparison-bandwidth}
\rho_0n\leq m<\frac n2.
\end{equation}
For every $\varepsilon>0$,
\begin{align}
&\sup_{p\in[c\log n/(2m),1]}
\biggl[
\Prob_p\left(\min_x d_G(x)=0\right)
\notag\\
&\qquad\qquad+
\Prob_p\left(
\min_x d_G(x)>0,\
\left\lVert
D_G^{-1/2}A_GD_G^{-1/2}-\mathsf P_{n,m}
\right\rVert_{\mathrm{op}}
>\varepsilon
\right)
\biggr]
\longrightarrow0.
\label{eq:macroscopic-operator-comparison}
\end{align}
\end{proposition}

\begin{proof}
We estimate the degree-normalization error and the centered adjacency error separately.
Set $Z_G:=A_G-pA_H$ and $\mathcal D_G:=D_G/\mu$. Lemma~\ref{lem:degrees-fixed-c} gives
\begin{equation}
\label{eq:macroscopic-minimum-degree}
\mathcal D_G\succeq d_-(c)I
\end{equation}
with high probability, uniformly in $p$. Since each $d_G(x)$ is binomial with variance
$\mu(1-p)$,
$$
\E\sum_{x\in\mathbb Z_n}
\left(\frac{d_G(x)}{\mu}-1\right)^2
=\frac{n(1-p)}{\mu}\leq\frac n\mu.
$$
Markov's inequality gives
\begin{equation}
\label{eq:macroscopic-degree-L2-event}
\sum_{x\in\mathbb Z_n}
\left(\frac{d_G(x)}{\mu}-1\right)^2
\leq
\frac{n}{\sqrt\mu}
\end{equation}
with failure probability at most $\mu^{-1/2}$, uniformly in $p$.

The function $y\mapsto y^{-1/2}$ is Lipschitz on $[d_-(c),\infty)$, with constant depending
only on $c$. Thus, on the intersection of \eqref{eq:macroscopic-minimum-degree} and
\eqref{eq:macroscopic-degree-L2-event}, the diagonal matrix
$R_G:=\mathcal D_G^{-1/2}-I$ satisfies
\begin{equation}
\label{eq:macroscopic-degree-normalization-bounds}
\sum_x R_G(x,x)^2
\leq
C_1(c)\frac n{\sqrt\mu},
\qquad
\left\lVert R_G\right\rVert_{\mathrm{op}}
\leq
C_2(c).
\end{equation}
Write $\lVert M\rVert_{\mathrm{HS}}^2:=\sum_{x,y}|M(x,y)|^2$. The columns of
$\mathsf P_{n,m}=A_H/(2m)$ have squared Euclidean norm $(2m)^{-1}$, so the diagonal form
of $R_G$ gives
\begin{equation}
\label{eq:macroscopic-smoothing-HS-identity}
\left\lVert \mathsf P_{n,m}R_G\right\rVert_{\mathrm{HS}}^2
=
\frac1{2m}\sum_xR_G(x,x)^2.
\end{equation}
Since the operator norm is at most the Hilbert--Schmidt norm and
$R_G\mathsf P_{n,m}=(\mathsf P_{n,m}R_G)^{\mathsf T}$, the bounds
\eqref{eq:macroscopic-comparison-bandwidth}, \eqref{eq:macroscopic-degree-normalization-bounds},
and \eqref{eq:macroscopic-smoothing-HS-identity} imply
\begin{equation}
\label{eq:macroscopic-smoothed-degree-error}
\left\lVert R_G\mathsf P_{n,m}\right\rVert_{\mathrm{op}}
=
\left\lVert \mathsf P_{n,m}R_G\right\rVert_{\mathrm{op}}
\leq
C_3(c,\rho_0)\mu^{-1/4}
=o(1).
\end{equation}
Expanding $\mathcal D_G^{-1/2}=I+R_G$ now gives
\begin{align}
&
\left\lVert
\mathcal D_G^{-1/2}\mathsf P_{n,m}\mathcal D_G^{-1/2}
-\mathsf P_{n,m}
\right\rVert_{\mathrm{op}}
\notag\\
&\qquad\leq
\left(
2+\left\lVert R_G\right\rVert_{\mathrm{op}}
\right)
\left\lVert \mathsf P_{n,m}R_G\right\rVert_{\mathrm{op}}
=o(1).
\label{eq:macroscopic-mean-operator-comparison}
\end{align}

It remains to bound the centered adjacency error. The row-variance sums of $Z_G$ are
at most $\mu$. Apply Lemma~\ref{lem:centered-adjacency-norm} with $B=2$ and use
$\mu\geq c\log n$ to obtain, with probability at least $1-n^{-2}$,
\begin{equation}
\label{eq:macroscopic-centered-adjacency-bound}
\left\lVert Z_G\right\rVert_{\mathrm{op}}
\leq
C_4(c)\sqrt\mu.
\end{equation}
On the event in \eqref{eq:macroscopic-minimum-degree},
\begin{equation}
\label{eq:macroscopic-normalized-noise-bound}
\left\lVert
\mathcal D_G^{-1/2}\frac{Z_G}{\mu}\mathcal D_G^{-1/2}
\right\rVert_{\mathrm{op}}
\leq
\frac{C_4(c)}{d_-(c)\sqrt\mu}
=o(1).
\end{equation}
Finally, the identity $\mu=2mp$ gives
\begin{equation}
\label{eq:macroscopic-normalized-adjacency-decomposition}
D_G^{-1/2}A_GD_G^{-1/2}
=
\mathcal D_G^{-1/2}
\left(
\mathsf P_{n,m}+\frac{Z_G}{\mu}
\right)
\mathcal D_G^{-1/2}.
\end{equation}
Combining \eqref{eq:macroscopic-mean-operator-comparison},
\eqref{eq:macroscopic-normalized-noise-bound}, and
\eqref{eq:macroscopic-normalized-adjacency-decomposition} proves
\eqref{eq:macroscopic-operator-comparison}. The errors in
\eqref{eq:macroscopic-smoothed-degree-error} and \eqref{eq:macroscopic-normalized-noise-bound}
vanish uniformly since $\mu\geq c\log n$. The failure probabilities in
\eqref{eq:macroscopic-degree-L2-event} and \eqref{eq:macroscopic-centered-adjacency-bound}
are uniform as well.
\end{proof}

When $m/n$ stays bounded away from zero, so does the lazy-walk spectral gap of $C_n^m$.
Weyl's inequality therefore converts Proposition~\ref{prop:macroscopic-operator-comparison}
into the required relative gap estimate.

\begin{proposition}[Sharp spectral gap when $m/n$ has a positive limit]
\label{prop:macroscopic-sharp-gap}
Under the assumptions of Theorem~\ref{thm:spectral-gap}, suppose in addition that
\begin{equation}
\label{eq:macroscopic-sharp-gap-bandwidth}
\frac mn\longrightarrow\rho\in(0,1/2].
\end{equation}
Then, for every $\varepsilon>0$,
\begin{equation}
\label{eq:macroscopic-sharp-gap-uniform}
\sup_{p\in[c\log n/(2m),1]}
\Prob_p\left(
G\text{ is disconnected}
\ \text{or}\
\left|
\frac{\gamma_{\lazy}(G)}{\Gamma_{n,m}}-1
\right|>\varepsilon
\right)
\longrightarrow0.
\end{equation}
\end{proposition}

\begin{proof}
By \eqref{eq:connectivity-supercritical-uniform}, $G$ is connected with high probability,
uniformly in $p$. For large $n$, apply Proposition~\ref{prop:macroscopic-operator-comparison}
with $\rho_0=\rho/2$. On the connectedness event, set
$\mathcal A_G:=D_G^{-1/2}A_GD_G^{-1/2}$. If $\lambda_j(M)$ denotes the $j$th largest
eigenvalue of a real symmetric matrix $M$, Weyl's inequality gives
\begin{equation}
\label{eq:macroscopic-Weyl-comparison}
\max_{1\leq j\leq n}
\left|
\lambda_j(\mathcal A_G)-\lambda_j(\mathsf P_{n,m})
\right|
\leq
\left\lVert\mathcal A_G-\mathsf P_{n,m}\right\rVert_{\mathrm{op}}
=o_{\Prob}(1)
\end{equation}
uniformly in $p$.

On the connectedness event,
\begin{align}
\left|
\gamma_{\lazy}(G)-\gamma_{\lazy}(C_n^m)
\right|
&=
\frac12
\left|
\lambda_2(\mathcal A_G)-\lambda_2(\mathsf P_{n,m})
\right|
\notag\\
&\leq
\frac12
\left\lVert\mathcal A_G-\mathsf P_{n,m}\right\rVert_{\mathrm{op}}.
\label{eq:macroscopic-gap-additive-comparison}
\end{align}
Lemma~\ref{lem:macroscopic-Fourier-gap} also gives
\begin{equation}
\label{eq:macroscopic-reference-positive-limit}
\Gamma_{n,m}
\longrightarrow
\Gamma_\rho
:=
\frac12\left(
1-\frac{\sin(2\pi\rho)}{2\pi\rho}
\right)
>0
\end{equation}
and
$\gamma_{\lazy}(C_n^m)/\Gamma_{n,m}\to1$. Dividing
\eqref{eq:macroscopic-gap-additive-comparison} by $\Gamma_{n,m}$ gives
\begin{align}
\left|
\frac{\gamma_{\lazy}(G)}{\Gamma_{n,m}}-1
\right|
&\leq
\frac{
\left\lVert\mathcal A_G-\mathsf P_{n,m}\right\rVert_{\mathrm{op}}
}{2\Gamma_{n,m}}
\notag\\
&\quad+
\left|
\frac{\gamma_{\lazy}(C_n^m)}{\Gamma_{n,m}}-1
\right|.
\label{eq:macroscopic-gap-relative-comparison}
\end{align}
The first term tends to zero in probability uniformly in $p$; the second tends to zero
deterministically. This proves \eqref{eq:macroscopic-sharp-gap-uniform}.
\end{proof}

We finish by combining Proposition~\ref{prop:sublinear-sharp-gap}, for $m=o(n)$, with
Proposition~\ref{prop:macroscopic-sharp-gap}, for $m/n$ tending to a positive limit.

\begin{proof}[Proof of Theorem~\ref{thm:spectral-gap}]
For $\varepsilon>0$, let $\mathcal B_{n,p}(\varepsilon)$ be the event
$$
G\text{ is disconnected}
\ \text{or}\
\left|
\frac{\gamma_{\lazy}(G)}{\Gamma_{n,m}}-1
\right|>\varepsilon.
$$
Fix any sequence $p_n\in[c\log n/(2m),1]$. Every subsequence has a further subsequence
on which $m/n\to\rho$ for some $\rho\in[0,1/2]$. If $\rho=0$, apply
Proposition~\ref{prop:sublinear-sharp-gap}; if $\rho>0$, apply
Proposition~\ref{prop:macroscopic-sharp-gap}. In either case,
$\Prob_{p_n}(\mathcal B_{n,p_n}(\varepsilon))\to0$ on the further subsequence. Hence
$\Prob_{p_n}(\mathcal B_{n,p_n}(\varepsilon))\to0$ for every admissible sequence $(p_n)$.

This also proves uniformity. Otherwise, along a subsequence one could choose $p_n$ whose
failure probability is at least half the supremum in \eqref{eq:gap-main-uniform-conclusion},
contradicting the convergence just proved for every admissible sequence $(p_n)$.
\end{proof}

\begin{proof}[Proof of Corollary~\ref{cor:explicit-spectral-gap}]
Combine Theorem~\ref{thm:spectral-gap} with
\eqref{eq:Gamma-sublinear-asymptotic} and
\eqref{eq:Gamma-macroscopic-asymptotic}.
\end{proof}

Finally, combine Theorem~\ref{thm:spectral-gap} with
Proposition~\ref{prop:host-spectral-gap} to obtain \eqref{eq:gap-host-comparison}.
On the connectedness event, which has high probability,
$$
\frac{\gamma_{\lazy}(G)}{\gamma_{\lazy}(C_n^m)}
=
\frac{\gamma_{\lazy}(G)}{\Gamma_{n,m}}
\frac{\Gamma_{n,m}}{\gamma_{\lazy}(C_n^m)},
$$
and both factors on the right converge in probability to one.

\section{Sharp hitting-time asymptotics}
\label{sec:hitting-time}

Fix $c>1$ and assume \eqref{eq:sharp-hitting-assumptions}. All estimates in this section are
uniform for $p\in[c\log n/(2m),1]$. Our aim is to separate the two contributions to the hitting time: the target-degree term $V_G/d_G(y)$ and the spatial term $r_{xy}(n-r_{xy})/\kappa_m$. Indeed, we prove
Theorem~\ref{thm:sharp-hitting} by extracting $V_G/d_G(y)$ exactly from the hitting-time equation and estimating the remainder through a deterministic cycle power on the blocks.

\subsection{Deterministic cycle powers}

We first collect the deterministic estimates needed for the block reduction. In this
subsection, let $1\leq m<n/2$ be arbitrary integers. We will later apply these estimates
to the quotient cycle power $C_q^R$.

All hitting times here refer to simple random walk. The volume of $C_n^m$ is $2mn$, and
$r_{ab}$ denotes cyclic distance, as in Subsection~\ref{sec:band-random-graphs-model}.
For a real symmetric matrix $Q$, let $Q^+$ denote its Moore--Penrose pseudoinverse.
The transition eigenvalues of simple random walk on $C_n^m$ are
\begin{equation}
\label{eq:cycle-transition-eigenvalues}
 \theta_{k,n}
 :=
 \frac1m\sum_{r=1}^m
 \cos\left(\frac{2\pi kr}{n}\right),
 \qquad 0\leq k<n.
\end{equation}

To control the Fourier formula for hitting times, we need a quadratic lower bound on
$1-\theta_{k,n}$ at low frequencies and a bound on the contribution of the remaining frequencies.
The following estimates provide both.

\begin{lemma}[Fourier bounds for the transition eigenvalues of $C_n^m$]
\label{lem:fourier-estimates}
There are absolute constants $a,A>0$ such that, for
$1\leq k\leq n/2$,
\begin{align}
\label{eq:fourier-gap-lower}
 1-\theta_{k,n}
 &\geq
 a\min\left\{1,\frac{m^2k^2}{n^2}\right\},\\
\label{eq:fourier-coefficient-upper}
 |\theta_{k,n}|
 &\leq
 A\min\left\{1,\frac{n}{mk}\right\}.
\end{align}
Consequently,
\begin{equation}
\label{eq:fourier-ratio-upper}
 \left|\frac{\theta_{k,n}}{1-\theta_{k,n}}\right|
 \leq
 A
 \begin{cases}
 \displaystyle\frac{n^2}{m^2k^2},
 &1\leq k\leq n/m,\\[1.2ex]
 \displaystyle\frac{n}{mk},
 &n/m<k\leq n/2.
 \end{cases}
\end{equation}
\end{lemma}

\begin{proof}
From \eqref{eq:cycle-transition-eigenvalues},
$$
 1-\theta_{k,n}
 =
 \frac2m\sum_{r=1}^m
 \sin^2\left(\frac{\pi kr}{n}\right).
$$
If $mk/n\leq1/4$, then $|\sin t|\geq 2|t|/\pi$ applies to every term
with $1\leq r\leq m$. Hence
$$
 1-\theta_{k,n}
 \geq
 \frac{8k^2}{mn^2}\sum_{r=1}^m r^2
 \geq a\frac{m^2k^2}{n^2}.
$$
If $mk/n>1/4$, put $x=\pi k/n\in(0,\pi/2]$ and use
\begin{equation}
\label{eq:sine-square-sum}
 \sum_{r=1}^m\sin^2(rx)
 =
 \frac m2-
 \frac{\sin(mx)\cos((m+1)x)}{2\sin x}.
\end{equation}
When $m\sin x\geq2$, the right-hand side is at least $m/4$.
When $m\sin x<2$, the inequalities
$2x/\pi\leq\sin x\leq x$ show that
$$
 \frac{\pi}{4}<mx<\pi.
$$
For $\lceil m/4\rceil\leq r\leq\lfloor m/2\rfloor$, we consequently
have
$$
 \frac{\pi}{16}<rx<\frac{\pi}{2}.
$$
For sufficiently large $m$, there are at least $m/5$ such indices, contributing at least
$\frac m5\sin^2(\pi/16)$. For each of the remaining finitely many values of $m$, the
continuous function
$$
x\longmapsto \frac1m\sum_{r=1}^m\sin^2(rx)
$$
has a positive minimum on
$\{x\in[0,\pi/2]:mx\geq\pi/4\}$. Decreasing $a$ gives
$$
 \frac1m\sum_{r=1}^m\sin^2(rx)\geq a
$$
with an absolute $a>0$, proving \eqref{eq:fourier-gap-lower}.

The geometric-sum identity gives
$$
 \left|\sum_{r=1}^m e^{2\pi ikr/n}\right|
 \leq
 \min\left\{m,\frac1{|\sin(\pi k/n)|}\right\}.
$$
For $1\leq k\leq n/2$,
$\sin(\pi k/n)\geq2k/n$, which proves
\eqref{eq:fourier-coefficient-upper}. Combining the two estimates proves
\eqref{eq:fourier-ratio-upper}.
\end{proof}

We now evaluate the Fourier formula uniformly in the two vertices. The low frequencies
produce the quadratic distance term; Lemma~\ref{lem:fourier-estimates} controls the tail.

\begin{lemma}[Uniform deterministic hitting-time asymptotics]
\label{lem:deterministic-hitting}
Suppose that $m\to\infty$ and $m<n/2$. Uniformly in distinct
$a,b\in\Z_n$,
\begin{equation}
\label{eq:deterministic-hitting-asymptotic}
 H_{C_n^m}(a,b)
 =
 n+
 \frac{3r_{ab}(n-r_{ab})}{m^2}
 +
 o\left(n+\frac{n^2}{m^2}\right).
\end{equation}
Moreover, for an absolute constant $C$,
\begin{align}
\label{eq:deterministic-hitting-global-upper}
 \max_{a,b}H_{C_n^m}(a,b)
 &\leq
 C\left(n+\frac{n^2}{m^2}\right),\\
\label{eq:deterministic-hitting-local-upper}
 \max_{r_{ab}\leq m+4}H_{C_n^m}(a,b)
 &\leq Cn.
\end{align}
\end{lemma}

\begin{proof}
We first derive the Fourier formula, then estimate its low frequencies and tail.
Fix $b\in\mathbb Z_n$ and put $h_b(x):=H_{C_n^m}(x,b)$. Since $C_n^m$ is regular, the
hitting-time equations give
$$
L_{C_n^m}h_b
=
2m\mathbf 1-2mn e_b
=
2mn(\mathfrak u_n-e_b),
$$
where $\mathfrak u_n:=n^{-1}\mathbf 1$. Since constants span the kernel of
$L_{C_n^m}$, for every $a\neq b$,
$$
H_{C_n^m}(a,b)
=
(e_a-e_b)^{\mathsf T}L_{C_n^m}^+\,2mn(\mathfrak u_n-e_b).
$$
Let
$$
\psi_k(x):=n^{-1/2}e^{2\pi ikx/n},
\qquad 0\leq k<n.
$$
The vectors $\psi_k$ form an orthonormal eigenbasis of $L_{C_n^m}$, with eigenvalues
$2m(1-\theta_{k,n})$. Since $L_{C_n^m}^+\mathfrak u_n=0$, expanding in this basis and
pairing $k$ with $n-k$ gives
\begin{equation}
\label{eq:exact-cycle-hitting-fourier}
 H_{C_n^m}(a,b)
 =
 \sum_{k=1}^{n-1}
 \frac{1-\cos(2\pi k r_{ab}/n)}{1-\theta_{k,n}}.
\end{equation}
Since
$$
 \sum_{k=1}^{n-1}
 \left(1-\cos\left(\frac{2\pi k r_{ab}}{n}\right)\right)=n,
$$
we may rewrite \eqref{eq:exact-cycle-hitting-fourier} as
\begin{equation}
\label{eq:cycle-hitting-base-plus-correction}
 H_{C_n^m}(a,b)
 =
 n+
 \sum_{k=1}^{n-1}
 \left(1-\cos\left(\frac{2\pi k r_{ab}}{n}\right)\right)
 \frac{\theta_{k,n}}{1-\theta_{k,n}}.
\end{equation}
Below, we always pair $k$ with $n-k$; in frequency ranges, $k$ denotes
$\min\{k,n-k\}\in[1,n/2]$.

\medskip
\noindent\emph{The quadratic term.}
First consider a subsequence on which $m/\sqrt n$ is bounded. For fixed $k\geq1$,
Taylor expansion gives
\begin{align*}
 1-\theta_{k,n}
 &=
 \frac1m\sum_{r=1}^m
 \left[
 \frac12\left(\frac{2\pi kr}{n}\right)^2
 +O\left(\frac{k^4r^4}{n^4}\right)
 \right]\\
 &=
 \frac{2\pi^2k^2}{mn^2}\sum_{r=1}^m r^2
 +
 O\left(\frac{k^4}{mn^4}\sum_{r=1}^m r^4\right)\\
 &=
 \frac{2\pi^2k^2m^2}{3n^2}\bigl(1+o(1)\bigr).
\end{align*}
Here $m\to\infty$, and the relative error is $O(k^2m^2/n^2)=o(1)$. Thus
$\theta_{k,n}\to1$ and
\begin{equation}
\label{eq:fixed-fourier-ratio-limit}
 \frac{m^2}{n^2}\frac{\theta_{k,n}}{1-\theta_{k,n}}
 \longrightarrow
 \frac{3}{2\pi^2k^2}.
\end{equation}
By \eqref{eq:fourier-ratio-upper}, the paired contribution from
$J<\min\{k,n-k\}\leq n/2$, after multiplication by $m^2/n^2$, is at most
\begin{equation}
\label{eq:fourier-tail-control}
 C\sum_{J<k\leq n/m}\frac1{k^2}
 +
 C\frac mn\sum_{n/m<k\leq n/2}\frac1k
 \leq
 \frac{C}{J}+C\frac{m}{n}\log m.
\end{equation}
The second term tends to zero when $m=O(\sqrt n)$. We may therefore let $n\to\infty$
and then $J\to\infty$. The tail bound is independent of $r_{ab}$, so convergence is
uniform over distinct $a,b$. The Fourier identity
\begin{equation}
\label{eq:quadratic-distance-fourier-series}
 \sum_{k=1}^{\infty}
 \frac{1-\cos(2\pi kt)}{k^2}
 =
 \pi^2t(1-t),
 \qquad 0\leq t\leq1,
\end{equation}
now yields
$$
 \frac{m^2}{n^2}
 \left(H_{C_n^m}(a,b)-n\right)
 =
 3\frac{r_{ab}}n\left(1-\frac{r_{ab}}n\right)+o(1)
$$
uniformly over distinct $a,b$. This proves
\eqref{eq:deterministic-hitting-asymptotic} when
$m=O(\sqrt n)$.

On a subsequence for which $m/\sqrt n\to\infty$,
Lemma~\ref{lem:fourier-estimates} gives
\begin{equation}
\label{eq:fast-cycle-correction-bound}
 \max_{a\ne b}\left|H_{C_n^m}(a,b)-n\right|
 \leq
 C\left(
 \frac{n^2}{m^2}
 +
 \frac nm\log m
 \right)
 =
 o(n).
\end{equation}
The quadratic term in \eqref{eq:deterministic-hitting-asymptotic} is also $o(n)$,
which proves the expansion when $m/\sqrt n\to\infty$. Every subsequence has a further
subsequence of one of these two types, so the expansion holds for the full sequence.

\medskip
\noindent\emph{Global and local bounds.}
Equations~\eqref{eq:cycle-hitting-base-plus-correction} and
\eqref{eq:fourier-ratio-upper} also give
$$
 \max_{a\ne b}\left|H_{C_n^m}(a,b)-n\right|
 \leq
 C\left(
 \frac{n^2}{m^2}\sum_{k\geq1}\frac1{k^2}
 +
 \frac nm\sum_{n/m<k\leq n/2}\frac1k
 \right)
 \leq
 C\left(\frac{n^2}{m^2}+\frac nm\log m\right).
$$
Since $\log m\leq m$, this is at most $C(n+n^2/m^2)$, proving
\eqref{eq:deterministic-hitting-global-upper}.

For the local bound, suppose that $r_{ab}\leq m+4$. The frequencies $k\leq n/m$ in
\eqref{eq:exact-cycle-hitting-fourier} contribute at most
$$
 C\sum_{k\leq n/m}
 \frac{\min\{1,k^2r_{ab}^2/n^2\}}
 {m^2k^2/n^2}
 \leq
 C\frac{nr_{ab}^2}{m^3}
 \leq Cn.
$$
For $k>n/m$, the denominator is bounded away from zero, and the numerators sum to at
most $Cn$. This proves \eqref{eq:deterministic-hitting-local-upper}.
\end{proof}

We will use effective resistance to turn energy estimates into pointwise estimates for
the block equation. The commute-time identity converts the preceding hitting-time bounds
into the resistance bounds we need.

\begin{corollary}[Resistance bounds for a cycle power]
\label{cor:cycle-resistance}
Suppose that $m\to\infty$ and $m<n/2$, and let $\mathcal R_{n,m}$ be the maximum effective
resistance in the unweighted graph $C_n^m$. Then
\begin{equation}
\label{eq:cycle-resistance-global}
 \mathcal R_{n,m}
 \leq
 C\left(\frac1m+\frac{n}{m^3}\right).
\end{equation}
If $r_{ab}\leq m+4$, then
\begin{equation}
\label{eq:cycle-resistance-local}
 R_{\eff}^{C_n^m}(a,b)\leq\frac Cm.
\end{equation}
\end{corollary}

\begin{proof}
The map
$$
x\longmapsto a+b-x\pmod n
$$
is an automorphism of $C_n^m$ that interchanges $a$ and $b$. Hence
$H_{C_n^m}(a,b)=H_{C_n^m}(b,a)$. Theorem~\ref{thm:commute-time} and
$\vol(C_n^m)=2mn$ give
$$
 R_{\eff}^{C_n^m}(a,b)
 =
 \frac{H_{C_n^m}(a,b)}{mn}.
$$
\eqref{eq:cycle-resistance-global} and
\eqref{eq:cycle-resistance-local} follow from the corresponding bounds in
Lemma~\ref{lem:deterministic-hitting}.
\end{proof}

The block reduction retains all interactions up to block distance $R$. Edges of the
original graph may also join a bounded number of layers beyond this range. The next
lemma shows that these layers contribute only $O(R^{-1})$ times the energy from
distances $1,\ldots,R$.

\begin{lemma}[Boundary layers of a cycle power]
\label{lem:boundary-layer-energy}
For every fixed integer $J\geq1$, there is $C_J<\infty$ such that, for
$R\geq2$, $R<q/2$, and every $z\in\R^q$,
\begin{equation}
\label{eq:boundary-layer-energy}
 \sum_{\substack{1\leq a\leq J\\R+a\leq\lfloor q/2\rfloor}}
 \sum_{i=0}^{q-1}
 (z_i-z_{i+R+a})^2
 \leq
 \frac{C_J}{R}
 \sum_{r=1}^{R}\sum_{i=0}^{q-1}
 (z_i-z_{i+r})^2.
\end{equation}
All indices are read modulo $q$.
\end{lemma}

\begin{proof}
Write
$$
 E_t(z):=\sum_{i=0}^{q-1}(z_i-z_{i+t})^2,
 \qquad
 \mathcal E_R(z):=\sum_{t=1}^R E_t(z).
$$
Fix $R+1\leq s\leq\lfloor q/2\rfloor$. For $1\leq r\leq R$,
$$
 (z_i-z_{i+s})^2
 \leq
 2(z_i-z_{i+r})^2
 +
 2(z_{i+r}-z_{i+s})^2.
$$
Summing in $i$ and averaging in $r$ gives
\begin{equation}
\label{eq:boundary-layer-recursion}
 E_s(z)
 \leq
 \frac2R\mathcal E_R(z)
 +
 \frac2R\sum_{r=1}^R E_{s-r}(z).
\end{equation}
For $s=R+1$, every index in the second sum lies in
$\{1,\ldots,R\}$, proving
$E_{R+1}(z)\leq4\mathcal E_R(z)/R$.
Suppose inductively that
$E_{R+b}(z)\leq C_b\mathcal E_R(z)/R$ for $1\leq b<a$.
For $s=R+a$, the second sum in \eqref{eq:boundary-layer-recursion} has lengths in
$\{a,\ldots,R+a-1\}$. Those at most $R$ contribute at most $\mathcal E_R(z)$;
the induction hypothesis controls the remaining $a-1$ terms. Since $a\leq J$ is fixed,
this gives $E_{R+a}(z)\leq C_a\mathcal E_R(z)/R$. Summing over
$R+a\leq\lfloor q/2\rfloor$ proves the lemma.
\end{proof}

The next lemma supplies two tools for passing from quadratic forms to potentials.
The first bounds a potential in terms of the source mass and effective resistance.
The second controls the change in the inverse when the Dirichlet form changes by a
small relative error.

\begin{lemma}[Resistance duality and stability of inverses]
\label{lem:resistance-duality-stability}
Let $Q$ be a positive semidefinite $q\times q$ matrix with
$\ker Q=\operatorname{span}\{r\}$, where every coordinate of $r$ is
positive. For any positive semidefinite matrix $T$, write
$\|z\|_T:=(z^{\mathsf T}Tz)^{1/2}$. Put
$$
 \ell_{ij}:=\frac{e_i}{r_i}-\frac{e_j}{r_j},
 \qquad
 \mathcal R_Q(i,j):=\ell_{ij}^{\mathsf T}Q^+\ell_{ij},
 \qquad
 \mathcal R_Q^{\max}:=\max_{i,j}\mathcal R_Q(i,j).
$$
If $B\in\R^q$ satisfies $\sum_iB_i=0$ and
$$
 g:=\sum_{i=0}^{q-1}B_i\frac{e_i}{r_i},
 \qquad
 \mathsf m(B):=\sum_i(B_i)_+,
$$
then
\begin{align}
\label{eq:source-resistance-dual-bound}
 g^{\mathsf T}Q^+g
 &\leq
 \mathsf m(B)^2\mathcal R_Q^{\max},\\
\label{eq:resistance-cauchy-schwarz}
 \left|\ell_{ij}^{\mathsf T}Q^+g\right|
 &\leq
 \sqrt{\mathcal R_Q(i,j)}
 \sqrt{g^{\mathsf T}Q^+g}.
\end{align}

Suppose, in addition, that $\mathcal S$ is positive semidefinite,
$\ker\mathcal S=\ker Q$, and, for some $0\leq\delta<1$,
\begin{equation}
\label{eq:abstract-form-comparison}
 (1-\delta)Q\preceq\mathcal S\preceq(1+\delta)Q.
\end{equation}
Then every $h\perp r$ satisfies
\begin{align}
\label{eq:abstract-inverse-stability}
 \left\|\mathcal S^+h-Q^+h\right\|_Q
 &\leq
 \frac{\delta}{1-\delta}\left\|h\right\|_{Q^+},\\
\label{eq:abstract-inverse-bound}
 \left\|\mathcal S^+h\right\|_Q
 &\leq
 \frac1{1-\delta}\left\|h\right\|_{Q^+}.
\end{align}
\end{lemma}

\begin{proof}
Choose nonnegative numbers $\gamma_{ij}$, with $B_i>0$ and $B_j<0$,
whose row sums are $(B_i)_+$ and whose column sums are $(-B_j)_+$.
Then
$$
 g=\sum_{\substack{i:B_i>0\\j:B_j<0}}\gamma_{ij}\ell_{ij},
 \qquad
 \sum_{i,j}\gamma_{ij}=\mathsf m(B).
$$
The triangle inequality for the seminorm induced by $Q^+$ gives
$$
 \|g\|_{Q^+}
 \leq
 \sum_{i,j}\gamma_{ij}\|\ell_{ij}\|_{Q^+}
 \leq
 \mathsf m(B)\sqrt{\mathcal R_Q^{\max}},
$$
which proves \eqref{eq:source-resistance-dual-bound}.
\eqref{eq:resistance-cauchy-schwarz} is the
Cauchy--Schwarz inequality for the same inner product.

For the inverse estimates, restrict to $r^\perp$, where $Q$ and $\mathcal S$ are positive
definite, and set
$$
 \mathcal A:=Q^{-1/2}\mathcal S Q^{-1/2}.
$$
\eqref{eq:abstract-form-comparison} implies
$\operatorname{spec}(\mathcal A)\subseteq[1-\delta,1+\delta]$.
Consequently,
$$
 \|\mathcal A^{-1}-I\|_{\op}
 \leq\frac{\delta}{1-\delta},
 \qquad
 \|\mathcal A^{-1}\|_{\op}
 \leq\frac1{1-\delta}.
$$
Apply these inequalities to $Q^{-1/2}h$. On $r^\perp$, we have
$$
\begin{aligned}
\|\mathcal S^{-1}h-Q^{-1}h\|_Q
&=
\|(\mathcal A^{-1}-I)Q^{-1/2}h\|_2,\\
\|\mathcal S^{-1}h\|_Q
&=
\|\mathcal A^{-1}Q^{-1/2}h\|_2.
\end{aligned}
$$
These identities prove \eqref{eq:abstract-inverse-stability} and
\eqref{eq:abstract-inverse-bound}; the pseudoinverses give the same statements on $\R^q$.
\end{proof}

\subsection{Exact decomposition of the hitting-time equation}

We next remove the target-degree term $V_G/d(y)$ exactly. The remainder satisfies a
Laplacian equation with a mean-zero source, which is the form needed for the block
argument.

For this step, let $G$ be any connected finite graph. Write $A,D$, and $L=D-A$ for its
adjacency, degree, and combinatorial Laplacian matrices, and put
$V_G:=\sum_{x\in V(G)}d(x)$.

\begin{lemma}[Exact extraction of the target-degree term]
\label{lem:exact-local-removal}
Fix a target $y$. Define
\begin{equation}
\label{eq:local-step-function}
 \phi_y(x):=
 \begin{cases}
 V_G/d(y),&x\neq y,\\
 0,&x=y,
 \end{cases}
\end{equation}
and put $u_y(x):=H_G(x,y)-\phi_y(x)$. Then
\begin{equation}
\label{eq:remainder-poisson-equation}
 Lu_y=b_y,
 \qquad
 b_y:=d-\frac{V_G}{d(y)}Ae_y,
 \qquad
 \sum_xb_y(x)=0,
\end{equation}
where $d=(d(x))_{x\in V(G)}$. In particular, for $x\neq y$,
\begin{equation}
\label{eq:exact-target-decomposition}
 H_G(x,y)=\frac{V_G}{d(y)}+u_y(x)-u_y(y).
\end{equation}
\end{lemma}

\begin{proof}
The hitting-time equations give
$$
 (LH_G(\,\cdot\,,y))(x)=d(x)
 \quad\text{for }x\neq y.
$$
The coordinates of $LH_G(\,\cdot\,,y)$ sum to zero, so its value at $y$ is
$d(y)-V_G$. Also,
$$
 (L\phi_y)(x)
 =
 \frac{V_G}{d(y)}A(x,y)
 \quad (x\neq y),
 \qquad
 (L\phi_y)(y)=-V_G.
$$
Subtracting gives \eqref{eq:remainder-poisson-equation}. Since $Ae_y$ has coordinate sum
$d(y)$, the source $b_y$ has sum zero. By definition, $u_y(y)=0$, and thus
$H_G(x,y)=V_G/d(y)+u_y(x)$ for $x\neq y$. Any other solution of $Lu=b_y$ differs by a
constant, so the potential difference $u_y(x)-u_y(y)$ gives
\eqref{eq:exact-target-decomposition} for every choice of solution.
\end{proof}

\subsection{Block decomposition of the Laplacian}

Return to $G(n,p,m)$ under \eqref{eq:sharp-hitting-assumptions}. We need blocks short
enough that their interaction range tends to infinity, but large enough for edge counts
and volumes to concentrate even at $p=c\log n/(2m)$. Set
\begin{equation}
\label{eq:block-parameters}
 K:=\left\lfloor(\log n)^{1/4}\right\rfloor,
 \qquad
 q:=\left\lceil\frac{nK}{m}\right\rceil.
\end{equation}
There will be $q$ blocks of size asymptotic to $m/K$, with interaction range asymptotic
to $K$. Since $nK/m\geq2K\to\infty$ and $K/m\to0$, rounding $q$ introduces a relative
error $o(1)$, and $q/n\to0$. In particular, the blocks are nonempty for large $n$.
With $t_i=\lfloor in/q\rfloor$, partition $\Z_n$ into
$$
 B_i:=\{t_i,t_i+1,\ldots,t_{i+1}-1\},
 \qquad 0\leq i<q.
$$
Write
$$
 s_i:=|B_i|,
 \qquad
 \bar s:=\frac nq,
 \qquad
 s_-:=\min_i s_i,
 \qquad
 s_+:=\max_i s_i,
$$
and set
\begin{equation}
\label{eq:block-range}
 R:=\left\lfloor\frac{m}{s_+}\right\rfloor-1.
\end{equation}
Then
\begin{equation}
\label{eq:block-parameter-asymptotics}
 \bar s\sim\frac mK,
 \qquad
R\sim K,
 \qquad
 \frac{s_+}{s_-}\longrightarrow1,
 \qquad
 \frac{R\bar s}{m}\longrightarrow1.
\end{equation}
Thus $R\to\infty$, and
$$
 R<\frac m{s_+}<\frac n{2s_+}\leq\frac q2,
$$
so $R<q/2$ for all sufficiently large $n$.

We separate block averages from fluctuations within blocks. Let $U\subseteq\R^n$ be
the space of block-constant functions, and put $W:=U^\perp$. Thus every function in
$W$ sums to zero on each block. Identify $\R^q$ isometrically with $U$ by
\begin{equation}
\label{eq:block-isometry-note}
 (\mathsf Ua)(x):=\frac{a_i}{\sqrt{s_i}},
 \qquad x\in B_i.
\end{equation}
Relative to $U\oplus W$, write
\begin{equation}
\label{eq:full-laplacian-block-matrix}
 L_G=
 \begin{pmatrix}
 M&N\\
 N^{\mathsf T}&C_W
 \end{pmatrix}.
\end{equation}
Define the deterministic block form
\begin{equation}
\label{eq:reference-block-form}
 a^{\mathsf T}M_0a
 :=
 p\sum_{\substack{0\leq i<j<q\\
   1\leq d_q(i,j)\leq R}}
 s_is_j
 \left(
 \frac{a_i}{\sqrt{s_i}}-\frac{a_j}{\sqrt{s_j}}
 \right)^2.
\end{equation}

The form $M_0$ describes a weighted cycle power on the blocks. The next lemma collects
the estimates needed to replace the random equation by this deterministic form.
Degree and volume bounds control the source. Coercivity of $C_W$ controls fluctuations
within blocks, while the Schur-complement estimate controls the block averages. Together,
the latter two bounds also give $\ker L_G=\operatorname{span}\{\1\}$ and hence
connectedness.

\begin{lemma}[Uniform estimates for the block decomposition]
\label{lem:uniform-block-event}
There are constants $\alpha=\alpha(c)>0$ and $A=A(c)<\infty$, a
deterministic sequence $\kappa_n\to0$, and events
$\mathcal B_{n,p}$ satisfying
$$
 \inf_{p\in[c\log n/(2m),1]}\Prob_p(\mathcal B_{n,p})
 \longrightarrow1
$$
such that the following statements hold on $\mathcal B_{n,p}$:
\begin{align}
\label{eq:block-degree-bounds-note}
 \alpha\mu\leq d_G(x)\leq A\mu
 &\quad\text{for every }x,
 \qquad \mu:=2mp,\\
\label{eq:block-volume-concentration-note}
 \max_i\left|
 \frac{\vol_G(B_i)}{\mu s_i}-1
 \right|
 &\leq\kappa_n,\\
\label{eq:block-W-coercivity}
 C_W&\succeq\alpha\mu I_W,\\
\label{eq:block-Schur-comparison-hitting}
 (1-\kappa_n)M_0
 \preceq
 M-NC_W^{-1}N^{\mathsf T}
 &\preceq
 (1+\kappa_n)M_0,\\
\label{eq:block-coupling-comparison}
 NC_W^{-1}N^{\mathsf T}
 &\preceq\kappa_nM_0.
\end{align}
The constants and $\kappa_n$ are independent of $p$. In particular, $G$ is connected on
$\mathcal B_{n,p}$ for all sufficiently large $n$.
\end{lemma}

\begin{proof}
We first prove the estimates for a slightly smaller host graph, then control the omitted
boundary layers.

\medskip
\noindent\emph{A host graph adapted to the blocks.}
Let $H^-$ contain all edges within blocks and all edges between blocks at cyclic block
distance at most $R$. If $1\leq d_q(i,j)\leq R$, the cyclic distance between
$x\in B_i$ and $y\in B_j$ is at most
$$
 (R+1)s_+-1
 =
 \left\lfloor\frac m{s_+}\right\rfloor s_+-1
 <m.
$$
The within-block edges are also admissible because $s_+-1<m$. Thus
$H^-\subseteq C_n^m$; retaining its edges with the same variables gives $G^-\subseteq G$.

The block parameters satisfy
\begin{equation}
\label{eq:hitting-block-required-conditions}
s_-\sim s_+\sim\frac mK,\qquad
R\sim K,\qquad R<\frac q2,\qquad
\mu\geq c\log n,\qquad \frac{m}{K^2}\longrightarrow\infty.
\end{equation}
Moreover, the degrees in $H^-$ satisfy
$$
 (2R+1)s_--1\leq d_{H^-}(x)\leq2m,
 \qquad
 \min_x d_{H^-}(x)=(1-o(1))\,2m.
$$
\medskip
\noindent\emph{Concentration estimates.}
Choose $\alpha=\alpha(c)>0$ with $cI_-(3\alpha)>1$. Since
$p d_{H^-}(x)\geq(1-o(1))\mu$, Chernoff's bound gives, uniformly in $p$,
$$
\Prob_p\bigl(d_{G^-}(x)<2\alpha\mu\bigr)
\leq
\exp\bigl(-(1-o(1))\mu I_-(3\alpha)\bigr).
$$
A union bound gives $\min_xd_{G^-}(x)\geq2\alpha\mu$ with probability tending to one,
uniformly in $p$. The upper degree bound follows from Lemma~\ref{lem:degrees-fixed-c},
and $d_G(x)\geq d_{G^-}(x)$ gives the lower bound in
\eqref{eq:block-degree-bounds-note}, after reducing $\alpha$ if necessary.

The volume $\vol_G(B_i)$ is a sum of independent Bernoulli variables with coefficients
one or two. Its mean is $\mu s_i$ and its variance is at most $C\mu s_i$, so Bernstein's
inequality gives, for $0<\varepsilon\leq1$,
$$
\Prob_p\left(
\left|\vol_G(B_i)-\mu s_i\right|>\varepsilon\mu s_i
\right)
\leq
2\exp(-a\varepsilon^2\mu s_i).
$$
For $1\leq d_q(i,j)\leq R$, every pair in $B_i\times B_j$ is admissible in $H^-$.
Thus $e_{G^-}(B_i,B_j)$ is binomial with mean $ps_is_j$, and
$$
\Prob_p\left(
\left|e_{G^-}(B_i,B_j)-ps_is_j\right|>\varepsilon ps_is_j
\right)
\leq
2\exp(-a\varepsilon^2ps_is_j).
$$
The relations
$$
 \frac{\mu s_-}{\log n}\longrightarrow\infty,
 \qquad
 \frac{ps_-^2}{\log n}
 \geq(1+o(1))\frac{cm}{2K^2}\longrightarrow\infty
$$
follow from \eqref{eq:block-parameter-asymptotics} and $m/\log n\to\infty$.
Since $q\leq n$, choose a deterministic $\tau_n\to0$ slowly enough to take union bounds
over all blocks and block pairs at distance at most $R$. Also, the row-variance sums of
$A_{G^-}-pA_{H^-}$ are at most $\mu$. Lemma~\ref{lem:centered-adjacency-norm}, with a
fixed polynomial failure probability, therefore gives the adjacency estimate below.
For a constant $A=A(c)$, the following hold simultaneously with probability tending to one,
uniformly in $p$:
\begin{align}
\label{eq:minus-degree-bound}
 \min_x d_{G^-}(x)&\geq2\alpha\mu,\\
\label{eq:minus-adjacency-bound}
 \norm{A_{G^-}-pA_{H^-}}_{\op}&\leq A\sqrt\mu,\\
\label{eq:preliminary-block-volume-concentration}
 \max_i\left|
 \frac{\vol_G(B_i)}{\mu s_i}-1
 \right|&\leq\tau_n,\\
\label{eq:short-range-block-edge-concentration}
 \left|
 \frac{e_{G^-}(B_i,B_j)}{ps_is_j}-1
 \right|
 &\leq\tau_n
 \quad (1\leq d_q(i,j)\leq R),
\end{align}
These hold together with \eqref{eq:block-degree-bounds-note}; the edge-count estimate is
simultaneous over all indicated block pairs.

Let
$$
 L_{G^-}=
 \begin{pmatrix}
 M^-&N^-\\
 (N^-)^{\mathsf T}&C_W^-
 \end{pmatrix}.
$$
\medskip
\noindent\emph{Coercivity and the Schur complement.}
Put
$$
Z^-:=A_{G^-}-pA_{H^-}.
$$
If $w\in W$, every off-diagonal block of $A_{H^-}$ annihilates $w$: these blocks are
either zero or all-ones matrices, and $w$ sums to zero on each block. The diagonal block
on $B_i$ is $J_{s_i}-I_{s_i}$, so $A_{H^-}w=-w$. It follows from
\eqref{eq:minus-degree-bound}--\eqref{eq:minus-adjacency-bound} that
$$
\begin{aligned}
w^{\mathsf T}C_W^-w
&=w^{\mathsf T}D_{G^-}w-w^{\mathsf T}(pA_{H^-}+Z^-)w\\
&\geq
\left(2\alpha\mu+p-\lVert Z^-\rVert_{\op}\right)\lVert w\rVert_2^2
\geq
\alpha\mu\lVert w\rVert_2^2
\end{aligned}
$$
for all sufficiently large $n$. Thus $C_W^-\succeq\alpha\mu I_W$.

For the coupling term, let $u=\mathsf Ua$ and $b_i:=a_i/\sqrt{s_i}$. Again, the
deterministic all-ones blocks cancel against the zero block sums of $w$, leaving only the
rectangular blocks of $Z^-$ in $u^{\mathsf T}L_{G^-}w$. Applying Cauchy--Schwarz twice
gives
$$
|a^{\mathsf T}N^-w|
\leq
C\lVert Z^-\rVert_{\op}\sqrt{s_+R}\,\lVert w\rVert_2
\left(
\sum_{i=0}^{q-1}
\sum_{\substack{0\leq j<q\\1\leq d_q(i,j)\leq R}}
(b_i-b_j)^2
\right)^{1/2}.
$$
Since
$$
a^{\mathsf T}M_0a
\geq
\frac{ps_-^2}{2}
\sum_{i=0}^{q-1}
\sum_{\substack{0\leq j<q\\1\leq d_q(i,j)\leq R}}
(b_i-b_j)^2,
$$
\eqref{eq:minus-adjacency-bound} implies
$$
\lVert (N^-)^{\mathsf T}a\rVert_2^2
\leq
C\frac{\mu s_+R}{ps_-^2}a^{\mathsf T}M_0a
\leq
CK^2a^{\mathsf T}M_0a,
$$
where the final inequality follows from \eqref{eq:hitting-block-required-conditions} and
$\mu=2mp$. Therefore
$$
a^{\mathsf T}N^-(C_W^-)^{-1}(N^-)^{\mathsf T}a
\leq
C\frac{K^2}{\mu}a^{\mathsf T}M_0a
=o(1)a^{\mathsf T}M_0a.
$$
On the other hand, \eqref{eq:short-range-block-edge-concentration} gives
$$
(1-\tau_n)M_0\preceq M^-\preceq(1+\tau_n)M_0.
$$
Combining the last two estimates yields
\begin{equation}
\label{eq:subgraph-Schur-comparison}
 (1-o(1))M_0
 \preceq
 M^--N^-(C_W^-)^{-1}(N^-)^{\mathsf T}
 \preceq
 (1+o(1))M_0.
\end{equation}
All errors are deterministic and uniform in $p$. Restricting $L_G\succeq L_{G^-}$ to $W$
gives $C_W\succeq C_W^-\succeq\alpha\mu I_W$, proving \eqref{eq:block-W-coercivity}.

\medskip
\noindent\emph{The boundary layers.}
We now obtain the upper comparison for the full graph. If an edge of $C_n^m$ joins
blocks at cyclic block distance $h$, then $(h-1)s_-\leq m$, so
$$
 h\leq\left\lceil\frac m{s_-}\right\rceil+1.
$$
Moreover,
$s_+-s_-\leq1$ and
$$
 m\left(\frac1{s_-}-\frac1{s_+}\right)
 \leq\frac{m}{s_-s_+}
 =
 O\left(\frac{K^2}{m}\right)=o(1).
$$
Together with $R=\lfloor m/s_+\rfloor-1$, these estimates show that
$h\leq R+4$ for all sufficiently large $n$.
For a block pair with at least one admissible edge, put
$h_{ij}:=e_{C_n^m}(B_i,B_j)\leq s_+^2$. Lemma~\ref{lem:bernstein} gives
$$
 \Prob_p\left(
 |e_G(B_i,B_j)-ph_{ij}|>\varepsilon ps_-^2
 \right)
 \leq
 2\exp(-a\varepsilon^2ps_-^2)
$$
for $\varepsilon\leq1$: the variance is at most $ps_+^2$, the deviation is
$\varepsilon ps_-^2$, and $s_+/s_-$ is bounded. Choose $\tau_n\to0$ slowly enough that
$\tau_n^2ps_-^2/\log n\to\infty$, while retaining
\eqref{eq:preliminary-block-volume-concentration} and
\eqref{eq:short-range-block-edge-concentration}. A union bound gives, for every interacting
block pair,
\begin{equation}
\label{eq:additive-block-edge-concentration}
 \left|
 e_G(B_i,B_j)-p\,e_{C_n^m}(B_i,B_j)
 \right|
 \leq
 \tau_np s_-^2.
\end{equation}
Let $\mathcal B_{n,p}$ be the intersection of the degree, volume,
centered-adjacency, and block-edge events used in
\eqref{eq:block-degree-bounds-note},
\eqref{eq:minus-degree-bound},
\eqref{eq:minus-adjacency-bound},
\eqref{eq:preliminary-block-volume-concentration},
\eqref{eq:short-range-block-edge-concentration}, and
\eqref{eq:additive-block-edge-concentration}. The six failure probabilities have sum
$o(1)$, so
$$
\inf_{p\in[c\log n/(2m),1]}\Prob_p(\mathcal B_{n,p})\longrightarrow1.
$$
We work on $\mathcal B_{n,p}$ for the remainder of the proof.

At block distances at most $R$, all $s_is_j$ pairs are admissible. At distances
$R<d_q(i,j)\leq R+4$, there are at most $s_+^2$ admissible pairs, so
\eqref{eq:additive-block-edge-concentration} gives, for large $n$,
$$
e_G(B_i,B_j)
\leq
p s_+^2+\tau_nps_-^2
\leq
2ps_+^2.
$$
For $a\in\mathbb R^q$, write
$$
z_i:=\frac{a_i}{\sqrt{s_i}},
\qquad 0\leq i<q.
$$
The function $\mathsf Ua$ equals $z_i$ on $B_i$ and has no energy on within-block edges.
Thus
$$
 a^{\mathsf T}Ma
 =
 \sum_{\substack{0\leq i<j<q\\d_q(i,j)\leq R+4}}
 e_G(B_i,B_j)(z_i-z_j)^2.
$$
Distances $d_q(i,j)\leq R$ contribute at most $(1+o(1))a^{\mathsf T}M_0a$.
By \eqref{eq:additive-block-edge-concentration} and the boundary-layer estimate
\eqref{eq:boundary-layer-energy}, the remaining terms are at most
\begin{align*}
 Cp s_+^2
 \sum_{\substack{1\leq a\leq4\\R+a\leq\lfloor q/2\rfloor}}
 \sum_{i=0}^{q-1}(z_i-z_{i+R+a})^2
 &\leq
 \frac{Cp s_+^2}{R}
 \sum_{r=1}^R\sum_{i=0}^{q-1}(z_i-z_{i+r})^2\\
 &\leq
 \frac CR\,a^{\mathsf T}M_0a.
\end{align*}
Consequently,
$$
 M\preceq
 \left(1+o(1)+O(R^{-1})\right)M_0
 =
 (1+o(1))M_0.
$$

The Schur complement has the variational characterization
$$
 a^{\mathsf T}(M-NC_W^{-1}N^{\mathsf T})a
 =
 \min_{w\in W}
 (\mathsf Ua+w)^{\mathsf T}L_G(\mathsf Ua+w).
$$
The same identity holds for $G^-$. Thus $L_G\succeq L_{G^-}$ and
\eqref{eq:subgraph-Schur-comparison} give the lower Schur-complement bound.
Choosing $w=0$ and using $M\preceq(1+o(1))M_0$ gives the upper bound. Finally,
$$
 NC_W^{-1}N^{\mathsf T}
 =
 M-(M-NC_W^{-1}N^{\mathsf T}),
$$
so the bounds for $M$ and its Schur complement give \eqref{eq:block-coupling-comparison}.
Choose $\kappa_n$ as a sufficiently large fixed multiple of the largest deterministic
error below:
$$
 \tau_n,\qquad
 R^{-1},\qquad \frac{K^2}{c\log n},\qquad
 \max_i\left|\frac{s_i}{\bar s}-1\right|,
 \qquad
 \left|\frac{R\bar s}{m}-1\right|.
$$
Each tends to zero independently of $p$. Since $\kappa_n\geq\tau_n$, the volume estimate
also holds with $\kappa_n$, giving
\eqref{eq:block-volume-concentration-note}--\eqref{eq:block-coupling-comparison} with a
single error sequence.

\medskip
\noindent\emph{Connectedness.}
Put $\mathfrak r=(\sqrt{s_0},\ldots,\sqrt{s_{q-1}})^{\mathsf T}$, so
$\mathsf U\mathfrak r=\1$. The weighted graph for $M_0$ contains the nearest-neighbour
cycle because $R\geq1$ and every $s_i$ is positive. Hence
$\ker M_0=\operatorname{span}\{\mathfrak r\}$. Also,
$N^{\mathsf T}\mathfrak r=0$ since $L_G\1=0$. Completing the square, for $a\in\R^q$
and $w\in W$, gives
\begin{align}
\label{eq:block-completed-square}
(\mathsf Ua+w)^{\mathsf T}L_G(\mathsf Ua+w)
&=
a^{\mathsf T}\mathcal Sa + \left(w+C_W^{-1}N^{\mathsf T}a\right)^{\mathsf T}
C_W\left(w+C_W^{-1}N^{\mathsf T}a\right),
\end{align}
where $\mathcal S=M-NC_W^{-1}N^{\mathsf T}$. By \eqref{eq:block-W-coercivity} and
\eqref{eq:block-Schur-comparison-hitting}, this vanishes exactly when
$a\in\operatorname{span}\{\mathfrak r\}$ and $w=0$. Thus
$\ker L_G=\operatorname{span}\{\1\}$, proving connectedness.
\end{proof}

To estimate potential differences between individual vertices, we need a bound on the maximal effective resistance of the reference block network. The following lemma gives this bound and the two estimates used to control the errors.

\begin{lemma}[Effective-resistance bounds for the reference block network]
\label{lem:block-resistance}
For block indices $i,j$, let
$$
 \mathcal R_0(i,j)
 :=
 \left(
 \frac{e_i}{\sqrt{s_i}}-\frac{e_j}{\sqrt{s_j}}
 \right)^{\mathsf T}
 M_0^+
 \left(
 \frac{e_i}{\sqrt{s_i}}-\frac{e_j}{\sqrt{s_j}}
 \right).
$$
Then
\begin{equation}
\label{eq:block-resistance-bound}
 \mathcal R_0^{\max}
 :=
 \max_{i,j}\mathcal R_0(i,j)
 \leq
 C\left(
 \frac{K}{pm^2}
 +
 \frac{n}{pm^3}
 \right).
\end{equation}
Furthermore,
\begin{align}
\label{eq:block-resistance-times-volume}
 V_G\mathcal R_0^{\max}&\leq CT_{n,m},\\
\label{eq:block-resistance-negligible-error}
 \frac{n^2\mathcal R_0^{\max}}{T_{n,m}^2}&=o(1)
\end{align}
on the degree and volume event in
Lemma~\ref{lem:uniform-block-event}.
\end{lemma}

\begin{proof}
In the coordinates $z_i=a_i/\sqrt{s_i}$, the form $M_0$ describes the network on $\Z_q$
with conductance $ps_is_j$ on each edge $\{i,j\}$ satisfying $1\leq d_q(i,j)\leq R$.
The vector $e_i/\sqrt{s_i}-e_j/\sqrt{s_j}$ injects unit current at $i$ and removes it
at $j$, so $\mathcal R_0(i,j)$ is the effective resistance in this network.
Each conductance is at least $ps_-^2$. Since $R\to\infty$ and $R<q/2$, Rayleigh
monotonicity and Corollary~\ref{cor:cycle-resistance} give
$$
 \mathcal R_0^{\max}
 \leq
 \frac{C}{ps_-^2}
 \left(\frac1R+\frac q{R^3}\right)
 \leq
 C\left(\frac{K}{pm^2}+\frac{n}{pm^3}\right).
$$
Since $V_G=\Theta(\mu n)=\Theta(pmn)$,
$$
 V_G\mathcal R_0^{\max}
 \leq
 C\left(\frac{nK}{m}+\frac{n^2}{m^2}\right)
 \leq CT_{n,m},
$$
because $K/m\to0$. Finally,
$p\geq c\log n/(2m)$ gives
$$
 \mathcal R_0^{\max}
 \leq
 C\left(
 \frac{K}{m\log n}
 +
 \frac{n}{m^2\log n}
 \right).
$$
We use $T_{n,m}\geq n$ when $m\geq\sqrt n$ and $T_{n,m}\geq n^2/m^2$ otherwise.
In the first case,
$$
 \frac{n^2\mathcal R_0^{\max}}{T_{n,m}^2}
 \leq
 C\left(
 \frac{K}{\sqrt n\log n}+\frac1{\log n}
 \right)
 =
 o(1).
$$
In the second case,
$$
 \frac{n^2\mathcal R_0^{\max}}{T_{n,m}^2}
 \leq
 C\left(
 \frac{Km^3}{n^2\log n}
 +
 \frac{m^2}{n\log n}
 \right)
 \leq
 C\left(
 \frac{K}{\sqrt n\log n}+\frac1{\log n}
 \right)
 =
 o(1).
$$
This proves
\eqref{eq:block-resistance-negligible-error}.
\end{proof}

\subsection{Reduction to a deterministic cycle power}

We now carry out the reduction in two steps. First, we replace the Schur complement by
$M_0$ and show that fluctuations within blocks contribute only $o(T_{n,m})$ to every
potential difference. Second, we replace the block-averaged source by a uniform measure
minus a point mass and solve the resulting equation on $C_q^R$.

\begin{lemma}[Removal of within-block fluctuations]
\label{lem:within-block-elimination}
On $\mathcal B_{n,p}$, fix $y\in\Z_n$ and let $u_y$ and $b_y$ be as in
Lemma~\ref{lem:exact-local-removal}. Decompose
$$
 u_y=\mathsf Ua+w,
 \qquad
 b_y=\mathsf U\beta+\eta,
 \qquad
 w,\eta\in W.
$$
Let $a_0:=M_0^+\beta$. Then, uniformly over targets $y\in B_j$ and vertices
$x\in B_i$,
\begin{equation}
\label{eq:within-block-elimination-conclusion}
 u_y(x)-u_y(y)
 =
 \frac{(a_0)_i}{\sqrt{s_i}}
 -
 \frac{(a_0)_j}{\sqrt{s_j}}
 +
 o(T_{n,m}).
\end{equation}
\end{lemma}

\begin{proof}
Work on $\mathcal B_{n,p}$, where $G$ is connected by Lemma~\ref{lem:uniform-block-event}.
We first bound the source, then control the Schur-complement error and the within-block
fluctuations.

\medskip
\noindent\emph{Source bounds.}
The degree estimates give
$$
 \frac{V_G}{d_G(y)}\leq C_cn,
 \qquad
 \alpha\mu n\leq V_G\leq A\mu n.
$$
Moreover,
\begin{align}
\label{eq:source-norm-bounds}
 \norm{b_y}_1&\leq2V_G,\\
 \norm{b_y}_2
 &\leq
 \norm{d}_2+\frac{V_G}{d_G(y)}\norm{Ae_y}_2
 \leq
 C_c\mu\sqrt n+C_cn\sqrt\mu
 \leq C_cn\sqrt\mu.
\end{align}
For the first bound, both nonnegative terms defining $b_y$ have mass $V_G$.
For the second, use $\norm{Ae_y}_2=\sqrt{d_G(y)}$, the degree bounds, and
$\mu\leq2m<n$. Orthogonal projection also gives $\norm\eta_2\leq\norm{b_y}_2$.

Put
$$
 \mathcal S:=M-NC_W^{-1}N^{\mathsf T},
 \qquad
 \widetilde\beta:=\beta-NC_W^{-1}\eta.
$$
\medskip
\noindent\emph{The block equation.}
Eliminating $w$ gives
\begin{equation}
\label{eq:block-poisson-equations}
 \mathcal Sa=\widetilde\beta,
 \qquad
 w=C_W^{-1}(\eta-N^{\mathsf T}a).
\end{equation}
Write
$\norm z_Q^2:=z^{\mathsf T}Qz$ and
$\norm h_{Q^+}^2:=h^{\mathsf T}Q^+h$, and let
$$
 \mathfrak r:=(\sqrt{s_0},\ldots,\sqrt{s_{q-1}})^{\mathsf T}.
$$
We have $\mathsf U\mathfrak r=\1$ and $\ker M_0=\operatorname{span}\{\mathfrak r\}$.
Since $L_G\1=0$, also $M\mathfrak r=0$ and $N^{\mathsf T}\mathfrak r=0$.
The Schur-complement comparison gives $\ker\mathcal S=\operatorname{span}\{\mathfrak r\}$.
Moreover,
$$
 \mathfrak r^{\mathsf T}\beta
 =
 \1^{\mathsf T}b_y=0,
 \qquad
 \mathfrak r^{\mathsf T}NC_W^{-1}\eta=0.
$$
Thus $\beta$, $NC_W^{-1}\eta$, and $\widetilde\beta$ lie in $\mathfrak r^\perp$.
For large $n$, we have $\kappa_n<1$, so Lemma~\ref{lem:resistance-duality-stability}
and \eqref{eq:block-Schur-comparison-hitting} give
\begin{equation}
\label{eq:inverse-form-comparison}
 \norm{\mathcal S^+g-M_0^+g}_{M_0}
 \leq
 C\kappa_n\norm g_{M_0^+}
 \qquad(g\perp\mathfrak r).
\end{equation}

\medskip
\noindent\emph{Energy bounds.}
We claim that the block-averaged source satisfies
\begin{equation}
\label{eq:beta-dual-bound}
 \norm\beta_{M_0^+}^2
 \leq
 C_cV_GT_{n,m}.
\end{equation}
Indeed, put
$$
 \mathfrak b_i:=\sum_{x\in B_i}b_y(x).
$$
Then $\sum_i\mathfrak b_i=0$ and
$$
 \beta=\sum_i\mathfrak b_i\frac{e_i}{\sqrt{s_i}},
 \qquad
 \sum_i(\mathfrak b_i)_+\leq\frac12\norm{b_y}_1\leq V_G.
$$
Lemma~\ref{lem:resistance-duality-stability} gives
$$
 \beta^{\mathsf T}M_0^+\beta
 \leq
 V_G^2\mathcal R_0^{\max},
$$
and \eqref{eq:beta-dual-bound} follows from
\eqref{eq:block-resistance-times-volume}.

For every block vector $z$,
\begin{align}
\label{eq:schur-source-error-duality}
 |z^{\mathsf T}NC_W^{-1}\eta|
 &\leq
 \left(z^{\mathsf T}NC_W^{-1}N^{\mathsf T}z\right)^{1/2}
 \left(\eta^{\mathsf T}C_W^{-1}\eta\right)^{1/2}\notag\\
 &\leq
 C_cn\sqrt{\kappa_n}\,\norm z_{M_0}.
\end{align}
The second inequality uses
$$
 \eta^{\mathsf T}C_W^{-1}\eta
 \leq
 \frac1{\alpha\mu}\norm\eta_2^2
 \leq
 C_cn^2
$$
and \eqref{eq:block-coupling-comparison}. Using the dual characterization
$$
 \norm h_{M_0^+}
 =
 \sup_{\substack{z\perp\mathfrak r\\\norm z_{M_0}>0}}
 \frac{|z^{\mathsf T}h|}{\norm z_{M_0}}
 \qquad(h\perp\mathfrak r)
$$
we obtain
\begin{equation}
\label{eq:schur-source-error-dual-norm}
 \norm{NC_W^{-1}\eta}_{M_0^+}
 \leq C_cn\sqrt{\kappa_n}.
\end{equation}

By \eqref{eq:block-poisson-equations},
$a=\mathcal S^+\widetilde\beta+\lambda\mathfrak r$ for some $\lambda\in\R$.
The inverse bound in Lemma~\ref{lem:resistance-duality-stability}, together with
\eqref{eq:beta-dual-bound} and \eqref{eq:schur-source-error-dual-norm}, yields
$$
 \norm a_{M_0}
 \leq
 C\norm{\widetilde\beta}_{M_0^+}
 \leq
 C_c\left(\sqrt{V_GT_{n,m}}+n\sqrt{\kappa_n}\right).
$$
Since $V_G\geq\alpha\mu n$ and $T_{n,m}\geq n$,
$$
\frac{n^2\kappa_n}{V_GT_{n,m}}
\leq
\frac{\kappa_n}{\alpha\mu}
\longrightarrow0.
$$
Consequently,
\begin{equation}
\label{eq:block-solution-energy}
 a^{\mathsf T}M_0a
 \leq
 C_cV_GT_{n,m}.
\end{equation}

\medskip
\noindent\emph{Within-block fluctuations.}
Coercivity, the coupling estimate, and the bounds on the source and block solution give
$$
 \norm{C_W^{-1}\eta}_2
 \leq
 \frac1{\alpha\mu}\norm\eta_2
 \leq
 C_c\frac n{\sqrt\mu},
$$
and
\begin{align*}
 \norm{C_W^{-1}N^{\mathsf T}a}_2^2
 &\leq
 \frac1{\alpha\mu}
 a^{\mathsf T}NC_W^{-1}N^{\mathsf T}a\\
 &\leq
 \frac{\kappa_n}{\alpha\mu}a^{\mathsf T}M_0a
 \leq
 C_c\kappa_n nT_{n,m}.
\end{align*}
It follows from the second equation in
\eqref{eq:block-poisson-equations} that
\begin{align}
\label{eq:w-norm-bound}
 \norm w_\infty
 &\leq\norm w_2\notag\\
 &\leq
 C_c\frac n{\sqrt\mu}
 +
 C_c\sqrt{\kappa_n nT_{n,m}}
 =
 o(T_{n,m}).
\end{align}
Here $\mu\geq c\log n$, $T_{n,m}\geq n$, and $\kappa_n\to0$.

Since $a_0=M_0^+\beta$,
$$
 a-a_0
 =
 (\mathcal S^+-M_0^+)\beta
 -
 \mathcal S^+NC_W^{-1}\eta
$$
up to a multiple of $\mathfrak r$. This multiple affects neither the $M_0$-seminorm nor
potential differences, since $M_0\mathfrak r=0$ and $\mathfrak r_i/\sqrt{s_i}=1$.
Equations~\eqref{eq:inverse-form-comparison} and \eqref{eq:schur-source-error-dual-norm},
with the inverse bound in Lemma~\ref{lem:resistance-duality-stability}, give
\begin{equation}
\label{eq:block-solution-comparison-energy}
 \norm{a-a_0}_{M_0}
 \leq
 C_c\kappa_n\sqrt{V_GT_{n,m}}
 +
 C_cn\sqrt{\kappa_n}.
\end{equation}

\medskip
\noindent\emph{Pointwise comparison.}
For $x\in B_i$ and $y\in B_j$, Cauchy--Schwarz in the resistance metric gives
\begin{align}
\label{eq:block-solution-comparison-pointwise}
 \left|
 \frac{a_i-(a_0)_i}{\sqrt{s_i}}
 -
 \frac{a_j-(a_0)_j}{\sqrt{s_j}}
 \right|
 &\leq
 \sqrt{\mathcal R_0^{\max}}\,
 \norm{a-a_0}_{M_0}\notag\\
 &=o(T_{n,m}).
\end{align}
Indeed, \eqref{eq:block-resistance-times-volume} bounds the first contribution from
\eqref{eq:block-solution-comparison-energy} by $C_c\kappa_nT_{n,m}$. The second,
divided by $T_{n,m}$, is at most
$$
 C_c\sqrt{\kappa_n}
 \left(
 \frac{n^2\mathcal R_0^{\max}}{T_{n,m}^2}
 \right)^{1/2}
 =
 o(1)
$$
by \eqref{eq:block-resistance-negligible-error}. Finally,
$|w(x)-w(y)|\leq2\norm w_\infty=o(T_{n,m})$. Together with
\eqref{eq:block-solution-comparison-pointwise}, this proves
\eqref{eq:within-block-elimination-conclusion}.
\end{proof}

It remains to solve $M_0a_0=\beta$. The source comes from two probability measures on
the blocks: the stationary mass and the distribution of the block reached in one step
from $y$. The first is asymptotically uniform; the second is supported within block
distance $R+4$ of the target block. We replace them by the uniform measure and a point
mass, respectively, and evaluate the resulting cycle-power equation.

\begin{lemma}[Identification of the block-averaged equation]
\label{lem:block-source-identification}
Let $x\in B_i$ and $y\in B_j$, and define $\beta$ by
$\mathsf U\beta=\operatorname{proj}_U b_y$, as in
Lemma~\ref{lem:within-block-elimination}. On $\mathcal B_{n,p}$, the vector
$a_0:=M_0^+\beta$ satisfies, uniformly in $x,y$,
\begin{equation}
\label{eq:block-source-identification-conclusion}
 \frac{(a_0)_i}{\sqrt{s_i}}
 -
 \frac{(a_0)_j}{\sqrt{s_j}}
 =
 \frac{3r_{xy}(n-r_{xy})}{m^2}
 +
 o(T_{n,m}).
\end{equation}
\end{lemma}

\begin{proof}
Write
\begin{equation}
\label{eq:block-probability-measures}
 \widehat\pi_\ell
 :=
 \frac{\vol_G(B_\ell)}{V_G},
 \qquad
 \nu_{y,\ell}
 :=
 \frac{e_G(y,B_\ell)}{d_G(y)}.
\end{equation}
Here $e_G(y,B_\ell)$ counts retained edges from $y$ to $B_\ell$. Both vectors are
probability measures on $\Z_q$. Summing \eqref{eq:remainder-poisson-equation} over
$B_\ell$ gives
\begin{equation}
\label{eq:block-source-mass}
 \mathfrak b_\ell
 :=
 \sum_{z\in B_\ell}b_y(z)
 =
 V_G(\widehat\pi_\ell-\nu_{y,\ell}).
\end{equation}
The projection defining $\beta$ gives $\beta_\ell=\mathfrak b_\ell/\sqrt{s_\ell}$.
Thus, in the coordinates $z_\ell=(a_0)_\ell/\sqrt{s_\ell}$, the equation $M_0a_0=\beta$
has quadratic form
$$
 p\sum_{\substack{k<\ell\\1\leq d_q(k,\ell)\leq R}}
 s_ks_\ell(z_k-z_\ell)^2
$$
and source
$\mathfrak b=(\mathfrak b_\ell)_{\ell\in\Z_q}$.

\medskip
\noindent\emph{Equalizing the conductances.}
We replace $ps_ks_\ell$ by $p\bar s^2$. Their relative difference is $o(1)$ uniformly
in $k$ and $\ell$; the inverse-stability estimate controls the resulting potential error.
In the block-value coordinates $z$, define $Q_s$ and $\overline Q$ by
$$
 z^{\mathsf T}Q_sz
 =
 p\sum_{\substack{k<\ell\\1\leq d_q(k,\ell)\leq R}}
 s_ks_\ell(z_k-z_\ell)^2,
 \qquad
 \overline Q:=p\bar s^2L_{C_q^R}.
$$
Since $\sum_\ell\mathfrak b_\ell=0$, the equation
$Q_sz=\mathfrak b$ is solvable and
$$
 z=Q_s^+\mathfrak b+\lambda\1
$$
for some $\lambda\in\R$. The additive constant cancels from every
difference $z_i-z_j$.
If
$\rho_n:=\max_\ell|s_\ell/\bar s-1|$, then
$\rho_n\leq1/\bar s=o(1)$ and
$$
 (1-C\rho_n)\overline Q
 \preceq Q_s\preceq
 (1+C\rho_n)\overline Q.
$$
For large $n$, $C\rho_n<1$, and both forms have kernel $\operatorname{span}\{\1\}$.
Their resistances are therefore uniformly comparable. Under
$a_\ell=\sqrt{s_\ell}z_\ell$, the resistance for $Q_s$ is exactly $\mathcal R_0$,
so both maximal resistances are at most $C\mathcal R_0^{\max}$.

Lemma~\ref{lem:resistance-duality-stability}, applied with
$Q=\overline Q$ and $\mathcal S=Q_s$ gives
$$
 \|(Q_s^+-\overline Q^+)\mathfrak b\|_{\overline Q}
 \leq
 C\rho_n\|\mathfrak b\|_{\overline Q^+}.
$$
The total positive mass of $\mathfrak b$ is at most $V_G$, so
\eqref{eq:source-resistance-dual-bound} gives
$$
 \|\mathfrak b\|_{\overline Q^+}
 \leq
 CV_G\sqrt{\mathcal R_0^{\max}}.
$$
Applying \eqref{eq:resistance-cauchy-schwarz} to the evaluation
$e_i-e_j$ now bounds the pointwise error by
$$
 C\rho_nV_G\mathcal R_0^{\max}
 =
 o(T_{n,m}),
$$
where the last equality follows from
\eqref{eq:block-resistance-times-volume}. Consequently,
\begin{equation}
\label{eq:block-pseudoinverse-representation}
 z_i-z_j
 =
 \frac{V_G}{p\bar s^2}
 (e_i-e_j)^{\mathsf T}
 L_{C_q^R}^+
 (\widehat\pi-\nu_y)
 +
 o(T_{n,m}).
\end{equation}

\medskip
\noindent\emph{The stationary measure and volume.}
We next show that $\widehat\pi$ is asymptotically uniform and that the coefficient in
\eqref{eq:block-pseudoinverse-representation} is asymptotic to $2qR$, the volume of
$C_q^R$. Let $\mathfrak u_q$ be uniform on $\Z_q$. Block balance gives
$$
 \sum_\ell\left|\frac{s_\ell}{n}-\frac1q\right|
 \leq\frac qn=\frac1{\bar s}.
$$
Write the block-volume estimate as
$$
 \vol_G(B_\ell)=\mu s_\ell(1+\xi_\ell),
 \qquad
 \max_\ell|\xi_\ell|\leq\kappa_n.
$$
Then
$V_G=\mu n(1+\bar\xi)$, where
$\bar\xi:=\sum_\ell(s_\ell/n)\xi_\ell$ and
$|\bar\xi|\leq\kappa_n$. Consequently,
$$
 \sum_\ell
 \left|\widehat\pi_\ell-\frac{s_\ell}{n}\right|
 \leq C\kappa_n.
$$
The triangle inequality, $\kappa_n\to0$, and $\bar s\to\infty$ now give
\begin{equation}
\label{eq:block-stationary-total-variation}
 \norm{\widehat\pi-\mathfrak u_q}_1=o(1).
\end{equation}
Furthermore,
\begin{equation}
\label{eq:block-volume-coefficient}
 \frac{V_G}{p\bar s^2}
 =
 (1+o(1))\,2qR.
\end{equation}
Indeed, summing \eqref{eq:block-volume-concentration-note} over the
blocks gives $V_G=(1+o(1))\,2mpn$, while
$n=q\bar s$ and $m=(1+o(1))R\bar s$. Therefore
$$
 \frac{V_G}{2pqR\bar s^2}
 =
 \frac{V_G}{2pmn}\,
 \frac{m}{R\bar s}
 =
 1+o(1).
$$
A neighbour of $y\in B_j$ in $B_k$ satisfies $(d_q(j,k)-1)s_-\leq m$.
As before, $s_+-s_-\leq1$, $m/(s_-s_+)=o(1)$, and
$R=\lfloor m/s_+\rfloor-1$ imply
\begin{equation}
\label{eq:neighbour-block-support}
 \operatorname{supp}(\nu_y)
 \subseteq
 \{\ell:d_q(\ell,j)\leq R+4\}.
\end{equation}

For the regular graph $C_q^R$,
\begin{equation}
\label{eq:cycle-power-hitting-pseudoinverse}
 H_{C_q^R}(i,j)
 =
 2qR(e_i-e_j)^{\mathsf T}
 L_{C_q^R}^+
 (\mathfrak u_q-e_j).
\end{equation}
This follows from the hitting-time equation with source
$2R\1-2qR e_j=2qR(\mathfrak u_q-e_j)$, after subtracting the potential at $j$.
We compare this expression with \eqref{eq:block-pseudoinverse-representation}.
The positive mass of $\widehat\pi-\nu_y$ is at most one; that of
$\widehat\pi-\mathfrak u_q$ is $\tfrac12\norm{\widehat\pi-\mathfrak u_q}_1=o(1)$.
Put
$$
 \mathcal R_*:=
 \max_{k,\ell}R_{\eff}^{C_q^R}(k,\ell)
 \leq
 C\left(\frac1R+\frac q{R^3}\right),
$$
where the inequality follows from
Corollary~\ref{cor:cycle-resistance}. For any signed measure $\rho$ of
total mass zero, Lemma~\ref{lem:resistance-duality-stability} gives
\begin{equation}
\label{eq:quotient-source-potential-bound}
 \left|
 (e_i-e_j)^{\mathsf T}L_{C_q^R}^+\rho
 \right|
 \leq
 \left(\sum_k(\rho_k)_+\right)\mathcal R_*.
\end{equation}
The coefficient error in \eqref{eq:block-pseudoinverse-representation} is $o(qR)$.
Apply \eqref{eq:quotient-source-potential-bound} to $\rho=\widehat\pi-\nu_y$ and
$\rho=\widehat\pi-\mathfrak u_q$. The errors from the coefficient and stationary
measure are at most
$$
 o(qR)\mathcal R_*+2qR\,o(1)\mathcal R_*
 =
 o\left(q+\frac{q^2}{R^2}\right)
 =
 o(T_{n,m}),
$$
uniformly in $i,j$, and $y$.

\medskip
\noindent\emph{Concentrating the target measure.}
By \eqref{eq:neighbour-block-support}, every $k$ with $\nu_{y,k}>0$ has
$d_q(k,j)\leq R+4$. The local resistance bound in Corollary~\ref{cor:cycle-resistance}
therefore gives
$$
 \max_{d_q(k,j)\leq R+4}
 R_{\eff}^{C_q^R}(k,j)
 \leq\frac CR.
$$
Since $\nu_y$ is a probability measure,
$$
\nu_y-e_j
=
\sum_{k\in\mathbb Z_q}\nu_{y,k}(e_k-e_j).
$$
The triangle inequality for the $L_{C_q^R}^+$-seminorm gives
$$
\begin{aligned}
\lVert\nu_y-e_j\rVert_{L_{C_q^R}^+}
&\leq
\sum_k\nu_{y,k}\lVert e_k-e_j\rVert_{L_{C_q^R}^+}\\
&=
\sum_k\nu_{y,k}\sqrt{R_{\eff}^{C_q^R}(k,j)}
\leq
\sqrt{\frac CR}.
\end{aligned}
$$
Consequently,
$$
(\nu_y-e_j)^{\mathsf T}L_{C_q^R}^+(\nu_y-e_j)
\leq\frac CR.
$$
Applying \eqref{eq:resistance-cauchy-schwarz} to $e_i-e_j$ and $\nu_y-e_j$, with the
preceding bound and the definition of $\mathcal R_*$, yields
\begin{align}
\label{eq:smeared-target-error}
 &2qR\left|
 (e_i-e_j)^{\mathsf T}L_{C_q^R}^+(\nu_y-e_j)
 \right|\notag\\
 &\qquad\leq
 2qR\sqrt{\frac{C\mathcal R_*}{R}}
 \leq
 C\sqrt{\left(q+\frac{q^2}{R^2}\right)q}
 =
 o(T_{n,m}).
\end{align}
For the last equality, note that
\begin{equation}
\label{eq:q-negligible-against-T}
 \frac q{T_{n,m}}
 \leq
 \begin{cases}
 CKm/n,&m\leq\sqrt n,\\
 CK/m,&m>\sqrt n,
 \end{cases}
 \longrightarrow0.
\end{equation}
Moreover,
$$
 \frac{q^2}{R^2}
 =
 (1+o(1))\frac{n^2}{m^2}
 =
 O(T_{n,m}).
$$
Thus $q+q^2/R^2=O(T_{n,m})$ and $q=o(T_{n,m})$, which justify the last equality in
\eqref{eq:smeared-target-error}. We have proved
\begin{equation}
\label{eq:block-to-cycle-power-hitting}
 z_i-z_j
 =
 H_{C_q^R}(i,j)+o(T_{n,m}).
\end{equation}

\medskip
\noindent\emph{Returning to vertex coordinates.}
It remains to replace block distance by the cyclic distance between $x$ and $y$.
If $i=j$, then $z_i-z_j=0$ and $r_{xy}\leq s_+$, so
$$
 \frac{r_{xy}(n-r_{xy})}{m^2}
 \leq
 C\frac{n}{mK}
 =
 o(T_{n,m});
$$
indeed, dividing $n/(mK)$ by $T_{n,m}$ gives at most $Cm/(nK)$ when
$m\leq\sqrt n$ and $C/(mK)$ when $m>\sqrt n$. This proves
\eqref{eq:block-source-identification-conclusion} for $i=j$.

Now suppose $i\neq j$. Since $R\to\infty$ and $R<q/2$, Lemma~\ref{lem:deterministic-hitting}
gives
\begin{equation}
\label{eq:quotient-quadratic-distance-asymptotic}
 H_{C_q^R}(i,j)
 =
 q+
 \frac{3d_q(i,j)(q-d_q(i,j))}{R^2}
 +
 o\left(q+\frac{q^2}{R^2}\right).
\end{equation}
The first term is negligible because $q=o(T_{n,m})$. Put $h:=d_q(i,j)\leq q/2$.
The two arcs between $B_i$ and $B_j$ contain $h$ and $q-h$ block intervals. An arc of
$h$ blocks beginning at $t_k$ has boundary-to-boundary length $t_{k+h}-t_k$, with
$t_{k+q}=t_k+n$, and
$$
\left|t_{k+h}-t_k-h\bar s\right|\leq2
$$
because $t_\ell=\lfloor \ell n/q\rfloor$. Replacing the boundary points by $x\in B_i$
and $y\in B_j$ changes the arc length by at most $2s_+$. The complementary arc differs
from $(q-h)\bar s$ by the same order. Taking the shorter arc and using $h\leq q/2$ gives
$$
 |r_{xy}-\bar s h|\leq C\bar s,
 \qquad
 n=q\bar s,
 \qquad
 m=(1+o(1))R\bar s.
$$
The first two relations imply
$$
 \left|
 r_{xy}(n-r_{xy})-\bar s^2h(q-h)
 \right|
 \leq Cn\bar s+C\bar s^2.
$$
Moreover,
$$
 \frac{n\bar s/m^2}{T_{n,m}}
 \leq
 \begin{cases}
 C m/(nK),&m\leq\sqrt n,\\
 C/(mK),&m>\sqrt n,
 \end{cases}
 \longrightarrow0.
$$
Also, $\bar s^2/m^2=O(K^{-2})=o(T_{n,m})$. Since $R\bar s/m\to1$ and
$h(q-h)/R^2=O(q^2/R^2)=O(T_{n,m})$, replacing $R^{-2}$ by $\bar s^2/m^2$ costs
only $o(T_{n,m})$. Therefore
\begin{equation}
\label{eq:quadratic-distance-coordinate-conversion}
 \frac{3d_q(i,j)(q-d_q(i,j))}{R^2}
 =
 \frac{3r_{xy}(n-r_{xy})}{m^2}
 +
 o(T_{n,m}).
\end{equation}
Equations~\eqref{eq:block-to-cycle-power-hitting}--\eqref{eq:quadratic-distance-coordinate-conversion}
prove \eqref{eq:block-source-identification-conclusion}.
\end{proof}

We can now combine the exact target-degree decomposition with the block reduction.
This gives the point-to-point expansion, whose maximum yields the final assertion.

\begin{proof}[Proof of Theorem~\ref{thm:sharp-hitting}]
Set $\Omega_{n,p}:=\mathcal B_{n,p}$. By Lemma~\ref{lem:uniform-block-event}, these events
have probability tending to one uniformly in $p$, and $G$ is connected on each event.

For distinct $x,y$, Lemma~\ref{lem:exact-local-removal} gives
$H_G(x,y)=V_G/d_G(y)+u_y(x)-u_y(y)$. Applying
Lemmas~\ref{lem:within-block-elimination} and~\ref{lem:block-source-identification} to
the potential difference, with their deterministic errors uniform over
$p\in[c\log n/(2m),1]$ and distinct $x,y$, gives a deterministic
$\widetilde\varepsilon_n\to0$, independent of $p$, such that on $\Omega_{n,p}$,
$$
\left|
H_G(x,y)
-\frac{V_G}{d_G(y)}
-\frac{3r_{xy}(n-r_{xy})}{m^2}
\right|
\leq\widetilde\varepsilon_nT_{n,m}.
$$
By \eqref{eq:step-variance-scale},
$$
\frac1{\kappa_m}=\frac3{m^2}+O(m^{-3}),
$$
and hence, uniformly in $x,y$,
$$
\left|
\frac{r_{xy}(n-r_{xy})}{\kappa_m}
-
\frac{3r_{xy}(n-r_{xy})}{m^2}
\right|
\leq C\frac{n^2}{m^3}
=o(T_{n,m}).
$$
Absorbing this error into the deterministic sequence gives
$$
\left|
H_G(x,y)
-\frac{V_G}{d_G(y)}
-\frac{r_{xy}(n-r_{xy})}{\kappa_m}
\right|
\leq\varepsilon_nT_{n,m},
\qquad \varepsilon_n\to0.
$$
On $\Omega_{n,p}$, define $R_{n,p}(x,y)$ as the expression inside the absolute value,
and set $R_{n,p}(x,y):=0$ on $\Omega_{n,p}^c$. This gives
\eqref{eq:pointwise-hitting-expansion}--\eqref{eq:pointwise-hitting-remainder}.

It remains to maximize the expansion. On $\mathcal B_{n,p}$,
\begin{equation}
\label{eq:degree-term-comparable-to-n}
n\leq \frac{V_G}{d_{\min}(G)}\leq C_cn,
\end{equation}
The lower bound uses $d_{\min}(G)\leq V_G/n$; the upper bound follows from
\eqref{eq:block-degree-bounds-note}. Since $r(n-r)\leq n^2/4$, the expansion gives
$$
H_{\max}(G)\leq \frac{V_G}{d_{\min}(G)}+\frac{n^2}{4\kappa_m}+o(T_{n,m}).
$$
For the matching lower bound, choose a minimum-degree target $y_\star$ and a starting
vertex $x_\star$ with $r_{x_\star y_\star}=\lfloor n/2\rfloor$. Then
$r_{x_\star y_\star}(n-r_{x_\star y_\star})=n^2/4+O(1)$, and the same expansion yields
$$
H_{\max}(G)=\frac{V_G}{d_{\min}(G)}+\frac{n^2}{4\kappa_m}+o(T_{n,m}).
$$
By \eqref{eq:degree-term-comparable-to-n}, the leading expression is comparable to
$T_{n,m}$, so the additive error is also a relative $o(1)$ error. The event probabilities
come from Lemma~\ref{lem:uniform-block-event}, completing the proof.
\end{proof}

\begin{proof}[Proof of Corollary~\ref{cor:explicit-maximal-hitting}]
By Lemma~\ref{lem:degrees-fixed-c}, uniformly over the permitted retention probabilities,
$$
V_G=\bigl(1+o_{\Prob}(1)\bigr)n(2mp),
\qquad
d_{\min}(G)
=
\bigl(\delta_{\min}(\theta_n)+o_{\Prob}(1)\bigr)(2mp).
$$
Since $\delta_{\min}(\theta_n)\geq\delta_{\min}(c)>0$, substituting these estimates into
\eqref{eq:sharp-hitting-conclusion} proves \eqref{eq:explicit-maximal-hitting}.
Continuity of the minimum-degree profile on $(1,\infty]$ gives
\eqref{eq:explicit-maximal-hitting-limit} when $\theta_n\to\theta\in[c,\infty]$.
For $p=1$, we have $G=C_n^m$ and $\theta_n=2m/\log n\to\infty$, so
$\delta_{\min}(\theta_n)\to1$. This proves \eqref{eq:host-maximal-hitting}.
\end{proof}

\section{Open problems}
\label{sec:conclusion}

\paragraph{The small-bandwidth connectivity transition.}
When $m/\log n\to\infty$, isolated vertices are the only asymptotically
relevant obstruction to connectivity in the critical window. What
happens when $m=O(\log n)$? In this range, deleting the boundary edges
of a cyclic interval may create a nontrivial component. Determine the
critical parametrization of $p$, the joint limiting laws of isolated
vertices and nontrivial interval components, and the resulting
connectivity probability. Related one-dimensional phenomena appear in
bounded-range and soft random geometric graphs \cite{WilsherDettmannGanesh2020,WilsherDettmannGanesh2023,Federico2024BoundedRange}.

\paragraph{Distance-dependent retention.}
Suppose edges are retained independently, with probability $p_r$ for
an edge of cyclic length $r$. The isolation intensity is
$$
n\prod_{r=1}^m(1-p_r)^2,
$$
whereas transport should depend on
$$
s_0:=\sum_{r=1}^m p_r,
\qquad
\kappa:=\frac{\sum_{r=1}^m p_rr^2}{s_0}.
$$
Assume that $\{r:p_r>0\}$ generates $\mathbb Z_n$. What additional regularity and nondegeneracy assumptions ensure a Poisson connectivity window, with every non-isolated vertex in the largest component? Under suitable assumptions, one also expects, when $\sqrt\kappa=o(n)$,
$$
\gamma_{\lazy}(G)=\left(\pi^2+o_{\Prob}(1)\right)\frac{\kappa}{n^2},
\qquad
H_G(x,y)=\frac{V_G}{d_G(y)}+\frac{r_{xy}(n-r_{xy})}{\kappa}
+o_{\Prob}\left(n+\frac{n^2}{\kappa}\right),
$$
uniformly in $x\ne y$. Related results on connectivity and random walks on spatial graphs have been developed 
\cite{Penrose2016Soft,Penrose2018Poisson,Iyer2018RandomConnection,MaoAnderson2011RandomConnection,BenjaminiBergerYadin2008}.

\paragraph{Higher-dimensional tori.}
Take the vertex set $(\mathbb Z/n\mathbb Z)^d$, with admissible
displacements in a symmetric finite-range set
$\mathcal N_m\subset\mathbb Z^d\setminus\{0\}$. Writing
$D_{d,m}=|\mathcal N_m|$, the natural connectivity window is
$n^d(1-p)^{D_{d,m}}\to e^{-t}$. Under what geometric assumptions on $\mathcal N_m$ does this window have a Poisson limit for isolated vertices, with every other vertex in the largest component?

Above connectivity, the covariance matrix of the uniform measure on $\mathcal N_m$ should determine the sharp spectral constant. The quadratic hitting-time term should become a difference of torus Green functions. Existing results on geometric graphs and supercritical
percolation clusters provide qualitative comparisons
\cite{MeesterRoy1996Continuum,Penrose2003RGG,BenjaminiMossel2003,Barlow2004,MathieuPiatnitski2007}; the challenge is to obtain finite-volume constants and uniform point-to-point formulas down to a fixed factor above the connectivity threshold.

\paragraph{Quantum dynamics.}
Quantum walks, the quantum analogue of classical random walks, are a universal model of quantum computation, and exploring the fundamental differences between quantum and classical random walks on different graphs is of widespread interest. For instance, let $M\subsetneq\mathbb Z_n$ be a nonempty marked set of vertices on graph $G$, and let $\mathrm{HT}(G,M)$ be the classical mean hitting time when the
initial vertex is sampled from $\pi_G$ conditioned on being unmarked. Given a coherent encoding of $\pi_G$, quantum-walk search algorithms find a marked vertex with constant probability in time
$$
\mathrm{QHT}(G,M)=\widetilde O(\sqrt{\mathrm{HT}(G,M)}),
$$ 
known as the quantum hitting time \cite{AmbainisGilyenJefferyKokainis2020, ApersChakrabortyNovoRoland2022}. Our results, namely Theorem~\ref{thm:spectral-gap} and Theorem~\ref{thm:sharp-hitting}, therefore prove that the quantum hitting time on band random graphs is
$$
\mathrm{QHT}(G(n,p,m),M)=\widetilde O\left(
\min\left\{
\sqrt n+\frac nm,\,
\frac{n^{3/2}}{m\sqrt{|M|}},
\right\}
\right)
$$
with high probability whenever $\log n\ll m<n/2$ and $p\geq c\log n/(2m)$, for fixed $c>1$. The quantum hitting time for Erd\H{o}s-R\'enyi graphs is known to scale as $\sqrt{n}$, whenever $p$ is above the threshold of connectivity \cite{ChakrabortyNovoAmbainisOmar2016}.

On the other hand, the time-averaged mixing time of quantum walks on band random graphs requires further information about spectral spacings throughout the spectrum. Can one obtain such estimates in this model, as has been done for Erd\H{o}s--R\'enyi graphs \cite{ChakrabortyLuhRoland2020Mixing}?

\section*{Acknowledgments}
SC and RPS acknowledge support from IIIT Hyderabad. SC also acknowledges funding from the National Quantum Mission, an initiative of the Department of Science and Technology, Govt. of India. SC also acknowledges funding from Fujitsu Ltd., Japan. KL was supported in part by Simons Grant MP-TSM-00001988 and DOE Office of Science Advanced Scientific Computing Research (ASCR) program, funding opportunity DE-FOA-0003432.  The authors would like to thank Vishesh Jain for helpful discussions in the early stages of this project.

\paragraph{Generative AI disclosure.}
The authors conceived of the project three years ago and developed all of the mathematical proofs over the years, without using any AI. Large language models were only used to assist with checking the proofs and editing the exposition. The authors take responsibility for the final manuscript.

\bibliographystyle{unsrt}
\bibliography{refgeom2}

@book{MeesterRoy1996Continuum,
  author    = {Meester, Ronald and Roy, Rahul},
  title     = {Continuum Percolation},
  series    = {Cambridge Tracts in Mathematics},
  volume    = {119},
  publisher = {Cambridge University Press},
  address   = {Cambridge},
  year      = {1996},
  doi       = {10.1017/CBO9780511895357},
  isbn      = {9780521475044}
}

@inproceedings{Bourgade2018RandomBandMatrices,
  author    = {Bourgade, Paul},
  title     = {Random Band Matrices},
  booktitle = {Proceedings of the International Congress of Mathematicians---Rio de Janeiro 2018},
  volume    = {4},
  pages     = {2759--2783},
  publisher = {World Scientific},
  year      = {2018},
  doi       = {10.1142/9789813272880_0159}
}

@book{BollobasRandomGraphs,
  author    = {B{\'e}la Bollob{\'a}s},
  title     = {Random Graphs},
  edition   = {2},
  series    = {Cambridge Studies in Advanced Mathematics},
  volume    = {73},
  publisher = {Cambridge University Press},
  address   = {Cambridge},
  year      = {2001},
  isbn      = {9780521809207}
}

@book{JLR2000,
  author    = {Svante Janson and Tomasz {\L}uczak and Andrzej Ruci{\'n}ski},
  title     = {Random Graphs},
  publisher = {Wiley},
  address   = {New York},
  year      = {2000},
  series    = {Wiley-Interscience Series in Discrete Mathematics and Optimization}
}

@book{BarbourHolstJanson1992,
  author    = {Barbour, Andrew D. and Holst, Lars and Janson, Svante},
  title     = {Poisson Approximation},
  series    = {Oxford Studies in Probability},
  volume    = {2},
  publisher = {Oxford University Press},
  address   = {Oxford},
  year      = {1992},
  doi       = {10.1093/oso/9780198522355.001.0001}
}

@book{hofstad2017randomgraphs,
  author    = {Remco van der Hofstad},
  title     = {Random Graphs and Complex Networks},
  volume    = {1},
  series    = {Cambridge Series in Statistical and Probabilistic Mathematics},
  publisher = {Cambridge University Press},
  address   = {Cambridge},
  year      = {2017},
  doi       = {10.1017/9781316779422}
}

@book{levin2017markov,
  author    = {Levin, David A. and Peres, Yuval},
  title     = {Markov Chains and Mixing Times},
  edition   = {Second},
  publisher = {American Mathematical Society},
  address   = {Providence, RI},
  year      = {2017}
}

@book{vershynin2018high,
  author    = {Vershynin, Roman},
  title     = {High-Dimensional Probability: An Introduction with Applications in Data Science},
  series    = {Cambridge Series in Statistical and Probabilistic Mathematics},
  volume    = {47},
  publisher = {Cambridge University Press},
  address   = {Cambridge},
  year      = {2018},
  doi       = {10.1017/9781108231596}
}

@inproceedings{ChandraRaghavanRuzzoSmolenskyTiwari1989,
  author    = {Chandra, Ashok K. and Raghavan, Prabhakar and Ruzzo, Walter L. and Smolensky, Roman and Tiwari, Prasoon},
  title     = {The Electrical Resistance of a Graph Captures Its Commute and Cover Times},
  booktitle = {Proceedings of the Twenty-First Annual ACM Symposium on Theory of Computing},
  pages     = {574--586},
  publisher = {ACM},
  year      = {1989},
  doi       = {10.1145/73007.73062}
}

@article{Tetali1991,
  author  = {Tetali, Prasad},
  title   = {Random Walks and the Effective Resistance of Networks},
  journal = {Journal of Theoretical Probability},
  volume  = {4},
  number  = {1},
  pages   = {101--109},
  year    = {1991},
  doi     = {10.1007/BF01046996}
}

@article{VonLuxburgRadlHein2014,
  author  = {von Luxburg, Ulrike and Radl, Agnes and Hein, Matthias},
  title   = {Hitting and Commute Times in Large Random Neighborhood Graphs},
  journal = {Journal of Machine Learning Research},
  volume  = {15},
  number  = {52},
  pages   = {1751--1798},
  year    = {2014},
  url     = {https://jmlr.org/papers/v15/vonluxburg14a.html}
}

@article{Sylvester2021,
  author  = {Sylvester, John},
  title   = {Random Walk Hitting Times and Effective Resistance in Sparsely Connected Erd{\H{o}}s--R{\'e}nyi Random Graphs},
  journal = {Journal of Graph Theory},
  volume  = {96},
  number  = {1},
  pages   = {44--84},
  year    = {2021},
  doi     = {10.1002/jgt.22551}
}

@book{Penrose2003RGG,
  author    = {Penrose, Mathew D.},
  title     = {Random Geometric Graphs},
  series    = {Oxford Studies in Probability},
  volume    = {5},
  publisher = {Oxford University Press},
  address   = {Oxford},
  year      = {2003}
}

@article{Penrose2016Soft,
  author  = {Penrose, Mathew D.},
  title   = {Connectivity of Soft Random Geometric Graphs},
  journal = {The Annals of Applied Probability},
  volume  = {26},
  number  = {2},
  pages   = {986--1028},
  year    = {2016},
  doi     = {10.1214/15-AAP1110}
}

@article{Penrose2018Poisson,
  author  = {Penrose, Mathew D.},
  title   = {Inhomogeneous Random Graphs, Isolated Vertices, and {Poisson} Approximation},
  journal = {Journal of Applied Probability},
  volume  = {55},
  number  = {1},
  pages   = {112--136},
  year    = {2018},
  doi     = {10.1017/jpr.2018.9}
}

@article{PenroseYang2025,
  author  = {Penrose, Mathew D. and Yang, Xiaochuan},
  title   = {Fluctuations of the Connectivity Threshold and Largest Nearest-Neighbour Link},
  journal = {The Annals of Applied Probability},
  volume  = {35},
  number  = {6},
  pages   = {3906--3941},
  year    = {2025},
  doi     = {10.1214/25-AAP2210}
}

@article{WilsherDettmannGanesh2020,
  author  = {Wilsher, Michael and Dettmann, Carl P. and Ganesh, Ayalvadi J.},
  title   = {Connectivity in One-Dimensional Soft Random Geometric Graphs},
  journal = {Physical Review E},
  volume  = {102},
  pages   = {062312},
  year    = {2020},
  doi     = {10.1103/PhysRevE.102.062312}
}

@article{WilsherDettmannGanesh2023,
  author  = {Wilsher, Michael and Dettmann, Carl P. and Ganesh, Ayalvadi J.},
  title   = {The Distribution of the Number of Isolated Nodes in the One-Dimensional Soft Random Geometric Graph},
  journal = {Statistics \& Probability Letters},
  volume  = {193},
  pages   = {109695},
  year    = {2023},
  doi     = {10.1016/j.spl.2022.109695}
}

@inproceedings{GuptaLeeLi2021,
  author    = {Gupta, Anupam and Lee, Euiwoong and Li, Jason},
  title     = {The Connectivity Threshold for Dense Graphs},
  booktitle = {Proceedings of the 2021 ACM--SIAM Symposium on Discrete Algorithms},
  pages     = {89--105},
  publisher = {SIAM},
  year      = {2021},
  doi       = {10.1137/1.9781611976465.7}
}

@article{FedericoHofstadHulshof2016,
  author  = {Federico, Lorenzo and van der Hofstad, Remco and Hulshof, Tim},
  title   = {Connectivity Threshold for Random Subgraphs of the Hamming Graph},
  journal = {Electronic Communications in Probability},
  volume  = {21},
  pages   = {1--8},
  year    = {2016},
  note    = {Paper No. 27},
  doi     = {10.1214/16-ECP4479}
}

@inproceedings{BoydGhoshPrabhakarShah2005,
  author    = {Boyd, Stephen P. and Ghosh, Arpita and Prabhakar, Balaji and Shah, Devavrat},
  title     = {Mixing Times for Random Walks on Geometric Random Graphs},
  booktitle = {Proceedings of the Seventh Workshop on Algorithm Engineering and Experiments and the Second Workshop on Analytic Algorithmics and Combinatorics},
  pages     = {240--249},
  publisher = {SIAM},
  year      = {2005}
}

@article{AvinErcal2007,
  author  = {Avin, Chen and Ercal, Gunes},
  title   = {On the Cover Time and Mixing Time of Random Geometric Graphs},
  journal = {Theoretical Computer Science},
  volume  = {380},
  number  = {1--2},
  pages   = {2--22},
  year    = {2007},
  doi     = {10.1016/j.tcs.2007.02.065}
}

@article{BenjaminiBergerYadin2008,
  author  = {Benjamini, Itai and Berger, Noam and Yadin, Ariel},
  title   = {Long-Range Percolation Mixing Time},
  journal = {Combinatorics, Probability and Computing},
  volume  = {17},
  number  = {4},
  pages   = {487--494},
  year    = {2008},
  doi     = {10.1017/S0963548308008948}
}

@article{MullerPenrose2020,
  author  = {M{\"u}ller, Tobias and Penrose, Mathew D.},
  title   = {Optimal Cheeger Cuts and Bisections of Random Geometric Graphs},
  journal = {The Annals of Applied Probability},
  volume  = {30},
  number  = {3},
  pages   = {1458--1483},
  year    = {2020},
  doi     = {10.1214/19-AAP1534}
}

@article{ChungHorn2007,
  author  = {Chung, Fan R. K. and Horn, Paul},
  title   = {The Spectral Gap of a Random Subgraph of a Graph},
  journal = {Internet Mathematics},
  volume  = {4},
  number  = {2--3},
  pages   = {225--244},
  year    = {2007},
  doi     = {10.1080/15427951.2007.10129296}
}

@article{BandeiraVanHandel2016,
  author  = {Bandeira, Afonso S. and van Handel, Ramon},
  title   = {Sharp Nonasymptotic Bounds on the Norm of Random Matrices with Independent Entries},
  journal = {The Annals of Probability},
  volume  = {44},
  number  = {4},
  pages   = {2479--2506},
  year    = {2016},
  doi     = {10.1214/15-AOP1025}
}

@article{ErdosKnowlesYauYin2013Delocalization,
  author  = {Erd{\H{o}}s, L{\'a}szl{\'o} and Knowles, Antti and Yau, Horng-Tzer and Yin, Jun},
  title   = {Delocalization and Diffusion Profile for Random Band Matrices},
  journal = {Communications in Mathematical Physics},
  volume  = {323},
  number  = {1},
  pages   = {367--416},
  year    = {2013},
  doi     = {10.1007/s00220-013-1773-3}
}

@article{Federico2024BoundedRange,
  author  = {Federico, Lorenzo},
  title   = {Connecting Probability for Random Bounded-Range One-Dimensional Network},
  journal = {Annali dell'Universit{\`a} di Ferrara},
  volume  = {70},
  pages   = {1389--1403},
  year    = {2024},
  doi     = {10.1007/s11565-024-00508-6}
}

@article{Gilbert1959,
  author  = {Gilbert, Edgar N.},
  title   = {Random Graphs},
  journal = {The Annals of Mathematical Statistics},
  volume  = {30},
  number  = {4},
  pages   = {1141--1144},
  year    = {1959},
  doi     = {10.1214/aoms/1177706098}
}

@article{ErdosRenyi1959,
  author  = {Erd\H{o}s, Paul and R\'enyi, Alfr\'ed},
  title   = {On Random Graphs. {I}},
  journal = {Publicationes Mathematicae Debrecen},
  volume  = {6},
  pages   = {290--297},
  year    = {1959}
}

@article{ErdosRenyi1960,
  author  = {Erd\H{o}s, Paul and R\'enyi, Alfr\'ed},
  title   = {On the Evolution of Random Graphs},
  journal = {Publications of the Mathematical Institute of the Hungarian Academy of Sciences},
  volume  = {5},
  pages   = {17--61},
  year    = {1960}
}

@article{Gilbert1961,
  author  = {Gilbert, Edgar N.},
  title   = {Random Plane Networks},
  journal = {Journal of the Society for Industrial and Applied Mathematics},
  volume  = {9},
  number  = {4},
  pages   = {533--543},
  year    = {1961},
  doi     = {10.1137/0109045}
}

@article{Penrose1999KConnectivity,
  author  = {Penrose, Mathew D.},
  title   = {On $k$-Connectivity for a Geometric Random Graph},
  journal = {Random Structures \& Algorithms},
  volume  = {15},
  number  = {2},
  pages   = {145--164},
  year    = {1999}
}

@article{GoelRaiKrishnamachari2005,
  author  = {Goel, Ashish and Rai, Sanatan and Krishnamachari, Bhaskar},
  title   = {Monotone Properties of Random Geometric Graphs Have Sharp Thresholds},
  journal = {The Annals of Applied Probability},
  volume  = {15},
  number  = {4},
  pages   = {2535--2552},
  year    = {2005},
  doi     = {10.1214/105051605000000575}
}

@article{AlonBenjaminiStacey2004,
  author  = {Alon, Noga and Benjamini, Itai and Stacey, Alan},
  title   = {Percolation on Finite Graphs and Isoperimetric Inequalities},
  journal = {The Annals of Probability},
  volume  = {32},
  number  = {3A},
  pages   = {1727--1745},
  year    = {2004},
  doi     = {10.1214/009117904000000414}
}

@article{BorgsChayesHofstadSladeSpencer2005,
  author  = {Borgs, Christian and Chayes, Jennifer T. and van der Hofstad, Remco and Slade, Gordon and Spencer, Joel},
  title   = {Random Subgraphs of Finite Graphs: {I}. The Scaling Window under the Triangle Condition},
  journal = {Random Structures \& Algorithms},
  volume  = {27},
  number  = {2},
  pages   = {137--184},
  year    = {2005},
  doi     = {10.1002/rsa.20051}
}

@article{Joos2015,
  author  = {Joos, Felix},
  title   = {Random Subgraphs in Sparse Graphs},
  journal = {SIAM Journal on Discrete Mathematics},
  volume  = {29},
  number  = {4},
  pages   = {2350--2360},
  year    = {2015},
  doi     = {10.1137/140976340}
}

@article{Barlow2004,
  author  = {Barlow, Martin T.},
  title   = {Random Walks on Supercritical Percolation Clusters},
  journal = {The Annals of Probability},
  volume  = {32},
  number  = {4},
  pages   = {3024--3084},
  year    = {2004},
  doi     = {10.1214/009117904000000748}
}

@article{BenjaminiMossel2003,
  author  = {Benjamini, Itai and Mossel, Elchanan},
  title   = {On the Mixing Time of a Simple Random Walk on the Super Critical Percolation Cluster},
  journal = {Probability Theory and Related Fields},
  volume  = {125},
  number  = {3},
  pages   = {408--420},
  year    = {2003},
  doi     = {10.1007/s00440-002-0246-y}
}

@article{MathieuPiatnitski2007,
  author  = {Mathieu, Pierre and Piatnitski, Andrey L.},
  title   = {Quenched Invariance Principles for Random Walks on Percolation Clusters},
  journal = {Proceedings of the Royal Society A: Mathematical, Physical and Engineering Sciences},
  volume  = {463},
  number  = {2085},
  pages   = {2287--2307},
  year    = {2007},
  doi     = {10.1098/rspa.2007.1876}
}

@book{LawlerLimic2010,
  author    = {Lawler, Gregory F. and Limic, Vlada},
  title     = {Random Walk: A Modern Introduction},
  series    = {Cambridge Studies in Advanced Mathematics},
  volume    = {123},
  publisher = {Cambridge University Press},
  address   = {Cambridge},
  year      = {2010},
  doi       = {10.1017/CBO9780511750854}
}

@book{LyonsPeres2016,
  author    = {Lyons, Russell and Peres, Yuval},
  title     = {Probability on Trees and Networks},
  series    = {Cambridge Series in Statistical and Probabilistic Mathematics},
  volume    = {42},
  publisher = {Cambridge University Press},
  address   = {Cambridge},
  year      = {2016},
  doi       = {10.1017/9781316672815}
}

@article{DiskinGeisler2025,
  author  = {Diskin, Sahar and Geisler, Anna},
  title   = {Minimum Degree $k$ and $k$-Connectedness Usually Arrive Together},
  journal = {Discrete Mathematics},
  volume  = {348},
  number  = {6},
  pages   = {114453},
  year    = {2025},
  doi     = {10.1016/j.disc.2025.114453}
}

@article{DiskinErdeKangKrivelevich2025,
  author  = {Diskin, Sahar and Erde, Joshua and Kang, Mihyun and Krivelevich, Michael},
  title   = {Percolation on High-Dimensional Product Graphs},
  journal = {Random Structures \& Algorithms},
  volume  = {66},
  number  = {1},
  pages   = {e21268},
  year    = {2025},
  doi     = {10.1002/rsa.21268}
}

@article{ErdosRiabov2025,
  author        = {Erd{\H{o}}s, L{\'a}szl{\'o} and Riabov, Volodymyr},
  title         = {The Zigzag Strategy for Random Band Matrices},
  journal       = {arXiv preprint arXiv:2506.06441},
  year          = {2025},
  eprint        = {2506.06441},
  archivePrefix = {arXiv},
  primaryClass  = {math.PR}
}

@article{Drogin2025,
  author        = {Drogin, Reuben},
  title         = {Localization of One-Dimensional Random Band Matrices},
  journal       = {arXiv preprint arXiv:2508.05802},
  year          = {2025},
  eprint        = {2508.05802},
  archivePrefix = {arXiv},
  primaryClass  = {math.PR}
}

@article{YauYin2025Delocalization,
  author        = {Yau, Horng-Tzer and Yin, Jun},
  title         = {Delocalization of One-Dimensional Random Band Matrices},
  journal       = {arXiv preprint arXiv:2501.01718},
  year          = {2025},
  eprint        = {2501.01718},
  archivePrefix = {arXiv},
  primaryClass  = {math.PR}
}

@misc{AldousFill2002,
  author = {David Aldous and James Allen Fill},
  title  = {Reversible {M}arkov Chains and Random Walks on Graphs},
  year   = {2002},
  note   = {Unfinished monograph, recompiled 2014,
            available at \url{https://www.stat.berkeley.edu/~aldous/RWG/book.html}}
}

@article{Iyer2018RandomConnection,
  author  = {Iyer, Srikanth K.},
  title   = {The Random Connection Model: Connectivity, Edge Lengths, and Degree Distributions},
  journal = {Random Structures \& Algorithms},
  volume  = {52},
  number  = {2},
  pages   = {283--300},
  year    = {2018},
  doi     = {10.1002/rsa.20741}
}

@inproceedings{MaoAnderson2011RandomConnection,
  author    = {Mao, Guoqiang and Anderson, Brian D. O.},
  title     = {On the Asymptotic Connectivity of Random Networks under the Random Connection Model},
  booktitle = {2011 Proceedings IEEE INFOCOM},
  pages     = {631--639},
  publisher = {IEEE},
  year      = {2011},
  doi       = {10.1109/INFCOM.2011.5935242}
}

@inproceedings{AmbainisGilyenJefferyKokainis2020,
  author    = {Ambainis, Andris and Gily{\'e}n, Andr{\'a}s
               and Jeffery, Stacey and Kokainis, Martins},
  title     = {{Quadratic Speedup for Finding Marked Vertices
                by Quantum Walks}},
  booktitle = {Proceedings of the 52nd Annual ACM SIGACT
               Symposium on Theory of Computing},
  series    = {STOC 2020},
  pages     = {412--424},
  publisher = {Association for Computing Machinery},
  year      = {2020},
  doi       = {10.1145/3357713.3384252}
}

@article{ApersChakrabortyNovoRoland2022,
  author  = {Apers, Simon and Chakraborty, Shantanav
             and Novo, Leonardo and Roland, J{\'e}r{\'e}mie},
  title   = {{Quadratic Speedup for Spatial Search
              by Continuous-Time Quantum Walk}},
  journal = {Physical Review Letters},
  volume  = {129},
  number  = {16},
  pages   = {160502},
  year    = {2022},
  doi     = {10.1103/PhysRevLett.129.160502}
}

@article{ChakrabortyNovoAmbainisOmar2016,
  author  = {Chakraborty, Shantanav and Novo, Leonardo
             and Ambainis, Andris and Omar, Yasser},
  title   = {{Spatial Search by Quantum Walk is Optimal
              for Almost all Graphs}},
  journal = {Physical Review Letters},
  volume  = {116},
  number  = {10},
  pages   = {100501},
  year    = {2016},
  doi     = {10.1103/PhysRevLett.116.100501}
}

@article{ChakrabortyLuhRoland2020Mixing,
  author  = {Chakraborty, Shantanav and Luh, Kyle
             and Roland, J{\'e}r{\'e}mie},
  title   = {{How Fast Do Quantum Walks Mix?}},
  journal = {Physical Review Letters},
  volume  = {124},
  number  = {5},
  pages   = {050501},
  year    = {2020},
  doi     = {10.1103/PhysRevLett.124.050501}
}

\end{document}